\documentclass[11pt]{amsart}

\usepackage{epigamath}

\usepackage[english]{babel}

\numberwithin{equation}{section}

\usepackage[shortlabels]{enumitem}

\usepackage{amsmath,amsfonts,amssymb,mathtools,extarrows}

\usepackage{nameref,zref-xr}     

\usepackage{scalerel}         
\usepackage{stmaryrd}

\usepackage[all]{xy}

\usepackage{xcolor}

\newtheorem{theorem}{Theorem}[section]
\newtheorem{proposition}[theorem]{Proposition}
\newtheorem{lemma}[theorem]{Lemma}

\newtheorem{conjecture}{Conjecture}

\theoremstyle{remark}
\newtheorem{remark}[theorem]{Remark}
\newtheorem{example}[theorem]{Example}

\theoremstyle{definition}
\newtheorem{definition}[theorem]{Definition}

\newcommand{\mbP}{\mathbb P}
\newcommand{\mbZ}{\mathbb Z}
\newcommand{\mbC}{\mathbb C}

\newcommand{\cP}{\mathcal P}

\newcommand{\oM}{\overline{\mathcal M}}

\newcommand{\tu}{{\widetilde u}}
\newcommand{\og}{\overline g}

\def\cM{{\mathcal{M}}}
\def\oM{{\overline{\mathcal{M}}}}

\def\mbQ{{\mathbb Q}}
\def\d{{\partial}}

\newcommand{\eps}{\varepsilon}

\newcommand{\str}{\mathrm{str}}

\newcommand{\cA}{\mathcal A}
\newcommand{\hcA}{\widehat{\mathcal A}}
\newcommand{\DR}{\mathrm{DR}}

\newcommand{\even}{\mathrm{even}}
\newcommand{\ct}{\mathrm{ct}}

\renewcommand{\th}{\widetilde h}

\newcommand{\Coef}{\mathrm{Coef}}

\newcommand{\Desc}{\mathrm{Desc}}
\newcommand{\SRT}{\mathrm{SRT}}
\newcommand{\RT}{\mathrm{RT}}
\newcommand{\ST}{\mathrm{ST}}
\newcommand{\tv}{\widetilde v}
\renewcommand{\top}{\mathrm{top}}
\newcommand{\red}{\mathrm{red}}

\newcommand{\ZZ}{\mathbb{Z}}
\newcommand{\QQ}{\mathbb{Q}}

\DeclareMathOperator{\Aut}{Aut}

\newcommand{\tOmega}{\widetilde\Omega}

\newcommand{\gl}{\mathrm{gl}}

\newcommand{\un}{{1\!\! 1}}
\newcommand{\mcF}{\mathcal{F}}

\newcommand{\oQ}{{\overline{Q}}}
\newcommand{\oP}{{\overline{P}}}

\newcommand{\pt}{\mathrm{pt}}

\newcommand{\ev}{\mathrm{ev}}

\newcommand{\ee}{\mathrm{e}}
\newcommand{\tH}{\widetilde{H}}
\newcommand{\od}{\overline{d}}

\newcommand{\obeta}{\overline{\beta}}
\newcommand{\ob}{\overline{b}}
\newcommand{\perm}{\mathrm{perm}}

\def\br#1{\llbracket #1 \rrbracket}
\newcommand{\ou}{\overline{u}}
\newcommand{\tB}{\widetilde{B}}
\newcommand{\Core}{\mathrm{Core}}
\newcommand{\tT}{\widetilde{T}}
\newcommand{\tSRT}{\widetilde{\SRT}}
\newcommand{\vir}{{\mathrm{vir}}}

\newcommand{\oS}{\overline{S}}

\newcommand{\oF}{\overline{\mathcal{F}}}
\newcommand{\unst}{\mathrm{unst}}
\newcommand{\st}{\mathrm{st}}
\newcommand{\vdim}{\mathrm{vdim}}
\newcommand{\tpsi}{\widetilde{\psi}}
\newcommand{\mcO}{\mathcal{O}}
\newcommand{\tGamma}{\widetilde{\Gamma}}
\newcommand{\mbE}{\mathbb{E}}
\newcommand{\tw}{\widetilde{w}}
\newcommand{\tc}{\widetilde{c}}
\newcommand{\wk}{{\mathrm{wk}}}
\newcommand{\veq}{\mathrel{\rotatebox{90}{$=$}}}
\newcommand{\lvl}{\mathrm{lvl}}
\newcommand{\DRT}{\mathrm{DRT}}
\newcommand{\mcL}{\mathcal{L}}
\newcommand{\ind}{\mathrm{ind}}
\newcommand{\oc}{\overline{c}}
\newcommand{\ozero}{\overline{0}}

\newcommand{\ow}{\overline{w}}
\newcommand{\tgamma}{\widetilde{\gamma}}
\newcommand{\tf}{\widetilde{f}}
\newcommand{\te}{\widetilde{e}}
\newcommand{\td}{\widetilde{d}}
\newcommand{\tI}{\widetilde{I}}
\newcommand{\cS}{\mathcal{S}}

\newcommand{\lra}{\longrightarrow}

\setlist[enumerate,1]{label={\rm(\arabic*)}, ref={\rm\arabic*}}

\newcommand{\supth}[1]{\ensuremath{#1^{\mathrm{th}}}}

\EpigaVolumeYear{8}{2024} \EpigaArticleNr{12} \ReceivedOn{November 29, 2022}
\InFinalFormOn{January 16, 2024}
\AcceptedOn{February 27, 2024}

\title{Tautological relations and integrable systems}
\titlemark{Tautological relations and integrable systems}

\author{Alexandr Buryak}
\address{Faculty of Mathematics, National Research University Higher School of Economics, 6 Usacheva str., Moscow, 119048, Russian Federation\\
Igor Krichever Center for Advanced Studies, Skolkovo Institute of Science and Technology, Bolshoy Boulevard 30, bld. 1, Moscow, 121205, Russian Federation\\
P.~G.~Demidov Yaroslavl State University, 14 Sovetskaya str., Yaroslavl, 150003, Russian Federation}
\email{aburyak@hse.ru}

\author{Sergey Shadrin}
\address{Korteweg--de Vries Instituut voor Wiskunde, Universiteit van Amsterdam, Postbus 94248, 1090 GE Amsterdam, Netherlands}
\email{s.shadrin@uva.nl}

\authormark{A.~Buryak and S.~Shadrin}

\AbstractInEnglish{We present a family of conjectural relations in the tautological cohomology of the moduli spaces of stable algebraic curves of genus $g$ with $n$ marked points. A large part of these relations has a surprisingly simple form: the tautological classes involved in the relations are given by stable graphs that are trees and that are decorated only by powers of the psi-classes at half-edges. We show that the proposed conjectural relations imply certain fundamental properties of the Dubrovin--Zhang (DZ) and the double ramification (DR) hierarchies associated to F-cohomological field theories. Our relations naturally extend a similar system of conjectural relations, which were proposed in an earlier work of the first author together with Gu\'er\'e and Rossi and which are responsible for the normal Miura equivalence of the DZ and the DR hierarchy associated to an arbitrary cohomological field theory. Finally, we prove all of the above-mentioned relations in the case $n=1$ and arbitrary $g$ using a variation of the method from a paper by Liu and Pandharipande; this can be of independent interest. In particular, this proves the main conjecture from our previous joint work together with Hern\'andez Iglesias. We also prove all of the above-mentioned relations in the case $g=0$ and arbitrary~$n$.}



\acknowledgement{The work on Sections 1, 2, 4, 5 was supported by the Russian Science Foundation (grant no. 20-71-10110), \href{https://rscf.ru/en/project/23-71-50012/}{https://rscf.ru/en/project/23-71-50012/}, which funds the work of A.~B. at the P.~G.~Demidov Yaroslavl State University. S.\,S. was supported by the Dutch Research Council.}

\begin{document}

%%%%%%%%%%%%%%%%%%%%%%%%%%%%%%%
% Title page
%%%%%%%%%%%%%%%%%%%%%%%%%%%%%%%

\removeabove{0.32cm}
\removebetween{0.32cm}
\removebelow{0.32cm}

\maketitle

\begin{prelims}

\DisplayAbstractInEnglish

\bigskip

\DisplayKeyWords

\medskip

\DisplayMSCclass

\end{prelims}

%%%%%%%%%%%%%%%%%%%%%
% Table of Contents
%%%%%%%%%%%%%%%%%%%%%

\newpage

\setcounter{tocdepth}{1}

\tableofcontents

%%%%%%%%%%%%%%%%%%%%%
% Content begins here
%%%%%%%%%%%%%%%%%%%%%

\section{Introduction}

A remarkable relation between the geometry of the moduli spaces $\oM_{g,n}$ of stable algebraic curves of genus $g$ with $n$ marked points and integrable systems has been an object of intensive research during more than 30 years. This relation was first manifested by Witten's conjecture, see~\cite{Wit91}, proved by Kontsevich, see~\cite{Kon92}, saying that the generating series of the integrals
\begin{gather}\label{eq:psi-integral}
\int_{\oM_{g,n}}\prod_{i=1}^n\psi_i^{d_i},\quad g,n\ge 0,\,\, d_i\ge 0,
\end{gather}
where~$\psi_i\in H^2(\oM_{g,n},\mbC)$, $1\le i\le n$, is the first Chern class of the $\supth{i}$ cotangent line bundle over $\oM_{g,n}$, gives a solution of the Korteweg--de Vries (KdV) hierarchy. It was then realized by Dubrovin and Zhang, see~\cite{DZ01}, that integrable systems appear in a much more general context where the central role is played by the notion of a \emph{cohomological field theory} (CohFT), introduced by Kontsevich and Manin; see~\cite{KM94}. A cohomological field theory is a family of cohomology classes on the moduli spaces $\oM_{g,n}$, depending also on a vector in the $\supth{n}$ tensor power of a fixed vector space~$V$, that satisfy certain compatibility properties with respect to natural maps between different moduli spaces. Given a CohFT, Dubrovin and Zhang constructed $N:=\dim V$ generating series~$w^{\top;\alpha}$, $1\le\alpha\le N$, by inserting cohomology classes forming the CohFT in the integrals~\eqref{eq:psi-integral} and proved that the resulting $N$-tuple $\ow^{\,\top}:=(w^{\top;1},\ldots,w^{\top;N})$ is a solution of an integrable hierarchy (of evolutionary PDEs with one spatial variable) canonically associated to our CohFT. This integrable hierarchy is called the \emph{Dubrovin--Zhang} (\emph{DZ}) \emph{hierarchy}, and the solution $\ow^{\,\top}$ is called the \emph{topological solution}. However, certain fundamental properties of this hierarchy, including the polynomiality of the equations, were left as an open problem. The polynomiality of the equations of the hierarchy was proved in~\cite{BPS12} for semisimple CohFTs (together with the polynomiality of the Hamiltonian structure, which we do not discuss in our paper). However, a satisfactory formula for the equations of the hierarchy was not found.   

In~\cite{Bur15}, the first author suggested a new construction of an integrable system associated to an arbitrary CohFT. The construction uses certain cohomology classes on $\oM_{g,n}$ called the \emph{double ramification} (\emph{DR}) \emph{cycles}, and so the new hierarchy was called the \emph{DR hierarchy}. In contrast to the DZ hierarchy, the equations of the DR hierarchy are polynomial by construction, with a relatively simple formula for the coefficients given in terms of the intersection numbers of cohomology classes forming the CohFT and basic cohomology classes on~$\oM_{g,n}$ including the DR cycles and the classes $\psi_i$. However, it is not known how to single out a solution of the DR hierarchy that has such a simple geometric interpretation as the topological solution of the DZ hierarchy. There is a choice of a solution that is natural for other reasons; see~\cite{Bur15,BDGR18}.

In~\cite{Bur15}, the first author conjectured that for an arbitrary CohFT, the DR and the DZ hierarchies are Miura equivalent, and in~\cite{BDGR18}, this conjecture was made more precise, giving a precise description of the required Miura transformation in terms of the partition function of the CohFT. In~\cite{BGR19}, the authors presented a family of relations in the cohomology of $\oM_{g,n}$ implying the Miura equivalence of the DR and the DZ hierarchies. 

In \cite{BR21,ABLR21}, the authors extended the construction of the DR hierarchy to objects that are much more general than CohFTs, the so-called \emph{F-CohFTs}, introduced in~\cite{BR21}. Regarding the DZ hierarchies, their generalization for F-CohFTs was not considered in the literature before. In this paper, following the approach from~\cite{BPS12}, we show that there is a natural generalization of the DZ hierarchies for F-CohFTs. 

In this paper, for any $m\ge 0$, we present a family of conjectural relations in the cohomology of $\oM_{g,n+m}$ parameterized by integers $d_1,\ldots,d_n\ge 0$ satisfying $\sum d_i\ge 2g+m-1$. For fixed $g,n,m,d_1,\ldots,d_n$, our conjectural relation lies in~$H^{2\sum d_i}(\oM_{g,n+m},\mbC)$. We explain that these conjectural relations naturally imply the following fundamental properties of the DR and the DZ hierarchies: 
\begin{itemize}
\item For $m\ge 2$, the relations imply the polynomiality of the DZ hierarchy associated to an arbitrary F-CohFT (Theorem~\ref{theorem:implication of conjecture1}).

\item For $m=1$, the relations imply that the DR and the DZ hierarchies associated to an arbitrary F-CohFT are related by a Miura transformation (Theorem~\ref{theorem:implication of conjecture2}). For $n=1$, the relation already appeared in~\cite{BHIS21} and proved to be true in the Gorenstein quotient of~$\oM_{g,2}$ in~\cite{Gub22}.

\item For $m=0$, the relations already appeared in~\cite{BGR19}. According to~\cite{DZ01,BDGR18}, the DR and the DZ hierarchies associated to an arbitrary CohFT are endowed with an additional structure, called a \emph{tau-structure}. There are Miura transformations that preserve a tau-structure; they are called \emph{normal Miura transformations}. By a result from~\cite{BGR19}, for $m=0$ the relations imply that the DR and the DZ hierarchies associated to an arbitrary CohFT are related by a normal Miura transformation (see Section~\ref{subsubsection:implication of conjecture3}).
\end{itemize}
We can thus view our family of conjectural relations as a natural extension of the family of relations presented in~\cite{BGR19}.

\begin{example}
Let us present our relations in the case $n=1$, leaving the general case to Section~\ref{section:conjectural relations}. For any $k\ge 1$ and $m,g,d\ge 0$, let us introduce the following set:
$$
\cS^{m,k}_{g,d}\coloneqq\left\{(\og,\od)\in\mbZ_{\ge 0}^k\times\mbZ_{\ge 0}^k\;\left|\;
\begin{minipage}{9cm}
\small
$d_1+\cdots+d_k+k-1=d$,\\ $g_1\ge\delta_{m\le 1}$, $g_2,\dots,g_{k}\geq 1$, $g_1+\cdots+g_k=g$,\\ $d_1+\cdots+d_i+i-1\leq 2(g_1+\cdots+g_{i})-2+m$ for any $i=1,\ldots,k-1$
\end{minipage}\right.\right\},
$$
where $\og=(g_1,\ldots,g_k)$ and $\od=(d_1,\ldots,d_k)$.
\begin{itemize}
\item Suppose $m\ge 2$. Our relations form a family of cohomological relations in $H^{2d}(\oM_{g,m+1},\mbC)$, for any $d\ge 2g-1+m$, and are given by
\begin{gather*}
\sum_{k\ge 1}(-1)^{k-1}\sum_{(\og,\od)\in\cS^{m,k}_{g,d}} 
\vcenter{\xymatrix@C=25pt@R=5pt{
{}\ar@{.}@/_/[dd] & & & & & & \\
& *+[o][F-]{{g_1}}\ar@{-}[lu]*{{}_{2\,}}\ar@{-}[ld]*{{}_{m+1\hspace{0.2cm}}}\ar@{-}[l]\ar@{-}[r]^<<<{\psi^{d_1}} & *+[o][F-]{{g_2}}\ar@{-}[r]^<<<{\psi^{d_2}} & \cdots & *+[o][F-]{{g_{k-1}}}\ar@{-}[l]\ar@{-}[r]^<<<<{\psi^{d_{k-1}}} & *+[o][F-]{{g_k}}\ar@{-}[r]*{{}_{\,1}}^<<<{\psi^{d_k}} &\\
& & & & & & }}=0,
\end{gather*}
where we use the standard way to represent tautological cohomology classes using decorated stable graphs (see the details in Section~\ref{section:conjectural relations}).

\item Suppose $m=1$. Our relations form a family of cohomological relations in $H^{2d}(\oM_{g,2},\mbC)$, for any $d\ge 2g$, and are given by
\begin{align*}
&\sum_{k\ge 1}(-1)^{k-1}\sum_{(\og,\od)\in\cS^{1,k}_{g,d}} 
\vcenter{\xymatrix@C=25pt@R=5pt{
& *+[o][F-]{{g_1}}\ar@{-}[l]*{{}_{2\,}}\ar@{-}[r]^<<<{\psi^{d_1}} & *+[o][F-]{{g_2}}\ar@{-}[r]^<<<{\psi^{d_2}} & \cdots & *+[o][F-]{{g_{k-1}}}\ar@{-}[l]\ar@{-}[r]^<<<<{\psi^{d_{k-1}}} & *+[o][F-]{{g_k}}\ar@{-}[r]*{{}_{\,1}}^<<<{\psi^{d_k}} &}}\\
&\;=\sum_{\substack{g_1,\ldots,g_l\ge 1\\\sum g_i=g}}\left(\prod_{i=1}^{l-1}\frac{g_i}{g_i+g_{i+1}+\cdots+g_l}\right) 
\xymatrix@C=25pt@R=5pt{
& *+[o][F-,]{{g_1}}\ar@{-}[l]*{{}_{2\,}}\ar@{-}[r] & *+[o][F-,]{{g_2}}\ar@{-}[r] & \cdots & *+[o][F-,]{{g_{l}}}\ar@{-}[l]\ar@{-}[r]*{{}_{\,1}} &}, 
\end{align*} 
where $l=d-2g+1$, $\xymatrix@C=17pt@R=0pt{
& *+[o][F-,]{{g}}\ar@{-}[l]\ar@{-}[r] & 
}
\coloneqq \lambda_g\DR_g(1,-1)$, $\DR_g(1,-1)$ denotes the double ramification cycle, and $\lambda_g$ is the top Chern class of the Hodge bundle over $\oM_{g,n}$.

\item Suppose $m=0$. Our relations form a family of cohomological relations in $H^{2d}(\oM_{g,1},\mbC)$, for any $d\ge 2g-1$, and are given by
\begin{align*}
&\sum_{k\ge 1}(-1)^{k-1}\sum_{(\og,\od)\in\cS^{0,k}_{g,d}} 
\xymatrix@C=25pt@R=5pt{
*+[o][F-]{{g_1}}\ar@{-}[r]^<<<{\psi^{d_1}} & *+[o][F-]{{g_2}}\ar@{-}[r]^<<<{\psi^{d_2}} & \cdots & *+[o][F-]{{g_{k-1}}}\ar@{-}[l]\ar@{-}[r]^<<<<{\psi^{d_{k-1}}} & *+[o][F-]{{g_k}}\ar@{-}[r]*{{}_{\,1}}^<<<{\psi^{d_k}} &}\\
&\;=\sum_{\substack{g_1,\ldots,g_l\ge 1\\\sum g_i=g}}\left(\prod_{i=1}^{l-1}\frac{g_i}{g_i+g_{i+1}+\cdots+g_l}\right)\pi_*\left(
\xymatrix@C=25pt@R=5pt{
& *+[o][F-,]{{g_1}}\ar@{-}[l]*{{}_{2\,}}\ar@{-}[r] & *+[o][F-,]{{g_2}}\ar@{-}[r] & \cdots & *+[o][F-,]{{g_{l}}}\ar@{-}[l]\ar@{-}[r]*{{}_{\,1}} &}\right), 
\end{align*} 
where $l=d-2g+2$ and the map $\pi\colon\oM_{g,2}\to\oM_{g,1}$ forgets the second marked point.
\end{itemize}
\end{example}

We then prove all our relations in the case $n=1$ and arbitrary~$g$ (Theorem~\ref{theorem:main}) using the method for constructing relations in $H^*(\oM_{g,n},\mbC)$ from the paper~\cite{LP11}. In particular, this proves the main conjecture from~\cite{BHIS21} and the conjectural relations from~\cite{BGR19} in the case $n=1$. We also prove all our relations in the case $g=0$ and arbitrary $n$ (Theorem~\ref{theorem:main2}).

Finally, we fill a gap in the understanding of the equations of the DZ hierarchy mentioned above. An equation of the DZ hierarchy associated to an F-CohFT is the sum of the polynomial part and the fractional part (which conjecturally vanishes). We present a geometric formula for the polynomial part. The formula expresses the coefficient of a monomial in the polynomial part as the intersection of some universal cohomology class on $\oM_{g,n}$ with an element of the F-CohFT. Since the polynomiality of the equations of the DZ hierarchy is proved for semisimple CohFTs, this gives a geometric formula for the equations of the DZ hierarchy in this case.

\subsection*{Organization of the paper} Our conjectural relations are presented in Section~\ref{section:conjectural relations}. For the case $m\ge 2$, this is formulated in Conjecture~\ref{conjecture1}, for the case $m=1$ in Conjecture~\ref{conjecture2}, and for the case $m=0$ in Conjecture~\ref{conjecture3}. As we already explained, Conjecture~\ref{conjecture3} was first proposed in~\cite{BGR19}. Then in Section~\ref{section:equivalent formulation}, we present an alternative formula for classes $B^m_{g,(d_1,\ldots,d_n)}$ appearing in the conjectures (Theorem~\ref{theorem:arbitrary m reformulation}) and a particularly elegant reformulation of Conjecture~\ref{conjecture1} (Theorem~\ref{theorem:equivalent relations for m2}). In Section~\ref{section:conjectural relations and fundamental properties}, we explain the role of our conjectures in the study of integrable systems associated to cohomological field theories and F-cohomological field theories (see Theorems~\ref{theorem:implication of conjecture1} and~\ref{theorem:implication of conjecture2} and Section~\ref{subsubsection:implication of conjecture3}). This section is independent of the other sections; a reader who is interested only in the geometrical part of our results can skip it. In Section~\ref{section:proof of main theorem}, we prove the conjectures in the case $n=1$, arbitrary $g$. In Section~\ref{section:genus 0}, we prove the conjectures in the case $g=0$, arbitrary $n$. In Section~\ref{section:reduction}, we prove that the whole system of conjectural relations for $m\ge 2$ (\textit{i.e.}, Conjecture~\ref{conjecture1}) follows from its subsystem with $\sum d_i=2g+m-1$ and $d_i\ge 1$ (Theorem~\ref{theorem:reductionSmall}). Finally, in the appendix, we review a localization formula for the moduli space of stable relative maps to $(\mbP^1,\infty)$, which is the main tool for our proof of the conjectures in the case $n=1$.

\subsection*{Notation and conventions}  

\begin{itemize}
\item We denote by $H^i(X)$ the cohomology groups of a topological space $X$ with coefficients in $\mbC$. Let $H^\even(X):=\bigoplus_{i\ge 0}H^{2i}(X)$.

\item Let $\llbracket n \rrbracket \coloneqq \{1,\dots,n\}$. Given a map $\llbracket n \rrbracket \to \mathbb{Z}_{\geq 0}$, $i\mapsto d_i$, and a subset $I\subseteq \llbracket n \rrbracket$, let~$d_I$ denote $\sum_{i\in I} d_i$ (in particular, $d_\emptyset = 0$ and $d_{\{i\}}=d_i$).

\item Let $(a)_n$ denote the Pochhammer symbol $(a)_n\coloneqq \Gamma(a+1)/\Gamma(a+1-n)$. In particular, $(a)_0=1$ and $(a)_n = a(a-1)\cdots (a-n+1)$ for $n\geq 1$.

\item We use the standard convention for sums over repeated Greek indices.

\item We will work with the moduli spaces $\oM_{g,n}$ of stable algebraic curves of genus $g$ with $n$ marked points, which are defined for $g,n\ge 0$ satisfying the condition $2g-2+n>0$. We will often omit mentioning this condition explicitly and silently assume that it is satisfied when a moduli space is considered. 
\end{itemize}

%%%%%%%%%%%%%%%%%%%%%%%%%%%%%%%%%%%%%%%%%%%%%%%%%%%%%%%%%%%%%
%%%%%%%%%%%%%%%%%%%%%%%%%%%%%%%%%%%%%%%%%%%%%%%%%%%%%%%%%%%%%

\section{Conjectural cohomological relations}\label{section:conjectural relations}

\subsection{Tautological cohomology of $\boldsymbol{\oM_{g,n}}$}

Let us recall briefly the standard notation concerning tautological cohomology classes on $\oM_{g,n}$, referring a reader to~\cite[Sections 0.2 and 0.3]{PPZ15} for more details.

We use the standard cohomology classes on $\oM_{g,n}$:
\begin{itemize}
\item The $\supth{i}$ \emph{psi class} $\psi_i\in H^2(\oM_{g,n})$, $1\le i\le n$, is the first Chern class of the cotangent line bundle $\mathbb{L}_i\to\oM_{g,n}$ whose fibers are the cotangent spaces at the $\supth{i}$ marked point on stable curves.

\item The $\supth{i}$ \emph{kappa class} $\kappa_i \in H^{2i}(\oM_{g,n})$, $i\ge 0$, is defined as $\kappa_i:=\pi_*(\psi_{n+1}^{i+1})$, where $\pi\colon\oM_{g,n+1}\to\oM_{g,n}$ is the map forgetting the last marked point.

\item The $\supth{i}$ \emph{Hodge class} $\lambda_i\in H^{2i}(\oM_{g,n})$, $i\ge 0$, is the $\supth{i}$ Chern class of the Hodge vector bundle $\mathbb{E}_g\to\oM_{g,n}$ whose fibers are the spaces of holomorphic differentials on stable curves. 
\end{itemize}

We denote by $G_{g,n}$ the set of stable graphs of genus $g$ with $n$ legs marked by numbers $1,\ldots,n$. For a stable graph $\Gamma$, we use the following notation:
\begin{itemize}
\item $V(\Gamma)$, $E(\Gamma)$, $H(\Gamma)$, and $L(\Gamma)$ are the sets of vertices, edges, half-edges, and legs of $\Gamma$, respectively. 

\item The leg of $\Gamma$ marked by $1\le i\le |L(\Gamma)|$ is denoted by~$\sigma_i$. 

\item For $h\in H(\Gamma)$, let $v(h)$ be the vertex incident to $h$. 

\item Set~$H^e(\Gamma):=H(\Gamma)\backslash L(\Gamma)$, and for $h\in H^e(\Gamma)$, we denote by $\iota(h)$ a unique half-edge that together with $h$ forms an edge of $\Gamma$. 

\item For $v\in V(\Gamma)$, let
\begin{itemize}
\item $g(v)$ be the genus of $v$,
\item $n(v)$ be the degree of $v$,
\item $H[v]$ be the set of half-edges incident to $v$,
\item $r(v):=2g(v)-2+n(v)$. 
\end{itemize}  
\end{itemize} 

Consider a stable graph $\Gamma$.
\begin{itemize}
\item We associate to $\Gamma$ the space $\oM_{\Gamma}:=\prod_{v\in V(\Gamma)}\oM_{g(v),n(v)}$.

\item There is a canonical morphism $\xi_\Gamma\colon\oM_{\Gamma}\to\oM_{g(\Gamma),|L(\Gamma)|}$. Here $g(\Gamma)\coloneqq b_1(\Gamma)+\sum_{v\in V(\Gamma)}g(v)$, where $b_1(\Gamma)$ is the first Betti number of $\Gamma$.

\item A \emph{decoration} on $\Gamma$ is a choice of numbers $x_i[v],y[h]\ge 0$, $i\ge 1$, $v\in V(\Gamma)$, $h\in H(\Gamma)$. Given such numbers, we say that we have a \emph{decorated stable graph}. 

\item To a decorated stable graph, we associate the cohomology classes
$$
\gamma:=\prod_{v\in V(\Gamma)}\prod_{i\ge 1}\kappa_i[v]^{x_i[v]}\cdot\prod_{h\in H(\Gamma)}\psi_h^{y[h]}\in H^*(\oM_\Gamma)\quad\text{and}\quad \xi_{\Gamma*}(\gamma)\in H^*(\oM_{g(\Gamma),|L(\Gamma)|}),
$$
where $\kappa_i[v]$ is the $\supth{i}$ kappa class on $\oM_{g(v),n(v)}$ and $\psi_h$ is a psi class on $\oM_{g(v(h)),n(v(h))}$. The class $\xi_{\Gamma*}(\gamma)$ is called a \emph{basic tautological class} on~$\oM_{g(\Gamma),|L(\Gamma)|}$. We will often denote it by a picture of the decorated stable graph.
\end{itemize}

Denote by $R^*(\oM_{g,n})$ the subspace of $H^*(\oM_{g,n})$ spanned by all basic tautological classes. The subspace $R^*(\oM_{g,n})\subset H^*(\oM_{g,n})$ is closed under multiplication and is called the \emph{tautological ring} of $\oM_{g,n}$. Let $R^i(\oM_{g,n}):=R^*(\oM_{g,n})\cap H^{2i}(\oM_{g,n})$. Linear relations between basic tautological classes are called \emph{tautological relations}.

The Hodge classes on $\oM_{g,n}$ are tautological; $\lambda_i\in R^i(\oM_{g,n})$.

Recall the string equation, see~\cite{Wit91}, in the following form. Let $\pi\colon\oM_{g,k+q+1}\to \oM_{g,k}$ be the projection forgetting the last $q+1$ marked points. Then
\begin{gather}\label{eq:Str1}
\pi_*\left(\prod_{i=1}^k \psi_i^{q_i}\right) = \sum_{\substack{0 \le p_i\le q_i,\,i\in\llbracket k\rrbracket \\ q_{\llbracket k \rrbracket} - p_{\llbracket k \rrbracket} = q+1}}  \frac{(q+1)!}{\prod_{i=1}^k (q_i-p_i)!} \prod_{i=1}^k \psi_i^{p_i}. 
\end{gather}

For $A=(a_1,\dots,a_n)\in\mbZ^n$, $\sum_{i=1}^n a_i =0$, let $\DR_g(A) \in H^{2g}(\oM_{g,n})$ be the {\it double ramification} (\emph{DR}\,) \emph{cycle}. Let us briefly recall  the definition. The positive $a_i$ define a partition $\mu=(\mu_1,\ldots,\mu_{l(\mu)})$, and the negative $a_i$ define a second partition $\nu=(\nu_1,\ldots,\nu_{l(\nu)})$ of the same size. Let $n_0:=n-l(\mu)-l(\nu)$, and consider the moduli space $\oM_{g,n_0}(\mbP^1,\mu,\nu)^\sim$ of stable relative maps to rubber $\mbP^1$ with ramification profiles $\mu,\nu$ over the points $0,\infty\in\mbP^1$, respectively. Then the double ramification cycle $\DR_g(A)$ is defined as the Poincar\'e dual to the pushforward of the virtual fundamental class $\left[\oM_{g,n_0}(\mbP^1,\mu,\nu)^\sim\right]^\vir$ to $\oM_{g,n}$ via the forgetful map $\oM_{g,n_0}(\mbP^1,\mu,\nu)^\sim\to\oM_{g,n}$. 

Abusing notation, for $A=(a_1,\ldots,a_n)\in\mbZ^n$ and $B=(b_1,\ldots,b_m)\in\mbZ^m$ satisfying $\sum a_i+\sum b_j=0$, we denote by $\DR_g(A,B)$ the double ramification cycle $\DR_g(a_1,\ldots,a_n,b_1,\ldots,b_m)$. 

We have $\DR_g(A)\in R^g(\oM_{g,n})$ (see, \textit{e.g.},~\cite{JPPZ17}). The restriction $\left.\DR_g(A)\right|_{\cM^{\ct}_{g,n}}\in H^{2g}(\cM^{\ct}_{g,n})$ depends polynomially on the $a_i$, where by $\cM^\ct_{g,n}\subset\oM_{g,n}$ we denote the moduli space of curves of compact type. This implies that the class $\lambda_g\DR_g(A)\in R^{2g}(\oM_{g,n})$ depends polynomially on the $a_i$. Moreover, the resulting polynomial (with the coefficients in $R^{2g}(\oM_{g,n})$) is homogeneous of degree $2g$. There is also the following property that we will need. If $g\ge 1$ and $\pi\colon\oM_{g,n+1}\to\oM_{g,n}$ is the map forgetting the last marked point, then, see~\cite[Lemma~5.1]{BDGR18},  
\begin{gather}\label{eq:divisibility of DR}
\text{the polynomial class $\pi_*\left(\lambda_g\DR_g\left(-\sum a_i,a_1,\ldots,a_n\right)\right)\in R^{2g}(\oM_{g,n})$ is divisible by $a_n^2$}.
\end{gather}

The polynomiality of the class $\DR_g(A)\in R^g(\oM_{g,n})$ has been proved by A. Pixton and D. Zagier (we thank A.~Pixton for informing us about that), but the proof is not published~yet.

\subsection{Preliminary combinatorial definitions} \label{sec:srt-definitions}

By a \emph{stable tree}, we mean a stable graph $\Gamma$ with the first Betti number $b_1(\Gamma)$ equal to zero. A \emph{stable rooted tree} is a stable tree together with a choice of a vertex $v\in V(T)$ called the \emph{root}.

Consider a stable rooted tree $T$.
\begin{itemize}
\item We denote by $H_+(T)$ the set of half-edges of $T$ that are directed away from the root. Clearly, $L(T)\subset H_+(T)$. Let $H^e_+(T):=H_+(T)\backslash L(T)$ and $H^e_-(T):=H^e(T)\backslash H^e_+(T)$. 

\item A \emph{path} in $T$ is a sequence of pairwise distinct vertices $v_1,\ldots,v_k\in V(T)$ such that for any $1\le i\le k-1$, the vertices $v_i$ and $v_{i+1}$ are connected by an edge. 

\item A vertex $w\in V(T)$ is called a \emph{descendant} of a vertex $v\in V(T)$ if $v$ is on the unique path from the root to $w$. Denote by $\Desc[v]$ the set of all descendants of $v$. Note that $v\in\Desc[v]$.

\item A vertex $w$ is called a \emph{direct descendant} of $v$ if $w\in\Desc[v]$, $w\ne v$, and $w$ and $v$ are connected by an edge. In this case, the vertex $v$ is called the \emph{mother} of $w$. 

\item For two half-edges $h_1,h_2\in H_+(T)$, we say that $h_2$ is a \emph{descendant} of $h_1$ if $h_1=h_2$ or $v(h_2)\in\Desc[v(\iota(h_1))]$.
  
\item A function $l \colon V(T)\to\mbZ_{\ge 1}$ is called a \emph{level function} if the following conditions are satisfied:
\begin{itemize}
\item[a)] The value of $l$ on the root is equal to $1$.
	
\item[b)] If a vertex $v$ is the mother of a vertex $v'$, then $l(v')>l(v)$.

\item[c)] For every $1\le i\le \deg(l)$, the set $l^{-1}(i)$ is nonempty, where $\deg(l):=\max_{v\in V(T)}l(v)$.
\end{itemize}

\item There is a natural level function $l_T \colon V(T)\to\mbZ_{\ge 1}$ uniquely determined by the condition that if a vertex~$v$ is the mother of a vertex $v'$, then $l_T(v')=l_T(v)+1$. We call this level function \emph{canonical}. The number $\deg(T):=\deg(l_T)$ is called the \emph{degree} of~$T$.

\item For a level function $l\colon V(T)\to\mbZ_{\ge 1}$, it is convenient to extend it to $H_+(T)$ by taking $l(h):=k$ if the half-edge $h$ is attached to a vertex of level $k$.

\item For $k\ge 1$, we set $g_k(T):=\sum_{\substack{v \in V(T) \\ l_T(v) \leq k}} g(v)$.
\end{itemize}

Let $m\ge 0$ and $n\ge 1$. Let us consider stable rooted trees $T$ with at least $n+m$ legs, where we split the set of legs $L(T)=\{\sigma_i\}$ into three subsets:
\begin{itemize}
\item[a)] The legs $\sigma_1, \dotsc, \sigma_n$ are called the \emph{regular legs}. 

\item[b)] The legs $\sigma_{n+1},\ldots,\sigma_{n+m}$ are called the \emph{frozen legs}; we require that they are attached to the root. 

\item[c)] Any \emph{extra legs}, whose set is denoted by $F(T)$, correspond to additional marked points that we will eventually forget.
\end{itemize}
The set of such trees will be denoted by $\SRT_{g,n,m;\circ}$. We will also use the following notation:
\begin{align*}
&\SRT_{g,n,m}:=\left\{T\in\SRT_{g,n,m;\circ}\;|\;F(T)=\emptyset\right\}\subset\SRT_{g,n,m;\circ},\\
&\SRT^k_{g,n,m;\circ}:=\left\{T\in\SRT_{g,n,m;\circ}\;|\;|V(T)|=k\right\}\subset\SRT_{g,n,m;\circ}.
\end{align*}

Consider a tree $T\in\SRT_{g,n,m;\circ}$.
\begin{itemize}
\item A vertex of $T$ is called \emph{potentially unstable} if it becomes unstable once we forget all of the extra legs.

\item Let 
$$
H^{em}_+(T):=H^e_+(T)\sqcup\{\sigma_1,\ldots,\sigma_{n+m}\},\quad \tH^{em}_+(T):=H^e_+(T)\sqcup\{\sigma_1,\ldots,\sigma_{n}\}.
$$

\item For $h\in \tH_+^{em}(T)$, define $I_h:=\{1\le i\le n\;|\;\text{$\sigma_i$ is a descendant of $h$}\}\subseteq \llbracket n \rrbracket$.
\end{itemize}

A tree $T\in\SRT_{g,n,m;\circ}$ is called \emph{balanced} if 
\begin{itemize}
\item[a)] there are no extra legs attached to the root;

\item[b)] for every vertex except the root, there is at least one extra leg attached to it.
\end{itemize}
The set of all balanced trees will be denoted by $\SRT^{(b)}_{g,n,m;\circ}\subset\SRT_{g,n,m;\circ}$.

For a balanced tree $T$, define a function $q\colon H^e_+(T) \to \mbZ_{\ge 0}$ by requiring that for a half-edge $h\in H_+^e(T)$, there are exactly $q(h)+1$ extra legs attached to the vertex~$v=v(\iota(h))$. Given an $n$-tuple of nonnegative integers $\od=(d_1,\ldots,d_n)$, we extend the function $q$ to the set $\tH_+^{em}(T)$ by setting $q(\sigma_i):=d_i$, $1\le i\le n$.

We say that a balanced tree $T\in\SRT^{(b)}_{g,n,m;\circ}$ is \emph{complete} if the following conditions are satisfied:
\begin{itemize}
\item[a)] Every vertex has at least one descendant $v\in V(T)$ with $l_T(v)=\deg(T)$.

\item[b)] We have $l_T(\sigma_i)=\deg(T)$ for $1\le i\le n$. 

\item[c)] Each vertex $v$ with $l_T(v)=\deg(T)$ is attached to at least one regular leg. 

\item[d)] For every $1\le l\le\deg(T)$, the set of vertices $l_T^{-1}(l)$ contains at least one vertex that is not potentially unstable.
\end{itemize}
The set of all complete trees will be denoted by $\SRT^{(b,c)}_{g,n,m;\circ}\subset\SRT^{(b)}_{g,n,m;\circ}$.

We say that a tree $T\in\SRT^{(b,c)}_{g,n,m;\circ}$ is \emph{admissible} if for every $1\leq k < \deg(T)$, the following condition is satisfied:
\begin{align}\label{eq:admissibility-SRT}
\sum_{\substack{h \in H_+^e(T) \\ l_T(h)=k}} q(h) \leq 2 g_k(T)-2+m.
\end{align}
The set of all admissible trees will be denoted by $\SRT^{(b,c,a)}_{g,n,m;\circ}\subset\SRT^{(b,c)}_{g,n,m;\circ}$.

Note that the sets $\SRT^{(b)}_{g,n,m;\circ}$ and $\SRT^{(b,c)}_{g,n,m;\circ}$ are infinite except for a finite number of triples $(g,n,m)$. However, the set $\SRT^{(b,c,a)}_{g,n,m;\circ}$ is finite, which follows from  condition~\eqref{eq:admissibility-SRT}. 

\subsection{Conjectural tautological relations}

For a balanced tree $T\in\SRT^{(b)}_{g,n,m;\circ}$ and an $n$-tuple of nonnegative integers $\od=(d_1,\ldots,d_n)$, define
$$
[T,\od]:=\xi_{T*}\left(\prod_{h \in \tH_+^{em}(T)} \psi_h^{q(h)}\right) \in R^{\sum d_i+\#F(T)}(\oM_{g,n+m+\#F(T)}).
$$
Consider the map
$$
\ee \colon \oM_{g,n+m+\#F(T)} \lra \oM_{g,n+m}
$$
forgetting all of the extra legs, and the class
$$
\ee_*[T,\od]\in R^{\sum d_i}(\oM_{g,n+m}).
$$

\begin{definition}\label{def:Bgdm-MainDefinition}
For $m\ge 0$, $g\ge 0$, $n\ge 1$, and $\od=(d_1,\ldots,d_n)\in\mbZ_{\ge 0}^n$, we define
\begin{equation*}
B^m_{g,\od}:=\sum_{T\in\SRT^{(b,c,a)}_{g,n,m;\circ}}(-1)^{\deg(T)-1}\ee_*[T,\od]\in R^{\sum d_i}(\oM_{g,n+m}).
\end{equation*}
\end{definition}

Note that
\begin{itemize}
\item the class $B^0_{g,\od}$ coincides with the class $B^g_{d_1,\ldots,d_n}$ from the paper~\cite{BGR19},

\item the class $B^1_{g,2g}$ coincides with the class $B^g$ from the paper~\cite{BHIS21}.
\end{itemize}

We can now formulate our first conjecture.

\begin{conjecture}\label{conjecture1}
For any $m\ge 2$, $g\ge 0$, $n\ge 1$, and $\od=(d_1,\ldots,d_n)\in\mbZ_{\ge 0}^n$ such that $\sum d_i\ge 2g+m-1$, we have $B^m_{g,\od}=0$ in $R^{\sum d_i}(\oM_{g,n+m})$.
\end{conjecture}

In order to present our second conjecture, we need more definitions. Let $n,k\ge 1$. Consider a stable tree $T\in\ST^k_{g,n,1}$ and integers $a_1,\ldots,a_{n+1}$ such that $a_1+\cdots+a_{n+1}=0$. There is a unique way to assign an integer~$a(h)$ to each half-edge $h\in H(T)$ in such a way that the following conditions hold:
\begin{itemize}

\item[a)] For any leg $\sigma_i\in L(T)$, we have $a(\sigma_i)=a_{i}$. 

\item[b)] If $h\in H^e(T)$, then $a(h)+a(\iota(h))=0$. 

\item[c)] For any vertex $v\in V(T)$, we have $\sum_{h\in H[v]}a(h)=0$. 

\end{itemize}
Consider the space $\oM_T=\prod_{v\in V(T)}\oM_{g(v),n(v)}$. For each vertex $v\in V(T)$, the numbers $a(h)$, $h\in H[v]$, define the double ramification cycle $\DR_{g(v)}\left(A_{H[v]}\right)\in R^{g(v)}(\oM_{g(v),n(v)})$, where $A_{H[v]}$ is the tuple $(a(h_1),\ldots,a(h_{n(v)}))$, where $\{h_1,\ldots,h_{n(v)}\}=H[v]$. If we multiply all of these DR cycles, we get the class
$$
\prod_{v\in V(T)}\DR_{g(v)}\left(A_{H[v]}\right)\in H^{2g}(\oM_T).
$$
We define a class $\DR_T(a_1,\ldots,a_{n+1})\in R^{g+k-1}(\oM_{g,n+1})$ by 
\begin{gather*}
\DR_T(a_1,\ldots,a_{n+1}):=\left(\prod_{h\in H^e_+(T)}a(h)\right)\xi_{T*}\left(\prod_{v\in V(T)}\DR_{g(v)}\left(A_{H[v]}\right)\right).
\end{gather*}
After multiplication by $\lambda_g$, this class becomes a polynomial in $a_1,\ldots,a_{n+1}$ with  coefficients in the space $R^{2g+k-1}(\oM_{g,n+1})$, which is homogeneous of degree $2g+k-1$. To the stable tree~$T$, we also assign a number~$C(T)$~by setting 
$$
C(T):=\prod_{v\in V(T)}\frac{r(v)}{\sum_{\tv\in\Desc[v]}r(\tv)}.
$$

We introduce the following cohomology class:
$$
\check{A}_g^k(a_1,\ldots,a_{n+1}):=\sum_{T\in\SRT^k_{g,n,1}}C(T)\lambda_g\DR_T(a_1,\ldots,a_{n+1})\in R^{2g+k-1}(\oM_{g,n+1}),
$$
depending on integers $a_1,\ldots,a_{n+1}$ with vanishing sum. The class 
$$
\check{A}_g^k\left(a_1,\ldots,a_n,-\sum a_i\right)
$$
is a polynomial in $a_1,\ldots,a_n$, homogeneous of degree $2g+k-1$. For an $n$-tuple of integers $\od=(d_1,\ldots,d_n)\in\mbZ_{\ge 0}^n$ satisfying $\sum d_i\ge 2g$, define 
\begin{gather*}
A^1_{g,\od}:=\Coef_{a_1^{d_1}\cdots a_n^{d_n}}\check{A}_g^{\sum d_i-2g+1}\left(a_1,\ldots,a_n,-\sum a_i\right)\in R^{\sum d_i}(\oM_{g,n+1}).
\end{gather*}

We can now present our second conjecture.

\begin{conjecture}\label{conjecture2}
For any $g\ge 0$, $n\ge 1$, and an $n$-tuple of nonnegative integers $\od=(d_1,\ldots,d_n)$ with $\sum d_i\ge 2g$, we have $B^1_{g,\od}=A^1_{g,\od}$ in $R^{\sum d_i}(\oM_{g,n+1})$.
\end{conjecture}

Note that the $n=1$ case of this conjecture appeared in~\cite{BHIS21}, and in~\cite{Gub22}, the author proved that the relation $B^1_{g,d}=A^1_{g,d}$ is true in the Gorenstein quotient of $R^d(\oM_{g,2})$.

Let us now recall the conjecture from~\cite{BGR19}, which together with our Conjectures~\ref{conjecture1} and~\ref{conjecture2} naturally forms a series of conjectures involving the classes $B^m_{g,\od}$ for all $m\ge 0$. Let $n,k\ge 1$. Consider the map $\pi\colon\oM_{g,n+1}\to\oM_{g,n}$ forgetting the last marked point. By~\cite[Lemma~2.2]{BGR19}, the class $\pi_*\check{A}_g^k(a_1,\ldots,a_n,-\sum a_i)$ (as a polynomial in $a_1,\ldots,a_n$) is divisible by~$\sum a_i$. Following the paper~\cite{BGR19}, we define
$$
A_g^k(a_1,\ldots,a_n):=\frac{1}{\sum a_i}\pi_*\check{A}_g^k\left(a_1,\ldots,a_n,-\sum a_i\right)\in R^{2g+k-2}(\oM_{g,n}),
$$
which is a polynomial in $a_1,\ldots,a_n$, homogeneous of degree $2g+k-2$, and then define
\begin{gather*}
A^0_{g,\od}:=\Coef_{a_1^{d_1}\cdots a_n^{d_n}}A_g^{\sum d_i-2g+2}(a_1,\ldots,a_n)\in R^{\sum d_i}(\oM_{g,n}) 
\end{gather*}
for any $n$-tuple $\od=(d_1,\ldots,d_n)\in\mbZ_{\ge 0}^n$ satisfying $\sum d_i\ge 2g-1$.

\begin{conjecture}[\textit{cf.} \cite{BGR19}]\label{conjecture3}
For any $g\ge 0$, $n\ge 1$, and an $n$-tuple of nonnegative integers $\od=(d_1,\ldots,d_n)$ with $\sum d_i\ge 2g-1$, we have $B^0_{g,\od}=A^0_{g,\od}$ in $R^{\sum d_i}(\oM_{g,n})$.
\end{conjecture}

We can prove all of the above conjectures in the following cases.

\begin{theorem}\label{theorem:main}
Conjectures~\ref{conjecture1},~\ref{conjecture2}, and~\ref{conjecture3} are true for $n=1$.
\end{theorem}

\begin{theorem}\label{theorem:main2}
Conjectures~\ref{conjecture1},~\ref{conjecture2}, and~\ref{conjecture3} are true for $g=0$.
\end{theorem}

The proofs will be presented in Sections~\ref{section:proof of main theorem} and~\ref{section:genus 0}, respectively.

%%%%%%%%%%%%%%%%%%%%%%%%%%%%%%%%%%%%%%%%%%%%%%%%%%%%%%%%%%%%%
%%%%%%%%%%%%%%%%%%%%%%%%%%%%%%%%%%%%%%%%%%%%%%%%%%%%%%%%%%%%%

\section{An equivalent formulation of the conjectures}\label{section:equivalent formulation}

\subsection{The case $\boldsymbol{m\ge 2}$}

Let $x_1,\dots,x_n$ be formal variables assigned to the legs $\sigma_1,\dots,\sigma_n$. Recall that for $I\subset\br{n}$ we use the notation $x_I=\sum_{i\in I}x_i$.

\begin{definition}
For $g,m\ge 0$ and $n\ge 1$, define the following class in $R^*(\oM_{g,n+m})[x_1,\ldots,x_n]$:
\begin{gather}\label{eq:formula for P}
P_{g,n,m}\coloneqq \prod_{i=1}^n x_i^{-1}\sum_{T\in\SRT_{g,n,m}} (-1)^{|E(T)|} \sum_{p\colon \tH_+^{em}(T)\to \mbZ_{\geq 0}} \xi_{T*}\left(\prod_{h \in \tH_+^{em}(T)} \psi_h^{p(h)}\right) \prod_{h \in\tH_+^{em}(T)} x_{I_h}^{p(h)+1}.
\end{gather}
\end{definition}

Note that given $T\in\SRT_{g,n,m}$ and $p\colon \tH_+^{em}(T)\to \mbZ_{\geq 0}$, we have $\xi_{T*}\left(\prod_{h \in \tH_+^{em}(T)} \psi_h^{p(h)}\right)=0$ unless $\sum_{h\in\tH^{em}_+[v]}p(h)\le 3g(v)-3+n(v)$ for each $v\in V(T)$. This implies that the second sum in the definition of $P_{g,n,m}$ has a finite number of nonzero terms. Clearly, the coefficient of $x_1^{d_1}\cdots x_n^{d_n}$ in $P_{g,n,m}$ is a tautological class from $R^{\sum d_i}(\oM_{g,n+m})$, and it is equal to the sum of $\psi_1^{d_1}\cdots \psi_n^{d_n}$ and a class supported on the boundary of $\oM_{g,n+m}$. 

\begin{example} \label{ex:threevertextree}
For instance, let $T\in\SRT_{g,3,2}$ be as in the following picture, with $g=g_1+g_2+g_3$ (the root vertex is labeled by $g_1$, and hence the legs $\sigma_4$ and $\sigma_5$ are attached to it) and the values of the function $p\colon \tH_+^{em}(T)\to\mbZ_{\geq 0}$ being $p_1,\ldots,p_5$ at the corresponding half-edges in the picture:
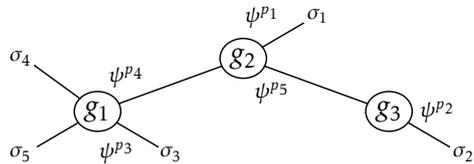
\begin{gather*} 
\vcenter{\xymatrix@C=15pt@R=5pt{
& & & & & & \\
& & & *+[o][F-]{{g_2}} \ar@{-}[ru]*{{}_{\,\,\sigma_1}}^<<<{\psi^{p_1}} \ar@{-}[rrd]_<<<{\psi^{p_5}} & & & \\ 
& *+[o][F-]{{g_1}}\ar@{-}[rru]^<<<{\psi^{p_4}}\ar@{-}[rd]*{{}_{\,\sigma_3}}_<<<{\psi^{p_3}}\ar@{-}[ul]*{{}_{\sigma_4\,\,\,}}\ar@{-}[dl]*{{}_{\sigma_5\,\,\,}} & &  & & *+[o][F-]{{g_3}} \ar@{-}[rd]*{{}_{\,\,\sigma_2}}^<<<{\hspace{-0.05cm}\psi^{p_2}} & \\
& & & & & &}}\,.
\end{gather*}
Then the class that this pair $(T,p)$ contributes to $P_{g,3,2}$ is multiplied by the factor
\begin{equation*}
 x_1^{p_1} x_2^{p_2+p_5+1} x_3^{p_3} (x_1+x_2)^{p_4+1}. 
\end{equation*}
\end{example}

\begin{remark}
Since in our notation we have $x_\emptyset =0$, this implies the following condition for a tree $T$ to be able to contribute nontrivially to $P_{g,n,m}$. Namely, for each $h \in \tH_+^{em}(T)$, there must exist at least one leg $\sigma_i$, $i=1,\ldots,n$, which is a descendant of $h$; that is, $I_h\ne\emptyset$. This means that any vertex of $T$ that does not have direct descendants should have at least one leg~$\sigma_i$, $1\le i\le n$, attached to it. 
\end{remark}

\begin{theorem}\label{theorem:equivalent relations for m2}
The following two statements are equivalent:
\begin{enumerate}
\item Conjecture~\ref{conjecture1} is true. 

\item We have $\deg P_{g,n,m}\leq 2g+m-2$ for all $g\geq 0$, $n\geq 1$, and $m\geq 2$. 
\end{enumerate}
\end{theorem}
\begin{proof}
Let us introduce an auxiliary cohomology class on $\oM_{g,n+m}$. We will say that a tree $T\in\SRT^{(b)}_{g,n,m;\circ}$ is \emph{nondegenerate} if it does not have potentially unstable vertices. The set of such trees will be denoted by $\SRT^{(b,nd)}_{g,n,m;\circ}\subset\SRT^{(b)}_{g,n,m;\circ}$. The set $\SRT^{(b,nd)}_{g,n,m;\circ}$ is infinite, except for a finite number of triples $(g,n,m)$. However, we have the following statement.

\begin{lemma}\label{lem35}
Given $\od\in\mbZ_{\ge 0}^n$, the set of trees $T\in\SRT^{(b,nd)}_{g,n,m;\circ}$ such that 
$
[T,\od]\ne 0\in R^{\sum d_i+|F(T)|}(\oM_{g,n+m+|F(T)|})
$
is finite.
\end{lemma}

\begin{proof}
For $T\in\SRT^{(b,nd)}_{g,n,m;\circ}$, denote by $\pi(T)\in\SRT_{g,n,m}$ the stable tree obtained from $T$ by forgetting all of the extra legs in~$T$. Clearly, we can identify $V(T)=V(\pi(T))$. It is sufficient to prove that for any $\tT\in\SRT_{g,n,m}$, there exists a constant $N_{\tT}$ such that if $\pi(T)=\tT$ and $[T,\od]\ne 0$ for $T\in\SRT^{(b,nd)}_{g,n,m;\circ}$, then $|F(T)|\le N_{\tT}$. Let us prove it by  induction on $|V(\tT)|$. If $|V(\tT)|=1$, then we can obviously set $N_{\tT}:=0$. Suppose  $|V(\tT)|\ge 2$. Let us choose a vertex $v\in V(\tT)$ that does not have direct descendants. Denote by $w$ the mother of~$v$ and by~$\tT'$ the stable tree obtained from $\tT$ by deleting the vertex $v$ together with all of the half-edges incident to it. Let us show that we can set
$$
N_{\tT}:=2 N_{\tT'}+3g(w)-2+n+m+|E(\tT)|.
$$
Indeed, denote by $h\in H^e_+[w]$ the half-edge such that $v(\iota(h))=v$ and by $T'$ the stable tree obtained from $T$ by deleting the vertex $v$ together with all of the half-edges incident to it. If $[T,\od]\ne 0$, then clearly 
$$
q(h)\le 3g(w)-3+n(w)\le 3g(w)-3+|E(\tT)|+n+m+|F(T')|.
$$
Since $q(h)=|F[v]|-1$, using the induction assumption, we see that $|F(T)|=|F(T')|+|F[v]|\le 2 N_{\tT'}+3g(w)-2+n+m+|E(\tT)|$, as required.
\end{proof}

For $g,m\ge 0$, $n\ge 1$, and $d_1,\ldots,d_n\ge 0$, let us consider the following cohomology class:
\begin{gather}\label{eq:definition of tB}
\tB^m_{g,\od}:=\sum_{T\in\SRT^{(b,nd)}_{g,n,m;\circ}}(-1)^{|E(T)|}\ee_*[T,\od]\in R^{\sum d_i}(\oM_{g,n+m}).
\end{gather}

Lemma~\ref{lem35} implies that this class is well defined.

\begin{proposition}\label{prop36}
Conjecture~\ref{conjecture1} is true if and only if $\tB^m_{g,\od}=0$ for all $g\ge 0$, $n\ge 1$, $m\ge 2$, and $d_1,\ldots,d_n\ge 0$ such that $\sum d_i>2g-2+m$.
\end{proposition}
\begin{proof}
The proof consists of two steps.

\emph{Step}~1.~ Let us prove that
\begin{gather}\label{eq:equivalent formula for tB}
\tB^m_{g,\od}=\sum_{T\in\SRT^{(b,c)}_{g,n,m;\circ}}(-1)^{\deg(T)-1}\ee_*[T,\od]\in R^{\sum d_i}(\oM_{g,n+m}),\quad\od\in\mbZ_{\ge 0}^n.
\end{gather}
Consider a tree $T\in\SRT^{(b)}_{g,n,m;\circ}$. It is clear that any potentially unstable vertex~$v$ in $T$ does not coincide with the root and has genus $0$, and there is exactly one half-edge $h_1$ from $\tH^{em}_+(T)$ attached to it. Moreover, there is exactly one half-edge $h_2$ from $H^e_-(T)$ attached to it, and $\ee_*[T,\od]=0$ unless $q(h_1)=q(\iota(h_2))$. Let us denote by $\tSRT^{(b)}_{g,n,m;\circ}$ the set of trees $T\in\SRT^{(b)}_{g,n,m;\circ}$ such that the condition 
\begin{gather}\label{eq:condition for tSRT}
q(h_1)=q(\iota(h_2))
\end{gather} 
is satisfied for any potentially unstable vertex. Note that $h_1$ in  condition~\eqref{eq:condition for tSRT} can be a leg~$\sigma_i$, $1\le i\le n$, so the set $\tSRT^{(b)}_{g,n,m;\circ}$ depends on the choice of $\od$.

Let us fix $\od$. Consider a tree $T\in\tSRT^{(b)}_{g,n,m;\circ}$ and a potentially unstable vertex $v\in V(T)$. Let us throw away the vertex $v$ together with all of the half-edges attached to it except $h_1$ and identify the half-edges $h_1$ and $\iota(h_2)$. We obtain a tree from $\tSRT^{(b)}_{g,n,m;\circ}$ with one less potentially unstable vertex. Repeating this procedure until there are no potentially unstable vertices, we obtain a tree from~$\SRT^{(b,nd)}_{g,n,m;\circ}$, which we denote by $\Core(T)$. We obviously have
$$
\ee_*[T,\od]=\ee_*[\Core(T),\od].
$$
Therefore, Equation~\eqref{eq:equivalent formula for tB} is a corollary of the following lemma.

\begin{lemma}
For any $\od\in\mbZ^n_{\ge 0}$ and $T\in\SRT^{(b,nd)}_{g,n,m;\circ}$, we have $\sum_{\substack{\tT\in\tSRT^{(b,c)}_{g,n,m;\circ}\\\Core(\tT)=T}}(-1)^{\deg(\tT)-1}=(-1)^{|E(T)|}$.
\end{lemma}
\begin{proof}
The proof is by the induction on~$|V(T)|$. The base case $|V(T)|=1$ is obvious.

Suppose  $|V(T)|>1$. Since $T=\Core(\tT)$ is obtained from $\tT$ by deleting some vertices, we can consider the set $V(T)$ as a subset of $V(\tT)$. Denote by $V^l(T)$ the set of vertices $v\in V(T)$ that do not have direct descendants. Consider a nonempty subset $I\subset V^l(T)$, and let 
$$
n_I:=|\{1\le i\le n\;|\;\text{$\sigma_i$ is attached to a vertex from $I$}\}|,\quad g_I:=\sum_{v\in I}g(v).
$$
Denote by $T'$ the tree obtained from $T$ by deleting all of the vertices from $I$ together with all of the half-edges attached to them, where we consider the half-edges $\iota(h)$, for $h\in H^e_-(T)$ attached to the vertices from $I$, as regular legs of the tree $T'$. Therefore, $T'\in\SRT^{(b,nd)}_{g-g_I,n-n_I+|I|,m;\circ}$. Denote by $S$ the set of trees $\tT\in\tSRT^{(b,c)}_{g,n,m;\circ}$ such that 
$$
\left\{\left.v\in V(\tT)\;\right|\;\text{$l_T(v)=\deg(\tT)$ and $v$ is not potentially unstable}\right\}=I.
$$
Deleting all of the vertices $v$ with $l_T(v)=\deg(T)$ together with all of the half-edges attached to them gives a bijection from the set $S$ to the set $\left\{\left.\tT'\in\tSRT^{(b,c)}_{g-g_I,n-n_I+|I|,m;\circ}\;\right|\;\Core(\tT')=T'\right\}$. Using the induction assumption, we obtain
$$
\sum_{\substack{\tT\in\tSRT^{(b,c)}_{g,n,m;\circ}\\\Core(\tT)=T}}\hspace{-0.2cm}(-1)^{\deg(\tT)-1}=\hspace{-0.1cm}\sum_{\substack{I\subset V^l(T)\\I\ne\emptyset}}\hspace{-0.1cm}(-1)^{|E(T)|-|I|+1}=(-1)^{|E(T)|+1}\sum_{k=1}^{|V^l(T)|}(-1)^k{|V^l(T)|\choose k}=(-1)^{|E(T)|},
$$
as required.
\end{proof}

\emph{Step}~2.~ Let us now prove the proposition using Equation~\eqref{eq:equivalent formula for tB}. For a tree $T\in\SRT^{(b,c)}_{g,n,m;\circ}$ and an $n$-tuple $\od$ satisfying $\sum d_i>2g-2+m$, denote by $l_0$ the minimal level such that $\sum_{\substack{h \in\tH_+^{em}(T) \\ l_T(h)=l_0}} q(h) > 2 g_{l_0}(T)-2+m$. Let us also introduce the following notation:
\begin{gather*}
I_0:=\left\{\left.h \in \tH_+^{em}(T)\;\right|\;l_T(h)=l_0\right\},\quad V_0:=\left\{\left.v\in V(T)\;\right|\;l_T(v)\le l_0\right\},\quad g_0:=g_{l_0}(T),
\end{gather*}
and for $h\in I_0$, let
\begin{gather*}
V_h:=\{v\in V(T)\;|\;\text{$v\in\Desc[v(\iota(h))]$}\},\quad g_h:=\sum_{v\in V_h}g(v).
\end{gather*}
Denote by $T_0$ the subtree of $T$ formed by the vertices from $V_0$, and denote by $T_h$ the subtree of~$T$ formed by the vertices from $V_h$, $h\in I_0$. We see that $T_0\in\SRT^{(b,c,a)}_{g_0,|I_0|,m;\circ}$ and $T_h\in\SRT_{g_h,|I_h|,1;\circ}$, where we consider $\iota(h)$ as a unique frozen leg in $T_h$. 

Suppose that Conjecture~\ref{conjecture1} is true. In the sum over $T\in\SRT^{(b,c)}_{g,n,m;\circ}$ on the right-hand side of~\eqref{eq:equivalent formula for tB}, we can collect together the terms with fixed $g_0$, $|I_0|$, and trees $T_h$. Denote it by $S$, and let $\od'$ be an $|I_0|$-tuple consisting of the numbers $q(h)$, $h\in I_0$. There is a gluing map from a product of moduli space, determined by the decomposition of the tree $T$ into the trees $T_0$ and $T_h$, $h\in I_0$, to $\oM_{g,n+m}$, and it is easy to see that the class $S$ is equal to the pushforward under this gluing map of the tensor product of the class $B^m_{g_0,\od'}$ and the classes determined by the trees $T_h$.  Since we assumed that Conjecture~\ref{conjecture1} is true, we conclude that $S=0$, and thus also $\tB^m_{g,\od}=0$.

Now suppose  $\tB^m_{g,\od}=0$ for all $g\ge 0$, $n\ge 1$, $m\ge 2$, and $d_1,\ldots,d_n\ge 0$ such that $\sum d_i>2g-2+m$. Let us prove that $B^m_{g,\od}=0$ in the same range of $g,n,m,\od$ by  induction on $2g-2+n+m$. The base case $2g-2+n+m=m-1$, which is achieved only for $g=0$ and $n=1$, is trivial because we obviously have $B^m_{0,d_1}=\tB^m_{0,d_1}=\psi_1^{d_1}$, which is zero because $d_1>m-2=\dim\oM_{0,m+1}$.  

Now suppose $2g-2+n+m>m-1$. We split the sum on the right-hand side of~\eqref{eq:equivalent formula for tB} as follows:
\begin{gather}\label{eq:split tB}
\tB^m_{g,\od}:=\sum_{\substack{T\in\SRT^{(b,c)}_{g,n,m;\circ}\\l_0<\deg(T)}}(-1)^{\deg(T)-1}\ee_*[T,\od]+\sum_{\substack{T\in\SRT^{(b,c)}_{g,n,m;\circ}\\l_0=\deg(T)}}(-1)^{\deg(T)-1}\ee_*[T,\od].
\end{gather}
For a tree $T\in\SRT^{(b,c)}_{g,n,m;\circ}$, it is clear that if $l_0<\deg(T)$, then $2g_0-2+|I_0|+m<2g-2+n+m$; this follows from the fact that for each $1\le\l\le\deg(T)$, there is at least one vertex $v$ with $l_T(v)=l$ that is not potentially unstable. Therefore, by the induction assumption, the first sum on the right-hand side of~\eqref{eq:split tB} is equal to zero. The second sum on the right-hand side of~\eqref{eq:split tB} is obviously equal to $B^m_{g,\od}$, and thus it is equal to zero, as required.
\end{proof}

Let us now continue the proof of Theorem~\ref{theorem:equivalent relations for m2}. Using Proposition~\ref{prop36}, we see that it is sufficient to prove the following lemma.

\begin{lemma}\label{lemma:tB and coefficient of P}
For any $n\ge 1$, $g,m\ge 0$, and $\od=(d_1,\ldots,d_n)\in\mbZ_{\ge 0}^n$, we have $\tB^m_{g,\od}=\Coef_{x_1^{d_1}\cdots x_n^{d_n}}P_{g,n,m}$.
\end{lemma}
\begin{proof}
Consider an arbitrary tree $T\in\SRT_{g,n,m}$ and a function $q\colon\tH^{em}_+(T)\to\mbZ_{\ge 0}$. For any $h\in H^e_+(T)$, let us attach $q(h)+1$ extra legs to the vertex $v(\iota(h))$. We obtain a tree from $\SRT^{(b,nd)}_{g,n,m;\circ}$, which we denote by $T_q$. The definition of the class $\tB^m_{g,\od}$ can be rewritten in the following way:
$$
\tB^m_{g,\od}=\sum_{T\in\SRT_{g,n,m}}(-1)^{|E(T)|}\sum_{\substack{q\colon\tH^{em}_+(T)\to\mbZ_{\ge 0}\\q(\sigma_i)=d_i}}\ee_*\xi_{T_q*}\left(\prod_{h\in\tH^{em}_+(T_q)}\psi_h^{q(h)}\right).
$$
Therefore, we have to prove that
$$
\sum_{T\in\SRT_{g,n,m}}(-1)^{|E(T)|}\sum_{q\colon\tH^{em}_+(T)\to\mbZ_{\ge 0}}\ee_*\xi_{T_q*}\left(\prod_{h\in\tH^{em}_+(T_q)}\psi_h^{q(h)}\right)\prod_{i=1}^n x_i^{q(\sigma_i)}=P_{g,n,m}.
$$
Clearly, this follows from the equation
\begin{multline}\label{eq:with and without extra legs}
\sum_{q\colon\tH^{em}_+(T)\to\mbZ_{\ge 0}}\ee_*\xi_{T_q*}\left(\prod_{h\in\tH^{em}_+(T_q)}\psi_h^{q(h)}\right)\prod_{i=1}^n x_i^{q(\sigma_i)}\\
=\sum_{p\colon \tH_+^{em}(T)\to \mbZ_{\geq 0}} \xi_{T*}\left(\prod_{h \in \tH_+^{em}(T)} \psi_h^{p(h)}\right) \prod_{h \in H_+^e(T)} x_{I_h}^{p(h)+1}\prod_{i=1}^n x_i^{p(\sigma_i)},\quad T\in\SRT_{g,n,m},
\end{multline}
which we are going to prove by  induction on $|V(T)|$.

The base case $|V(T)|=1$ is obvious. We proceed to the induction step and assume  $|V(T)|\ge 2$. Choose a vertex $v\in V(T)$ that does not have direct descendants. Denote by~$\th$ a unique half-edge from $H^e_-(T)$ incident to $v$, and let $h':=\iota(\th)$. Denote by $T'$ the tree obtained from $T$ by deleting the vertex $v$ together with all of the half-edges incident to it. We have $T'\in\SRT_{g-g(v),n-|I_{h'}|+1,m}$, and we have a natural inclusion $\tH^{em}_+(T')\subset\tH^{em}_+(T)$.

For a function $q\colon\tH^{em}_+(T)\to\mbZ_{\ge 0}$, set $q':=q|_{\tH^{em}_+(T')}$. We can express the cohomology class in a summand on the left-hand side of~\eqref{eq:with and without extra legs} as follows:
$$
\ee_*\xi_{T_q*}\left(\prod_{h\in\tH^{em}_+(T_q)}\psi_h^{q(h)}\right)=\gl_*\left(\ee_*\xi_{T'_{q'}*}\left(\prod_{h\in\tH^{em}_+(T'_{q'})}\psi_h^{q'(h)}\right)\otimes \pi_*\left(\prod_{j\in I_{h'}}\psi_j^{q(\sigma_j)}\right)\right),
$$
where $\gl\colon \oM_{g-g(v),n-|I_{h'}|+1+m}\times\oM_{g(v),|I_{h'}|+1}\to\oM_{g,n+m}$ is the map given by gluing the marked points corresponding to the half-edges $\th$ and $h'$, and $\pi\colon \oM_{g(v),|I_{h'}|+2+q'(h')}\to\oM_{g(v),|I_{h'}|+1}$ is the map forgetting the $q'(h')+1$ marked points corresponding to the extra legs attached to $v\in V(T_q)$. Therefore, the left-hand side of~\eqref{eq:with and without extra legs} is equal to 
\begin{gather*}
\sum_{q'\colon\tH^{em}_+(T')\to\mbZ_{\ge 0}}\gl_*\left(\ee_*\xi_{T'_{q'}*}\left(\prod_{h\in\tH^{em}_+(T'_{q'})}\psi_h^{q'(h)}\right)\prod_{i\in\br{n}\backslash I_{h'}} x_i^{q'(\sigma_i)}\otimes\sum_{l\colon I_{h'}\to\mbZ_{\ge 0}}\pi_*\left(\prod_{j\in I_{h'}}(\psi_j x_j)^{l(j)}\right)\right).
\end{gather*}
From the string equation~\eqref{eq:Str1}, it follows that
$$
\sum_{l\colon I_{h'}\to\mbZ_{\ge 0}}\pi_*\left(\prod_{j\in I_{h'}}(\psi_j x_j)^{l(j)}\right)=x_{I_{h'}}^{q'(h')+1}\sum_{l\colon I_{h'}\to\mbZ_{\ge 0}}\prod_{j\in I_{h'}}(\psi_j x_j)^{l(j)},
$$
which allows us to rewrite the previous expression as
$$
\gl_*\left(\sum_{q'\colon\tH^{em}_+(T')\to\mbZ_{\ge 0}}\ee_*\xi_{T'_{q'}*}\left(\prod_{h\in\tH^{em}_+(T'_{q'})}\psi_h^{q'(h)}\right)x_{I_{h'}}^{q'(h')+1}\prod_{i\in\br{n}\backslash I_{h'}} x_i^{q'(\sigma_i)}\otimes\sum_{l\colon I_{h'}\to\mbZ_{\ge 0}}\prod_{j\in I_{h'}}(\psi_j x_j)^{l(j)}\right).
$$
Applying the induction assumption to the first factor in the tensor product, we obtain exactly the expression on the right-hand side of~\eqref{eq:with and without extra legs}, as required.
\end{proof}

This concludes the proof of Theorem~\ref{theorem:equivalent relations for m2}.\end{proof}

\subsection{Another reformulation for any $\boldsymbol{m\ge 0}$}

In the case $m\le 1$, an analog of Theorem~\ref{theorem:equivalent relations for m2} becomes more subtle. Namely, we no longer can combine conjectural relations and their corollaries as we did in the proof of Theorem~\ref{theorem:equivalent relations for m2} since the conjectural identities given in Conjectures~\ref{conjecture2} and~\ref{conjecture3} for $B^1_{g,\od}$ and $B^0_{g,\od}$, respectively, have nontrivial right-hand sides. However, there exists some simplification of the expression for $B^m_{g,\od}$ along the same lines that holds for any~$m$.

Consider a tree $T\in \SRT_{g,n,m}$ and a function $p\colon \tH^{em}_+(T)\to\ZZ_{\ge 0}$. Let $[T,p]$ denote the class
\begin{align*}
[T,p]\coloneqq \xi_{T*} \left(\prod_{h \in \tH_+^{em}(T)} \psi_h^{p(h)}\right) \in R^{|E(T)|+\sum_{h\in \tH_+^{em}(T)} p(h) }(\oM_{g,n+m}).
\end{align*}

Let us also define two coefficients. One coefficient $C_{\lvl}(T,p)$ is a weighted count of possible level structures and depends only on the structure of the tree $T$ and the function $p$, and another coefficient $C_{\str}(T,p,\od)$ is a combinatorial coefficient that also takes into account a given vector~$\od$ (the subscript ``$\str$'' refers to the fact that it reflects the combinatorics of the string equation). 

A level function $l\colon V(T)\to\mbZ_{\ge 1}$ is called \emph{$p$-admissible} if for every $1\le i< \deg(l)$, we have the following inequality:
\begin{gather*}
\sum_{\substack{h\in \tH_+^{em}(T)\\l(h)\le i}} p(h) + \left|\left\{\left.h \in H_+^e(T)\;\right|\;l(h) < i\right\}\right| \le 2\sum_{\substack{v\in V(T)\\l(v)\leq i}} g(v)-2+m
\end{gather*}
(informally, the left-hand side of this expression is just the degree of the class defined by the stable tree obtained from $T$ by cutting at the level $i$). Let $\mathcal{L}(T,p)$ denote the set of all $p$-admissible level functions on~$T$. Define 
\begin{align*}
C_{\lvl}(T,p)\coloneqq \sum_{l\in\mathcal{L}(T,p)}(-1)^{\deg(l)-1}. 	
\end{align*}

For each $h\in\tH_+^{em}(T)$, let $H_h \subset \tH_+^{em}(T)$ be the subset of all half-edges that are descendants of~$h$ excluding $h$ itself. The combinatorial coefficient $C_{\str}(T,p,\od)$ is set to zero (or, alternatively, just undefined) unless $|E(T)|+\sum_{h\in \tH_+^{em}(T)} p(h)=d_{\llbracket n\rrbracket}$. If that equality holds, then  
\begin{align*}
C_{\str}(T,p,\od) \coloneqq \frac{1}{\prod_{i=1}^n (d_i+1)!} \prod_{h\in \tH_+^{em}(T)} \left(\sum_{i\in I_h}(d_i+1)  - \sum_{h'\in H_h}(p(h')+1)\right)_{(p(h) +1)},
\end{align*}
where we recall that $(a)_{(b)}=a(a-1)\cdots(a-b+1)$ is the Pochhammer symbol. Note that from this definition, $C_{\str}(T,p,\od)$ is equal to zero unless $I_h$ is nonempty for every $h \in \tH_+^{em}(T)$ (or, equivalently, unless each vertex that does not have direct descendants has at least one regular leg attached to it). 

\begin{example} 
Recall Example~\ref{ex:threevertextree}. In this case, the corresponding coefficients are $C_{\lvl}(T,p)=\delta_{p_3+p_4\le 2g_1}\delta_{p_1+p_3+p_4+p_5+1\le 2(g_1+g_2)}$ (here by $\delta$-symbols we denote the functions that take value $1$ once the inequality in the subscript is satisfied and $0$ otherwise). Note that there is at most one admissible level function $l\in \mathcal{L}(T,p)$, which assigns to a vertex of genus $g_i$ the value $i$. The coefficient $C_{\str}(T,p,\od)$ is equal to
\begin{align*}
\frac{(d_2-p_2)_{(p_5+1)}(d_1+d_2-p_1-p_2-p_5-1)_{(p_4+1)}\prod_{i=1}^3(d_i+1)_{(p_i+1)}}{\prod_{i=1}^3 (d_i+1)!}.
\end{align*}
\end{example}

A simplified formula for $B^m_{g,\od}$ is given by the following theorem.

\begin{theorem}\label{theorem:arbitrary m reformulation}
We have
\begin{align*}
B^m_{g,\od} = \sum_{T\in\SRT_{g,n,m}}\sum_{\substack{p\colon \tH_+^{em}(T)\to \ZZ_{\ge 0}\\|E(T)|+\sum_{h\in \tH_+^{em}(T)} p(h)=d_{\llbracket n\rrbracket}}} C_{\lvl}(T,p) C_{\str}(T,p,\od) \cdot [T,p].
\end{align*}
\end{theorem}
\begin{proof} 
The proof follows essentially the same ideas as some steps of the proof of Theorem~\ref{theorem:equivalent relations for m2}. However, we arrange it a bit differently to stress the origin of the coefficients $C_{\lvl}(T,p)$ and $C_{\str}(T,p,\od)$. In a nutshell, we just carefully describe in steps the pushfoward $\ee_*$ in the definition of the class $B^m_{g,\od}$.

Recall Definition~\ref{def:Bgdm-MainDefinition}. Our goal is to explicitly compute $\ee_*[T,\od]$ for any $T\in\SRT^{(b,c,a)}_{g,n,m;\circ}$. We perform the pushforward in two steps. In the first step, we formally apply the string equation given by Equation~\eqref{eq:Str1}  to the vertices of the graph $T$. To this end, in order to efficiently treat the case of potentially unstable vertices, we introduce an extra convention that extends the string equation.
In the string equation~\eqref{eq:Str1}, one has to assume  $2g-2+k>0$ and $p_i\geq 0$, $i=1,\ldots,k$. It is convenient to formally extend the range of applications of the string equation. We consider the case $g=0$, $k=2$, and $q\geq 0$, and we formally set
\begin{gather}\label{eq:Str2}
	\pi_*(\psi_1^q|_{\oM_{0,q+3}})=\psi_1^{-1}|_{\oM_{0,2}}
\end{gather}
for the map $\pi\colon\oM_{0,q+3}\to\oM_{0,2}$ that forgets the last $q+1$ marked points. This map has to be understood formally: we are going to apply $\ee_*$ to the graphs in $\SRT^{(b,c,a)}_{g,n,m;\circ}$, and Equation~\eqref{eq:Str2} just means that at an intermediate step of the computation, we will use unstable vertices of genus $0$ with two incident half-edges (these vertices will disappear in the final formula for the pushforward). To make this precise, let us introduce auxiliary definitions. 

We consider \emph{rooted trees} in $\RT_{g,n,m}$, which is an extension of $\SRT_{g,n,m}$ where we allow trees to have unstable vertices of type $(0,2)$; that is, we allow vertices $v$ of genus $g(v)=0$ with just two incident half-edges, one in $\tH_+^{em}(T)$ and one in the direction of the mother of $v$. The root cannot be unstable, however. Denote by $\DRT_{g,n,m}$ the set of pairs $(T,p)$, where $T\in\RT_{g,n,m}$ and $p\colon \tH_+^{em}(T) \to \ZZ_{\ge -1}$ is a function such that $p^{-1}(-1)$ is exactly the subset of $\tH_+^{em}(T)$ attached to unstable vertices. Let us call such pairs \emph{decorated rooted trees}.

In the same way as for stable rooted trees, for $T\in\RT_{g,n,m}$, we consider level functions $l\colon V(T)\to \ZZ_{\ge 1}$ and the canonical level function $l_T\colon V(T)\to\mbZ_{\ge 1}$. We say that a rooted tree in~$\RT_{g,n,m}$ is \emph{complete} if the canonical level function satisfies exactly the same conditions as the canonical level function of a complete balanced stable rooted tree in $\SRT_{g,n,m;\circ}^{(b,c)}$ (\textit{cf.}~Section~\ref{sec:srt-definitions}) with the condition ``\emph{for every $1\le l\le\deg(T)$, the set of vertices $l_T^{-1}(l)$ contains at least one vertex that is not potentially unstable}'' replaced by ``\emph{for every $1\le l\le\deg(T)$, the set of vertices $l_T^{-1}(l)$ contains at least one vertex that is not unstable}.'' Let $\RT^{(c)}_{g,n,m}$ denote the subset of complete rooted trees. Denote by $\DRT^{(c)}_{g,n,m}\subset\DRT_{g,n,m}$ the subset of pairs $(T,p)$ where $T\in\RT^{(c)}_{g,n,m}$. Let $\DRT^{(c,a)}_{g,n,m}\subset \DRT^{(c)}_{g,n,m}$ be the set of so-called \emph{admissible} complete decorated rooted trees $(T,p)$ satisfying the additional system of inequalities
\begin{align}\label{eq:RT-admissible}
\sum_{\substack{h\in H^e_+(T)\\ l_T(h)\le k}} p(h) + \left|\left\{\left.h\in H^e_+(T)\;\right|\; l_T(h)<k\right\}\right| \le 2 g_k(T)-2+m,\quad\text{for any } 1\le k<\deg(T).
\end{align}

We can now proceed to the description of the first step in our computation of the pushforward $\ee_*[T,\od]$, $T\in\SRT^{(b,c,a)}_{g,n,m;\circ}$. We construct a map $f_1\colon \SRT_{g,n,m;\circ}^{(b,c,a)}\times\mbZ_{\ge 0}^n \to \QQ\left<\DRT^{(c)}_{g,n,m}\right>$ as follows.
\begin{itemize}
\item For a pair $(T,\od)\in\SRT_{g,n,m;\circ}^{(b,c,a)}\times\mbZ_{\ge 0}^n$, consider a potentially unstable vertex $v$. There is exactly one half-edge in $\tH_+^{em}(T)$ attached to $v$; denote it by $h'$, and denote by $\th$ the half-edge attached to $v$ that is directed to the mother of $v$. If $q(h')=q(\iota(\th))$, then we replace $v$ with an unstable genus $0$ vertex that retains just the two half-edges $h'$ and $\th$ (forgetting all of the extra legs attached to $v$), and we set $p(h'):=-1$. We do this for each potentially unstable vertex $v$. If there is at least one potentially unstable vertex $v$ with $q(h')\ne q(\iota(\th))$, then we set $f_1(T,\od):=0$ (note that if $q(h')\ne q(\iota(\th))$, then $\ee_*[T,\od]=0$).

\item Consider a not potentially unstable vertex $v$ of $T$ that also does not coincide with the root. Let $h_1,\ldots,h_k$ be the half-edges in~$\tH_+^{em}(T)$ attached to $v$, and denote by $\th$ the half-edge attached to $v$ that is directed to the mother of $v$. We replace $v$ with a vertex of the same genus that retains just the half-edges $h_1,\ldots,h_k,\th$ and take the weighted sum over all choices of $p(h_1),\ldots,p(h_k)$ such that $\sum q(h_i)-\sum p(h_i) = q(\iota(\th))+1$ with the weights equal to the coefficients on the right-hand side of Equation~\eqref{eq:Str1}. We do this for each vertex $v$ that is not potentially unstable and that is not the root.

\item For the root of $T$, we just define $p(h):=q(h)$ for any half-edge $h\in\tH^{em}_+(T)$ attached to it.
\end{itemize}
Clearly, as the outcome of the above procedure, we obtain a linear combination of decorated rooted trees from $\DRT^{(c)}_{g,n,m}$. We take this linear combination as the value of the function $f_1$ on the pair $(T,\od)$.

We claim that
\begin{gather} \label{eq:Formula-f1}
f_1(T,\od)=\sum_{\substack{(T',p)\in\DRT^{(c)}_{g,n,m}\\(T',p)\sim(T,\od)}} \frac{\prod_{h\in \tH_+^{em}(T)} \left(\sum_{i\in I_h}(d_i+1)-\sum_{h'\in H_h}(p(h')+1)\right)_{\left(p(h) +1\right)}}{\prod_{i=1}^n (d_i+1)!}(T',p),
\end{gather}
where the notation $(T',p) \sim (T,\od)$ means that $(T',p)$ is obtained from $(T,\od)$ by removing all extra legs and making a choice of the values of the function $p$ according to the construction of the map $f_1$. Indeed, the coefficient of $(T',p)$ on the right-hand side of~\eqref{eq:Formula-f1} is equal to the product of coefficients prescribed by choices of summands made according to Equation~\eqref{eq:Str1}: we just rewrite the product of these coefficients in terms of the vector $\od$ and the function $p$ on~$\tH_+^{em}(T')$.

Moreover, several remarks are in order:
\begin{itemize}
\item It follows from the string equations that Equation~\eqref{eq:admissibility-SRT} for $T$ implies Equation~\eqref{eq:RT-admissible} for every $(T',p)\sim (T,\od)$; thus $(T',p)$ belongs to $\DRT^{(c,a)}_{g,n,m}$. Therefore, the image of $f_1$ belongs to $\mbQ\left<\DRT^{(c,a)}_{g,n,m}\right>$.

\item It also follows from the string equations that 
\begin{gather}\label{eq:pTprime-condition}
\sum_{\substack{h\in \tH^{em}_+(T')}} p(h) + |E(T')|  = d_{\llbracket n \rrbracket}.
\end{gather} 

\item The function $q$ on $\tH_+^{em}(T)$ is uniquely reconstructed from $\od$ and the function $p$ on $\tH_+^{em}(T')$, but, given arbitrary $(T',p)\in\DRT^{(c,a)}_{g,n,m}$ and $\od\in\mbZ_{\ge 0}^n$, such a function $q$ may not exist. Note however that if  condition~\eqref{eq:pTprime-condition} is satisfied, then such a function $q$ exists if and only if the coefficient of $(T',p)$ on the right-hand side of~\eqref{eq:Formula-f1} is nonzero.  

\item The vertices of $T$ and $T'$ are in a natural bijection, and the canonical level functions for both trees are identified by this bijection. In particular, $\deg(T')=\deg(T)$.  
\end{itemize}

In the second step, we contract all unstable vertices of a decorated rooted tree $(T',p)\in \DRT^{(c,a)}_{g,n,m}$. This defines a map $f_2\colon\DRT^{(c,a)}_{g,n,m}\to\SRT_{g,n,m}$, $f_2\colon(T',p)\mapsto T''$, where the functions~$p$ and~$l_{T'}$ defined for $T'$ naturally descend to functions $p\colon \tH_+^{em} (T'')\to \ZZ_{\ge 0}$ and $l\colon V(T'') \to \ZZ_{\ge 1}$ (all values of $p$ are now nonnegative since there are no unstable vertices left). The resulting level function $l\colon V(T'') \to \ZZ_{\ge 1}$ is obviously $p$-admissible. Also, given $\od\in\mbZ_{\ge 0}^n$,  condition~\eqref{eq:pTprime-condition} is equivalent to the condition $\sum_{\substack{h\in \tH^{em}_+(T'')}} p(h) + |E(T'')|  = d_{\llbracket n \rrbracket}$. Conversely, given $T''\in\SRT_{g,n,m}$, $p\colon V(T'')\to\mbZ_{\ge 0}$, and $l\in\mathcal{L}(T'',p)$, there is a unique $(T',p)\in\DRT^{(c,a)}_{g,n,m}$ such that $f_2(T',p) = T''$ and $p$ and $l$ on $T''$ are induced by the corresponding functions $p$ and $l_{T'}$ on~$T'$ (with $\deg(T') = \deg(l)$). Finally, note that, also given  $\od\in\mbZ_{\ge 0}^n$, the coefficient of $(T',p)$ in Equation~\eqref{eq:Formula-f1} is equal to $C_{\str}(T'',p,\od)$.

This implies that for any $(T,\od)\in\SRT_{g,n,m;\circ}^{(b,c,a)}\times\mbZ_{\ge 0}^n$, the class $e_*[T,\od]$ is equal to 
\begin{align*}
\ee_*[T,\od] = \sum_{\substack{(T',p)\in\DRT^{(c,a)}_{g,n,m}\\(T',p) \sim (T,\od)}}C_{\str}\left(f_2(T',p),p,\od\right)\left[f_2(T',p),p\right],
\end{align*}
and therefore
\begin{align*}
B^m_{g,\od} & = \sum_{T\in\SRT^{(b,c,a)}_{g,n,m;\circ}}(-1)^{\deg(T)-1}\ee_*[T,\od]\\
& = \sum_{T\in\SRT^{(b,c,a)}_{g,n,m;\circ}}  \sum_{\substack{(T',p)\in\DRT^{(c,a)}_{g,n,m}\\(T',p) \sim (T,\od)}}(-1)^{\deg(l_{T'})-1}C_{\str}\left(f_2(T',p),p,\od\right)\left[f_2(T',p),p\right]\\
&=\sum_{\substack{(T',p)\in\DRT^{(c,a)}_{g,n,m}\\\sum_{\substack{h\in \tH^{em}_+(T')}} p(h) + |E(T')|  = d_{\llbracket n \rrbracket}}}(-1)^{\deg(l_{T'})-1}C_{\str}\left(f_2(T',p),p,\od\right)\left[f_2(T',p),p\right]\\
&=\sum_{T''\in\SRT_{g,n,m}}\sum_{\substack{p\colon\tH^{em}_+(T'')\to\mbZ_{\ge 0}\\\sum_{\substack{h\in \tH^{em}_+(T'')}} p(h) + |E(T'')|  = d_{\llbracket n \rrbracket}}}C_{\str}\left(T'',p,\od\right)\left[T'',p\right]\cdot\sum_{l\in\mathcal{L}(T'',p)}(-1)^{\deg(l)-1},
\end{align*}
which proves the theorem.
\end{proof}

%%%%%%%%%%%%%%%%%%%%%%%%%%%%%%%%%%%%%%%%%%%%%%%%%%%%%%%%%%%%%
%%%%%%%%%%%%%%%%%%%%%%%%%%%%%%%%%%%%%%%%%%%%%%%%%%%%%%%%%%%%%

\section{Conjectural relations and the fundamental properties of the DR and the DZ hierarchies}\label{section:conjectural relations and fundamental properties}

In this section, we show how certain fundamental properties of the Dubrovin--Zhang (DZ) and the double ramification (DR) hierarchies associated to F-cohomological field theories (F-CohFTs) on $\oM_{g,n}$ naturally follow from our Conjectures~\ref{conjecture1} and~\ref{conjecture2}. This extends the results from~\cite{BGR19}, where the authors showed that Conjecture~\ref{conjecture3} naturally implies that the DZ and the DR hierarchies associated to an arbitrary CohFT are Miura equivalent. Actually, in all of the aspects of the theory of integrable systems associated to CohFTs that we consider in this paper, a CohFT can be replaced by a slightly more general object called a \emph{partial CohFT}, introduced in~\cite{LRZ15}. In the theory of the DR hierarchies, it was first noticed in~\cite[Section~9.1]{BDGR18}. So we will systematically work with partial CohFTs instead of CohFTs.

We also discuss the role of the classes $B^m_{g,\od}$ with $\sum d_i=2g-2+m$, which are not involved in the conjectural relations from Conjectures~\ref{conjecture1},~\ref{conjecture2}, and~\ref{conjecture3}. It occurs that these classes control the polynomial parts of the equations (which are conjecturally polynomial) of the DZ hierarchies and the Miura transformation that conjecturally relates the DZ and the DR hierarchies. 

\subsection{Differential polynomials and evolutionary PDEs}

Let us fix $N\ge 1$ and consider formal variables $w^1,\ldots,w^N$. Let us briefly recall main notions and notation in the formal theory of evolutionary PDEs with one spatial variable (and refer a reader, for example, to~\cite{BRS21} for details):
\begin{itemize}
\item To the formal variables $w^\alpha$, we attach formal variables $w^\alpha_d$ with $d\ge 0$, and we introduce the ring of \emph{differential polynomials} $\cA_w:=\mbC[[w^*]][w^*_{\ge 1}]$ (in~\cite{BRS21}, it is denoted by~$\cA^0_w$). We identify $w^\alpha_0=w^\alpha$ and also set $w^\alpha_x:=w^\alpha_1$, $w^{\alpha}_{xx}:=w^\alpha_2$, \ldots.

\item The operator $\d_x\colon\cA_w\to\cA_w$ is defined by $\d_x:=\sum_{d\ge 0}w^\alpha_{d+1}\frac{\d}{\d w^\alpha_d}$.

\item Let $\cA_{w;d}\subset\cA_w$ be the homogeneous component of (differential) degree $d$, where $\deg w^\alpha_i:=i$. For $f\in\cA_w$, we denote by $f^{[d]}\in\cA_{w;d}$ the image of $f$ under the canonical projection $\cA_w\to\cA_{w;d}$. We will also use the notation $\cA_{w;\le d}:=\bigoplus_{i=0}^d\cA_{w;i}$.

\item The extended space of differential polynomials is defined by $\hcA_w:= \cA_w[[\eps]]$. Let $\hcA_{w;k}\subset\hcA_w$ be the homogeneous component of degree~$k$, where $\deg\eps:=-1$.  

\item A \emph{Miura transformation} (that is close to the identity) is a change of variables $w^\alpha\mapsto \tw^\alpha(w^*_*,\eps)$ of the form $\tw^\alpha(w^*_*,\eps)=w^\alpha+\eps f^\alpha(w^*_*,\eps)$, where
$f^\alpha\in\hcA_{w;1}$. 

\item A system of \emph{evolutionary PDEs} (with one spatial variable) is a system of equations of the form $\frac{\d w^\alpha}{\d t}=P^\alpha$, $1\le\alpha\le N$, where $P^\alpha\in\hcA_w$. Two systems $\frac{\d w^\alpha}{\d t}=P^\alpha$ and $\frac{\d w^\alpha}{\d s}=Q^\alpha$ are said to be \emph{compatible} (or, equivalently, the flows $\frac{\d}{\d t}$ and $\frac{\d}{\d s}$ are said to \emph{commute}) if $\sum_{n\ge 0}\left(\frac{\d P^\alpha}{\d w^\beta_n}\d_x^n Q^\beta-\frac{\d Q^\alpha}{\d w^\beta_n}\d_x^n P^\beta\right)=0$ for any $1\le\alpha\le N$. 
\end{itemize}

\subsection{F-CohFTs and partial CohFTs}

\subsubsection{Definitions}

\begin{definition}[\textit{cf.} \cite{BR21}]
An \emph{F-cohomological field theory} (F-CohFT) is a system of linear maps 
$$
c_{g,n+1}\colon V^*\otimes V^{\otimes n} \lra H^\even(\oM_{g,n+1}),\quad g,n\ge 0,\,\, 2g-1+n>0,
$$
where $V$ is an arbitrary finite-dimensional vector space, together with a special element $e\in V$, called the \emph{unit}, such that, choosing a basis $e_1,\ldots,e_{\dim V}$ of $V$ and the dual basis $e^1,\ldots,e^{\dim V}$ of~$V^*$, the following axioms are satisfied:
\begin{itemize}
\item[(i)] The maps $c_{g,n+1}$ are equivariant with respect to the $S_n$-action permuting the $n$ copies of~$V$ in $V^*\otimes V^{\otimes n}$ and the last $n$ marked points in $\oM_{g,n+1}$, respectively.

\item[(ii)] We have $\pi^* c_{g,n+1}(e^{\alpha_0}\otimes \otimes_{i=1}^n e_{\alpha_i}) = c_{g,n+2}(e^{\alpha_0}\otimes \otimes_{i=1}^n  e_{\alpha_i}\otimes e)$ for $1 \leq\alpha_0,\alpha_1,\ldots,\alpha_n\leq \dim V$, where $\pi\colon\oM_{g,n+2}\to\oM_{g,n+1}$ is the map that forgets the last marked point. Moreover, $c_{0,3}(e^{\alpha}\otimes e_\beta \otimes e) = \delta^\alpha_\beta$ for $1\leq \alpha,\beta\leq \dim V$.

\item[(iii)] We have $\gl^* c_{g_1+g_2,n_1+n_2+1}(e^{\alpha_0}\otimes \otimes_{i=1}^{n_1+n_2} e_{\alpha_i}) = c_{g_1,n_1+2}(e^{\alpha_0}\otimes \otimes_{i\in I} e_{\alpha_i} \otimes e_\mu)\otimes c_{g_2,n_2+1}(e^{\mu}\otimes \otimes_{j\in J} e_{\alpha_j})$ for $1 \leq\alpha_0,\alpha_1,\ldots,\alpha_{n_1+n_2}\leq \dim V$, where $I \sqcup J = \llbracket n_1+n_2+1\rrbracket\backslash\{1\}$, $|I|=n_1$, $|J|=n_2$, and $\gl\colon\oM_{g_1,n_1+2}\times\oM_{g_2,n_2+1}\to \oM_{g_1+g_2,n_1+n_2+1}$ is the corresponding gluing map.

The gluing map $\gl$ creates a nodal curve sewing the last marked point on a curve in $\oM_{g_1,n_1+2}$ (the point labeled by $n_1+2$) and the first marked point on a curve in $\oM_{g_2,n_2+1}$ (the point labeled by $1$) into a node. Under the gluing map $\gl$, the first marked point of a curve in $\oM_{g_1,n_1+2}$ becomes the first marked point on the resulting nodal curve in $\oM_{g_1+g_2,n_1+n_2+1}$, and the other marked points are relabeled according to the identification $I \sqcup J = \llbracket n_1+n_2+1\rrbracket\backslash\{1\}$.
\end{itemize}
It is easy to see that the validity of the above properties does not depend on the choice of a basis of $V$. 
\end{definition}

\begin{definition}\cite{LRZ15}
A \emph{partial CohFT} is a system of linear maps 
$$
c_{g,n}\colon V^{\otimes n} \lra H^\even(\oM_{g,n}),\quad 2g-2+n>0,
$$
where $V$ is an arbitrary finite-dimensional vector space, together with a special element $e\in V$, called the \emph{unit}, and a symmetric nondegenerate bilinear form $\eta\in (V^*)^{\otimes 2}$, called the \emph{metric}, such that, choosing a basis $e_1,\ldots,e_{\dim V}$ of $V$, the following axioms are satisfied:
\begin{itemize}
\item[(i)] The maps $c_{g,n}$ are equivariant with respect to the $S_n$-action permuting the $n$ copies of~$V$ in $V^{\otimes n}$ and the $n$ marked points in $\oM_{g,n}$, respectively.

\item[(ii)] We have $\pi^* c_{g,n}( \otimes_{i=1}^n e_{\alpha_i}) = c_{g,n+1}(\otimes_{i=1}^n  e_{\alpha_i}\otimes e)$ for $1 \leq\alpha_1,\ldots,\alpha_n\leq \dim V$, where $\pi\colon\oM_{g,n+1}\to\oM_{g,n}$ is the map that forgets the last marked point. Moreover, $c_{0,3}(e_{\alpha}\otimes e_\beta \otimes e) =\eta(e_\alpha\otimes e_\beta) =:\eta_{\alpha\beta}$ for $1\leq \alpha,\beta\leq \dim V$.

\item[(iii)] We have $\gl^* c_{g_1+g_2,n_1+n_2}( \otimes_{i=1}^{n_1+n_2} e_{\alpha_i}) = \eta^{\mu \nu}c_{g_1,n_1+1}(\otimes_{i\in I} e_{\alpha_i} \otimes e_\mu)\otimes c_{g_2,n_2+1}(\otimes_{j\in J} e_{\alpha_j}\otimes e_\nu)$ for $1\leq\alpha_1,\ldots,\alpha_{n_1+n_2}\leq \dim V$, where $I \sqcup J = \llbracket n_1+n_2\rrbracket$, $|I|=n_1$, $|J|=n_2$, and $\gl\colon\oM_{g_1,n_1+1}\times\oM_{g_2,n_2+1}\to \oM_{g_1+g_2,n_1+n_2}$ is the corresponding gluing map and where~$\eta^{\alpha\beta}$ is defined by $\eta^{\alpha \mu}\eta_{\mu \beta} = \delta^\alpha_\beta$ for $1\leq \alpha,\beta\leq \dim V$.
\end{itemize}
\end{definition}

Clearly, given a partial CohFT $\{c_{g,n}\colon V^{\otimes n}\to H^{\even}(\oM_{g,n})\}$, the maps $\tc_{g,n+1}\colon V^*\otimes V^{\otimes n}\to H^{\even}(\oM_{g,n+1})$ defined by $\tc_{g,n+1}(e^{\alpha_0}\otimes\otimes_{i=1}^n e_{\alpha_i}):=\eta^{\alpha_0\mu} c_{g,n+1}(e_\mu\otimes\otimes_{i=1}^n e_{\alpha_i})$ form an F-CohFT.

\subsubsection{Various potentials associated to partial CohFTs and F-CohFTs}

First consider an arbitrary partial CohFT $\{c_{g,n}\colon V^{\otimes n}\to H^{\even}(\oM_{g,n})\}$ with $\dim V=N$, metric $\eta\colon V^{\otimes 2}\to \mbC$, and unit $e\in V$. We fix a basis $e_1,\ldots,e_N\in V$ and define the \emph{potential} of our partial CohFT by
\begin{gather*}
\mcF:=\sum\frac{\eps^{2g}}{n!}\left(\int_{\oM_{g,n}}c_{g,n}(\otimes_{i=1}^n e_{\alpha_i})\prod_{i=1}^n\psi_i^{d_i}\right)\prod_{i=1}^n t^{\alpha_i}_{d_i}\in\mbC[[t^*_*,\eps]].
\end{gather*}
The potential satisfies the \emph{string} and the \emph{dilaton} equations:
\begin{align}
&\frac{\d\mcF}{\d t^\un_0}=\sum_{n\ge 0}t^\alpha_{n+1}\frac{\d\mcF}{\d t^\alpha_n}+\frac{1}{2}\eta_{\alpha\beta}t^\alpha_0 t^\beta_0+\eps^2\int_{\oM_{1,1}}c_{1,1}(e),\label{eq:string equation}\\
&\frac{\d\mcF}{\d t^\un_1}=\sum_{n\ge 0}t^\alpha_n\frac{\d\mcF}{\d t^\alpha_n}+\eps\frac{\d\mcF}{\d\eps}-2\mcF+\eps^2\int_{\oM_{1,1}}\psi_1 c_{1,1}(e),\label{eq:dilaton equation}
\end{align}
where $\frac{\d}{\d t^\un_0}:=A^\mu\frac{\d}{\d t^\mu_0}$ and the coefficients $A^\mu$ are given by $e=A^\mu e_\mu$. Let us also define formal power series $w^{\top;\alpha}:=\eta^{\alpha\mu}\frac{\d^2\mcF}{\d t^\mu_0\d t^\un_0}$ and $w^{\top;\alpha}_n:=\frac{\d^n w^{\top;\alpha}}{(\d t^\un_0)^n}$.

For $d\ge 0$, denote by $\mbC[[t^*_*]]^{(d)}$ the subset of $\mbC[[t^*_*]]$ formed by infinite linear combinations of monomials $\prod t^{\alpha_i}_{d_i}$ with $\sum d_i\ge d$. Clearly, $\mbC[[t^*_*]]^{(d)}\subset \mbC[[t^*_*]]$ is an ideal. From the string equation~\eqref{eq:string equation}, it follows that
\begin{gather}\label{eq:main property of wtop}
w^{\top;\alpha}_n=t^\alpha_n+\delta_{n,1}A^\alpha+R^\alpha_n(t^*_*)+O(\eps^2)\quad \text{for some $R^\alpha_n\in\mbC[[t^*_*]]^{(n+1)}$}.
\end{gather}

The following obvious statement will be very useful, so we would like to present it as a separate lemma.

\begin{lemma}\label{lemma:wt-change}
Suppose that a family of formal power series $\tw^\alpha_n\in\mbC[[t^*_*,\eps]]$, $1\le\alpha\le N$, $n\ge 0$, satisfies the property $\tw^\alpha_n=t^\alpha_n+\delta_{n,1}A^\alpha+R^\alpha_n(t^*_*)+O(\eps^2)$ for some $R^\alpha_n\in\mbC[[t^*_*]]^{(n+1)}$. Then any formal power series in the variables $t^\alpha_a$, $1\le\alpha\le N$, $a\ge 0$, and $\eps$ can be expressed as a formal power series in $\left(\widetilde{w}^\beta_b-\delta_{b,1}A^\beta\right)$ and $\eps$ in a unique way. In particular, for any two differential polynomials $P,Q\in\hcA_w$, the equality $P|_{w^\alpha_a=\tw^\alpha_a}=Q|_{w^\alpha_a=\tw^\alpha_a}$ implies that $P=Q$.
\end{lemma}

In~\cite[Proposition~7.2]{BDGR18}, the authors proved that there exists a unique differential polynomial $\cP\in\hcA_{w;-2}$ such that the difference
$$
\mcF^{\red}:=\mcF-\left.\cP\right|_{w^\gamma_c=w^{\top;\gamma}_c}
$$
satisfies the condition
$$
\Coef_{\eps^{2g}}\left.\frac{\d^n\mcF^{\red}}{\d t^{\alpha_1}_{d_1}\cdots\d t^{\alpha_n}_{d_n}}\right|_{t^*_*=0}=0\quad\text{if } \sum d_i\le 2g-2
$$
(formally, this was proved for CohFTs, but the proof works for partial CohFTs as well). The formal power series $\mcF^{\red}$ is called the \emph{reduced potential} of the partial CohFT. Consider the expansions $\mcF^\red=\sum_{g\ge 0}\eps^{2g}\mcF^\red_g$ and $\cP=\sum_{g\ge 1}\eps^{2g}\cP_g$. Note that $\mcF^\red_0=\mcF_0$. In \cite[Proposition~3.5]{BGR19}, the authors proved that 
\begin{gather}\label{eq:reduced correlators}
\left.\frac{\d^n \mcF^{\red}_g}{\d t^{\alpha_1}_{d_1}\cdots\d t^{\alpha_n}_{d_n}}\right|_{t^*_*=0}=\int_{\oM_{g,n}}B^0_{g,\od}c_{g,n}\left(\otimes_{i=1}^n e_{\alpha_i}\right),\quad
\begin{minipage}{7cm}
$n\ge 1,\, d_1,\ldots,d_n\ge 0,\, \sum d_i\ge 2g-1$,\\
$1\le\alpha_1,\ldots,\alpha_n\le N$.
\end{minipage}
\end{gather}
Also note that $\left.\mcF^\red_g\right|_{t^*_*=0}=0$.

We can now present an explicit formula for the differential polynomial~$\cP$.

\begin{theorem}\label{theorem:geometric formula for reducing Miura transformation}
For $g\ge 1$, we have $\left.\cP_g\right|_{w^*_*=0}=0$ and
\begin{gather}\label{eq:formula for Pg}
\left.\frac{\d^n\cP_g}{\d w^{\alpha_1}_{d_1}\cdots\d w^{\alpha_n}_{d_n}}\right|_{w^*_*=0}=\int_{\oM_{g,n}}B^0_{g,\od}c_{g,n}(\otimes_{i=1}^n e_{\alpha_i}),\quad 
\begin{minipage}{6cm}
$n\ge 1$,\\
$d_1,\ldots,d_n\ge 0,\, \sum d_i=2g-2$,\\
$1\le\alpha_1,\ldots,\alpha_n\le N$.
\end{minipage}
\end{gather}
\end{theorem}
\begin{proof}
Let us first recall the construction of the differential polynomial $\cP$, following~\cite[proof of Proposition~3.5]{BGR19}.

The reduced potential $\mcF^\red$ is constructed by a recursive procedure that kills all of the monomials $\eps^{2g}\prod t^{\alpha_i}_{d_i}$ with $g\ge 1$ and $\sum d_i\le 2g-2$ in the potential~$\mcF$. Let us assign to such a monomial a level $l:=(g-1)^2+\sum d_i$. We see that $(g-1)^2\le l<g^2$, which implies that $g$ is uniquely determined by $l$. Denote this $g$ by $g(l)$. Let us define a sequence of formal power series $\mcF^{(-1)}:=\mcF,\mcF^{(0)},\mcF^{(1)},\ldots$ by
\begin{gather}\label{eq:formula for Fl}
\mcF^{(l)}:=\mcF^{(l-1)}-\sum_{n\ge 0}\eps^{2g(l)}\hspace{-0.2cm}\underline{\sum_{\substack{d_1,\ldots,d_n\ge 0\\\sum d_i=l-(g(l)-1)^2}}\hspace{-0.1cm}\left(\Coef_{\eps^{2g(l)}}\left.\frac{\d^n\mcF^{(l-1)}}{\d t^{\alpha_1}_{d_1}\cdots\d t^{\alpha_n}_{d_n}}\right|_{t^*_*=0}\right)\frac{\prod (w^{\alpha_i}_{d_i}-\delta_{d_i,1}A^{\alpha_i})}{n!}},\quad l\ge 0.
\end{gather}
From~\eqref{eq:main property of wtop} it follows that going from $\mcF^{(l-1)}$ to $\mcF^{(l)}$, we kill all of the monomials of level $l$ and do not touch the monomials of level strictly less than $l$. Then $\mcF^{\red}$ is equal to the limit $\mcF^\red:=\lim_{l\to\infty}\mcF^{(l)}$, and the underlined terms (multiplied by $\eps^{2g(l)}$), for all $l\ge 0$, produce the required differential polynomial $\cP$. 

In order to show that $\cP\in\hcA_{w;-2}$, we have to check that $\cP_g^{[k]}=0$ for $g\ge 2$ and $k<2g-2$. Suppose that this is not true, and denote by $g_0$ the minimal $g$ such that $\cP^{[k]}_g\ne 0$ for some $k<2g-2$. Consider the operator 
\begin{gather}\label{eq:Loperator}
L:=\frac{\d}{\d t^\un_1}-\sum_{n\ge 0}t^\alpha_n\frac{\d}{\d t^\alpha_n}-\eps\frac{\d}{\d\eps}.\end{gather}
The dilaton equation~\eqref{eq:dilaton equation} implies that $L w^{\top;\alpha}_n=n w^{\top;\alpha}_n$. Therefore, we have the following sequence of equalities:
$$
\eps^2\int_{\oM_{1,1}}\psi_1 c_{1,1}(e)=(L+2)\mcF=(L+2)\mcF^\red+\underline{\sum_{g\ge 2}\eps^{2g}\sum_{k<2g-2}(k+2-2g)\left.\cP^{[k]}_g\right|_{w^\gamma_n=w^{\top;\gamma}_c}}.
$$
We see that the coefficient $\eps^{2g_0}$ in $(L+2)\mcF$ is zero. However, the coefficient of $\eps^{2g_0}$ in $(L+2)\mcF^\red$ belongs to $\mbC[[t^*_*]]^{(2g_0-2)}$, while the coefficient of $\eps^{2g_0}$ in the underlined sum does not belong to $\mbC[[t^*_*]]^{(2g_0-2)}$, which follows from~\eqref{eq:main property of wtop}. This contradiction proves that $\cP_g\in\cA_{w;2g-2}$, and we also see that $\left.\cP_g\right|_{w^*_*=0}=0$.

In order to prove Equation~\eqref{eq:formula for Pg}, note that the underlined terms in~\eqref{eq:formula for Fl} give a differential polynomial from $\cA_{w;\le l-(g(l)-1)^2}$. Using that $\cP_g\in\cA_{w;2g-2}$, we therefore obtain
$$
\left.\frac{\d^n\cP_g}{\d w^{\alpha_1}_{d_1}\cdots\d w^{\alpha_n}_{d_n}}\right|_{w^*_*=0}=\Coef_{\eps^{2g}}\left.\frac{\d^n\mcF^{(g^2-2)}}{\d t^{\alpha_1}_{d_1}\cdots\d t^{\alpha_n}_{d_n}}\right|_{t^*_*=0},\quad g,n\ge 1,\,\sum d_i=2g-2.
$$

Following~\cite[proof of Proposition~3.5]{BGR19}, we call a tree $T\in\SRT^{(b,c)}_{g,n,0;\circ}$ \emph{$(j,d)$-admissible} if for any $1\le k<\deg(T)$, we have $g_k(T)\le j$ and 
$$
\sum_{\substack{h\in H^e_+(T)\\l(h)=k}}q(h)\le\begin{cases}
2g_k(T)-2&\text{if $g_k(T)<j$},\\
d&\text{if $g_k(T)=j$}.
\end{cases} 
$$
Denote by $\SRT^{(b,c);(j,d)}_{g,n,0;\circ}\subset \SRT^{(b,c)}_{g,n,0;\circ}$ the set of such trees. By~\cite[Equation~(19)]{BGR19}, we have
$$
\Coef_{\eps^{2g}}\left.\frac{\d^n\mcF^{(g^2-2)}}{\d t^{\alpha_1}_{d_1}\cdots\d t^{\alpha_n}_{d_n}}\right|_{t^*_*=0}=\int_{\oM_{g,n}}\left(\sum_{T\in\SRT^{(b,c);(g,2g-3)}_{g,n,0;\circ}}(-1)^{\deg(T)-1}\ee_*[T,\od]\right)c_{g,n}(\otimes_{i=1}^n e_{\alpha_i})
$$
for $g,n\ge 1$ and $\sum d_i=2g-2$, and it remains to check that
$$
\sum_{T\in\SRT^{(b,c);(g,2g-3)}_{g,n,0;\circ}}(-1)^{\deg(T)-1}\ee_*[T,\od]=B^0_{g,\od}.
$$
Clearly, $\SRT^{(b,c);(g,2g-3)}_{g,n,0;\circ}\subset\SRT^{(b,c,a)}_{g,n,0;\circ}$. The converse inclusion is not necessarily true; however, $\left\{\left.T\in\SRT^{(b,c,a)}_{g,n,0;\circ}\;\right|\;\ee_*[T,\od]\ne 0\right\}\subset\SRT^{(b,c);(g,2g-3)}_{g,n,0;\circ}$ thanks to the following statement proved in~\cite{BGR19}.

\begin{lemma}[\textit{cf.}~\protect{\cite[Lemma~3.6]{BGR19}}]
Let $g\ge 0$, $n\ge 1$, $d_1,\ldots,d_n\ge 0$, and $T\in\SRT^{(b,c)}_{g,n,0;\circ}$. Suppose $\ee_*[T,\od]\ne 0$ and $g_k(T)=g_{k+1}(T)$ for some $1\le k<\deg(T)$. Then 
$$
\sum_{\substack{h\in\tH^{em}_+(T)\\l(h)=k+1}}q(h)>\sum_{\substack{h\in\tH^{em}_+(T)\\l(h)=k}}q(h).
$$
\end{lemma}

This concludes the proof of Theorem~\ref{theorem:geometric formula for reducing Miura transformation}.
\end{proof}

Now consider  an arbitrary F-CohFT $\{c_{g,n+1}\colon V^*\otimes V^{\otimes n}\to H^{\even}(\oM_{g,n+1})\}$ with $\dim V=N$ and unit $e\in V$. We fix a basis $e_1,\ldots,e_N\in V$ and assign to the F-CohFT a collection of potentials $\mcF^{\alpha,a}$, $1\le\alpha\le N$, $a\ge 0$, by setting 
\begin{gather*}
\mcF^{\alpha,a}:=\sum\frac{\eps^{2g}}{n!}\left(\int_{\oM_{g,n+1}}c_{g,n+1}(e^\alpha\otimes\otimes_{i=1}^n e_{\alpha_i})\psi_1^a\prod_{i=1}^n\psi_{i+1}^{d_i}\right)\prod_{i=1}^n t^{\alpha_i}_{d_i}\in\mbC[[t^*_*,\eps]]. 
\end{gather*}
These potentials satisfy the following system of equations, which can be considered as an analog of the string equation~\eqref{eq:string equation}:
\begin{gather}\label{eq:string for F-CohFT}
\frac{\d\mcF^{\alpha,a}}{\d t^\un_0}=\mcF^{\alpha,a-1}+\sum_{n\ge 0}t^\beta_{n+1}\frac{\d\mcF^{\alpha,a}}{\d t^\beta_n},\quad 1\le\alpha\le N,\,a\ge 0,
\end{gather}
where we adopt the convention $\mcF^{\alpha,-1}:=t^\alpha_0$. There is also an analog of the dilaton equation:
\begin{gather}\label{eq:dilaton for F-CohFT}
\frac{\d\mcF^{\alpha,a}}{\d t^\un_1}=\sum_{n\ge 0}t^\beta_n\frac{\d\mcF^{\alpha,a}}{\d t^\beta_n}+\eps\frac{\d\mcF^{\alpha,a}}{\d\eps}-\mcF,\quad 1\le\alpha\le N,\,a\ge 0.
\end{gather}
Similarly to the case of a partial CohFT, let us define $w^{\top;\alpha}:=\frac{\d\mcF^{\alpha,0}}{\d t^\un_0}$ and $w^{\top;\alpha}_n:=\frac{\d^n w^{\top;\alpha}}{(\d t^\un_0)^n}$. From~\eqref{eq:string for F-CohFT}, it follows that the formal power series $w^{\top;\alpha}$ satisfy the property~\eqref{eq:main property of wtop}.

For $m,n\ge 1$, denote by $\perm_{m,n}$ the map $\oM_{g,n+m}\to\oM_{g,n+m}$ induced by the permutation of marked points $(1,\ldots,m,m+1,\ldots,m+n)\mapsto(m+1,\ldots,m+n,1,\ldots,m)$.

\begin{theorem}\label{theorem:reducing transformation for F-CohFT}
Let us fix $1\le\alpha\le N$, $a\ge 0$, $k\ge 0$, and $k$-tuples $\obeta=(\beta_1,\ldots,\beta_k)$, $\ob=(b_1,\ldots,b_k)$, where $1\le\beta_i\le N$, $b_i\ge 0$.
\begin{enumerate}
\item\label{th46-1} There exists a unique differential polynomial $\tOmega^{\alpha,a}_{\obeta,\ob}\in\hcA_{w;k-1}$ such that the difference
$$
\Omega^{\red;\alpha,a}_{\obeta,\ob}:=\frac{\d^k\mcF^{\alpha,a}}{\d t^{\beta_1}_{b_1}\cdots\d t^{\beta_k}_{b_k}}-\left.\tOmega^{\alpha,a}_{\obeta,\ob}\right|_{w^\gamma_c=w^{\top;\gamma}_c}
$$
satisfies the condition
$$
\Coef_{\eps^{2g}}\left.\frac{\d^n\Omega^{\red;\alpha,a}_{\obeta,\ob}}{\d t^{\alpha_1}_{d_1}\cdots\d t^{\alpha_n}_{d_n}}\right|_{t^*_*=0}=0\quad\text{if}\quad \sum d_i\le 2g-1+k.
$$
We consider the expansions $\tOmega^{\alpha,a}_{\obeta,\ob}=\sum_{g\ge 0}\eps^{2g}\tOmega^{\alpha,a}_{\obeta,\ob,g}$ and $\Omega^{\red;\alpha,a}_{\obeta,\ob}=\sum_{g\ge 0}\eps^{2g}\Omega^{\red;\alpha,a}_{\obeta,\ob,g}$.

\item\label{th46-2} We have $\left.\Omega^{\red;\alpha,a}_{\obeta,\ob,g}\right|_{t^*_*=0}=0$ and
\begin{gather*}
\left.\frac{\d^n \Omega^{\red;\alpha,a}_{\obeta,\ob,g}}{\d t^{\alpha_1}_{d_1}\cdots\d t^{\alpha_n}_{d_n}}\right|_{t^*_*=0}=\int_{\oM_{g,n+k+1}}\perm_{k+1,n}^*\left(B^{k+1}_{g,\od}\right)c_{g,n+k+1}\left(e^\alpha\otimes\otimes_{i=1}^k e_{\beta_i}\otimes \otimes_{j=1}^n e_{\alpha_j}\right)\psi_1^a\prod_{i=1}^k\psi_{i+1}^{b_i},
\end{gather*}
where $g\ge 0$, $n\ge 1$, $d_1,\ldots,d_n\ge 0$, $\sum d_i\ge 2g+k$, and $1\le\alpha_1,\ldots,\alpha_n\le N$.

\item\label{th46-3} We have $\left.\tOmega^{\alpha,a}_{\obeta,\ob,g}\right|_{w^*_*=0}=0$ and
\begin{gather*}
\left.\frac{\d^n\tOmega^{\alpha,a}_{\obeta,\ob,g}}{\d w^{\alpha_1}_{d_1}\cdots\d w^{\alpha_n}_{d_n}}\right|_{w^*_*=0}=\int_{\oM_{g,n+k+1}}\perm_{k+1,n}^*\left(B^{k+1}_{g,\od}\right)c_{g,n+k+1}\left(e^\alpha\otimes\otimes_{i=1}^k e_{\beta_i}\otimes \otimes_{j=1}^n e_{\alpha_j}\right)\psi_1^a\prod_{i=1}^k\psi_{i+1}^{b_i},
\end{gather*}
where $g\ge 0$, $n\ge 1$, $d_1,\ldots,d_n\ge 0$, $\sum d_i=2g-1+k$, and $1\le\alpha_1,\ldots,\alpha_n\le N$.
\end{enumerate}
\end{theorem}

\begin{proof}
\eqref{th46-1}~ The result is analogous to~\cite[Proposition~7.2]{BDGR18}, whose proof we briefly explained in the proof of Theorem~\ref{theorem:geometric formula for reducing Miura transformation}. Regarding the existence part, we construct the formal power series~$\Omega^{\red;\alpha,a}_{\obeta,\ob}$, together with the differential polynomial $\tOmega^{\alpha,a}_{\obeta,\ob}$, by a recursive procedure that kills all of the monomials $\eps^{2g}\prod t^{\alpha_i}_{d_i}$ with $g\ge 0$ and $\sum d_i\le 2g-1+k$ in~$\frac{\d^k\mcF^{\alpha,a}}{\d t^{\beta_1}_{b_1}\cdots\d t^{\beta_k}_{b_k}}$. One then checks that the resulting differential polynomial $\tOmega^{\alpha,a}_{\obeta,\ob}$ has degree $k-1$ using the equation $L\frac{\d^k\mcF^{\alpha,a}}{\d t^{\beta_1}_{b_1}\cdots\d t^{\beta_k}_{b_k}}=(k-1)\frac{\d^k\mcF^{\alpha,a}}{\d t^{\beta_1}_{b_1}\cdots\d t^{\beta_k}_{b_k}}$ (where the operator $L$ was defined in~\eqref{eq:Loperator}), which is a consequence of Equation~\eqref{eq:dilaton for F-CohFT}, and  property~\eqref{eq:main property of wtop}, which, as we already remarked, holds for F-CohFTs. The uniqueness part follows again from~\eqref{eq:main property of wtop}.

\eqref{th46-2}~ The result is analogous to~\cite[Proposition~3.5]{BGR19}, and the proof is obtained from the proof of that proposition by easily seen adjustments.

\eqref{th46-3}~ This is analogous to Theorem~\ref{theorem:geometric formula for reducing Miura transformation}.
\end{proof}

\subsection{The DZ hierarchy for an F-CohFT and Conjecture~\ref{conjecture1}}

We consider an arbitrary F-CohFT. 

\subsubsection{Construction of the DZ hierarchy}

Denote by $\cA^\wk_w$ the ring of formal power series in the shifted variables $(w^\alpha_n-A^\alpha\delta_{n,1})$, and let $\hcA^\wk_w:=\cA^\wk_w[[\eps]]$. We have the obvious inclusion $\hcA_w\subset\hcA^\wk_w$. From  property~\eqref{eq:main property of wtop}, it follows that for any $1\le\alpha\le N$, $a\ge 0$, and $k$-tuples $\obeta\in\llbracket N\rrbracket^k$ and $\ob\in\mbZ_{\ge 0}^k$, there exists a unique element $\Omega^{\alpha,a}_{\obeta,\ob}\in\hcA^\wk_w$ such that
$$
\frac{\d^k\mcF^{\alpha,a}}{\d t^{\beta_1}_{b_1}\cdots\d t^{\beta_k}_{b_k}}=\left.\Omega^{\alpha,a}_{\obeta,\ob}\right|_{w^\gamma_c=w^{\top;\gamma}_c}.
$$
Clearly, we have $\frac{\d w^{\top;\alpha}}{\d t^\beta_b}=\left.\d_x\Omega^{\alpha,0}_{\beta,b}\right|_{w^\gamma_c=w^{\top;\gamma}_c}$, which implies that the $N$-tuple of formal powers series $w^{\top;\alpha}$ satisfies the system of generalized PDEs
\begin{gather}\label{eq:DZ for F-CohFT}
\frac{\d w^\alpha}{\d t^\beta_b}=\d_x\Omega^{\alpha,0}_{\beta,b},\quad 1\le\alpha,\beta\le N,\, b\ge 0,
\end{gather}
which we call the \emph{Dubrovin--Zhang} (\emph{DZ}) \emph{hierarchy} associated to our F-CohFT. We say ``generalized PDEs'' because the right-hand sides are not differential polynomials but elements of the larger ring $\hcA^\wk_w$. The $N$-tuple $\ow^{\,\top}:=(w^{\top;1},\ldots,w^{\top;N})$ is clearly a solution of the DZ hierarchy, which is called the \emph{topological solution}.

\subsubsection{Conjecture~\ref{conjecture1} and the polynomiality of the DZ hierarchy}\label{subsec:poly}

From Theorem~\ref{theorem:reducing transformation for F-CohFT}, it follows that the validity of Conjecture~\ref{conjecture1} for some fixed $m\ge 2$ implies that $\Omega^{\red;\alpha,a}_{\obeta,\ob}=0$ if $l(\obeta)=l(\ob)=m-1$, or equivalently $\Omega^{\alpha,a}_{\obeta,\ob}=\tOmega^{\alpha,a}_{\obeta,\ob}\in\hcA_{w;m-2}$. In particular, we obtain the following result.

\begin{theorem}\label{theorem:implication of conjecture1}
The validity of Conjecture~\ref{conjecture1} for $m=2$ implies that the right-hand sides of the equations of the DZ hierarchy~\eqref{eq:DZ for F-CohFT} associated to an arbitrary F-CohFT are differential polynomials of degree~$1$. Moreover, part~\eqref{th46-3} of Theorem~\ref{theorem:reducing transformation for F-CohFT} then gives a geometric formula for these differential polynomials.
\end{theorem}

Also note that, assuming the validity of Conjecture~\ref{conjecture1} for $m=2$, we have
\begin{align*}
&\frac{\d}{\d t^\gamma_c}\frac{\d w^{\top;\alpha}}{\d t^\beta_b}=\left.\sum_{n\ge 0}\frac{\d\left(\d_x\Omega^{\alpha,0}_{\beta,b}\right)}{\d w^\mu_n}\d_x^{n+1}\Omega^{\mu,0}_{\gamma,c}\right|_{w^\delta_d=w^{\top;\delta}_d}\\
&\hspace{0.8cm}\veq\\
&\frac{\d}{\d t^\beta_b}\frac{\d w^{\top;\alpha}}{\d t^\gamma_c}=\left.\sum_{n\ge 0}\frac{\d\left(\d_x\Omega^{\alpha,0}_{\gamma,c}\right)}{\d w^\mu_n}\d_x^{n+1}\Omega^{\mu,0}_{\beta,b}\right|_{w^\delta_d=w^{\top;\delta}_d},
\end{align*}
which, by Lemma~\ref{lemma:wt-change}, implies that the flows of the DZ hierarchy commute pairwise.

\subsection{The DR hierarchy for an F-CohFT and Conjectures~\ref{conjecture2} and~\ref{conjecture3}}

\subsubsection{The definition of the DR hierarchy}

Consider an arbitrary F-CohFT of rank $N$. Let $u^1,\ldots,u^N$ be formal variables, and consider the associated ring of differential polynomials $\hcA_u$. Define differential polynomials $P^\alpha_{\beta,d}\in\hcA_{u;0}$, $1\le\alpha,\beta\le N$, $d\ge 0$, by
\begin{align*}
P^\alpha_{\beta,d}\coloneqq & \sum_{\substack{g,n\geq 0,\,2g+n>0\\k_1,\ldots,k_n\geq 0\\\sum_{j=1}^n k_j=2g}} \frac{\eps^{2g}}{n!}
\prod_{j=1}^n u^{\alpha_j}_{k_j} 
\\ \notag & \hphantom{\sum_{g,n\geq 0},} \times\Coef_{(a_1)^{k_1}\cdots(a_n)^{k_n}} \left(\int_{\oM_{g,n+2}}
\hspace{-0.5cm}
\lambda_g
{\DR_g\bigg(-\sum_{j=1}^n a_j,0,a_1,\ldots,a_n\bigg)}  \psi_2^d c_{g,n+2}(e^\alpha\otimes e_\beta\otimes \otimes_{j=1}^n e_{\alpha_j}) \right).
\end{align*}
The \emph{DR hierarchy}, see~\cite{BR21}, is the following system of evolutionary PDEs:
\begin{gather*}
\frac{\d u^\alpha}{\d t^\beta_d}=\d_x P^\alpha_{\beta,d},\quad 1\le \alpha,\beta\le N,\; d\ge 0.
\end{gather*}
In~\cite[Theorem~5.1]{BR21}, the authors proved that all of the equations of the DR hierarchy are compatible with each other. 

In~\cite[Theorem 1.5]{ABLR21}, the authors proved that
\begin{gather}\label{eq:string for P}
  \frac{\d P^\alpha_{\beta,d}}{\d u^{\un}}=
  \begin{cases}
    P^\alpha_{\beta,d-1}&\text{if $d\ge 1$},\\
    \delta^\alpha_\beta&\text{if $d=0$}.
  \end{cases}
\end{gather}

\begin{lemma}\label{lem48}
The DR hierarchy satisfies the following properties.
\begin{enumerate}
\item\label{lem48-1} $\displaystyle\frac{\d P^\alpha_{\beta,0}}{\d u^{\un}_x}=0$.

\item\label{lem48-2} The DR hierarchy has a unique solution $\ou^\str=(u^{\str;1},\ldots,u^{\str;N})$ satisfying the condition $\left.u^{\str;\alpha}\right|_{t^*_{\ge 1}=0}=t^\alpha_0$, where we identify the derivatives~$\d_x$ and~$\frac{\d}{\d t^\un_0}$.
\end{enumerate}
\end{lemma}

\begin{proof}
\eqref{lem48-1}~ This follows immediately from~\eqref{eq:divisibility of DR} and the definition of $P^{\alpha}_{\beta,0}$.

\eqref{lem48-2}~ Using Equation~\eqref{eq:string for P} and part~\eqref{lem48-1}, we compute
$$
\frac{\d}{\d u^\un_x}\d_x P^\alpha_{\beta,0}=\d_x\frac{\d P^\alpha_{\beta,0}}{\d u^\un_x}+\frac{\d P^\alpha_{\beta,0}}{\d u^\un}=\delta^\alpha_\beta,
$$
which completes the proof.
\end{proof}

\subsubsection{A collection of potentials associated to the DR hierarchy}

Define $u^{\str;\alpha}_n:=\frac{\d^n u^{\str;\alpha}}{(\d t^\un_0)^n}$. Let us introduce~$N$ formal power series $\mcF^{\DR;\alpha}\in\mbC[[t^*_*,\eps]]$, $1\le\alpha\le N$, by the relation
$$
\frac{\d\mcF^{\DR;\alpha}}{\d t^\beta_b}:=\left.P^\alpha_{\beta,b}\right|_{u^\gamma_n=u^{\str;\gamma}_n},
$$
with the constant terms defined to be equal to zero, $\left.\mcF^{\DR;\alpha}\right|_{t^*_*=0}:=0$. Consider the expansion $\mcF^{\DR;\alpha}=\sum_{g\ge 0}\eps^{2g}\mcF^{\DR;\alpha}_g$.

\begin{theorem}\label{theorem:DR correlators for F-CohFT}
Let $g\ge 0$, $n\ge 1$, $d_1,\ldots,d_n\ge 0$, and $1\le\alpha,\alpha_1,\ldots,\alpha_n\le N$. Then we have
$$
\left.\frac{\d^n\mcF^{\DR;\alpha}_g}{\d t^{\alpha_1}_{d_1}\cdots\d t^{\alpha_n}_{d_n}}\right|_{t^*_*=0}=
\begin{cases}
0 & \text{if\, $\sum d_i\le 2g-1$},\\
\int_{\oM_{g,n+1}}\perm_{1,n}^*\left(A^1_{g,\od}\right)c_{g,n+1}\left(e^\alpha\otimes\otimes_{j=1}^n e_{\alpha_j}\right)&\text{if\, $\sum d_i\ge 2g$}.
\end{cases}
$$
\end{theorem}

\begin{proof}
The proof is very similar to the proof of analogous statements in the case of the DR hierarchy associated to a CohFT (see~\cite[Proposition~6.10]{BDGR18} and~\cite[Theorem~6.1]{BDGR20}). So we only very briefly sketch the details.

Consider a stable tree $T\in\SRT_{g,n,1}$. We will call a level function $l\colon V(T)\to\mbZ_{\ge 1}$ \emph{injective} if $|l^{-1}(i)|=1$ for each $1\le i\le\deg(l)$. Clearly, such a function gives a bijection between the sets $V(T)$ and $\br{|V(T)|}$. Denote by $\mcL^{(i)}(T)$ the set of all injective level functions on $T$. For $l\in\mcL^{(i)}(T)$, denote by $\overline{(T,l)}$ the stable tree from the set $\SRT_{g,n+|V(T)|,1}$ obtained as follows:
\begin{itemize}
\item We attach to each vertex $v$ of $T$ a new regular leg and label it by the number $l(v)+n$. 

\item We relabel the unique frozen leg of $T$ by the number $n+|V(T)|+1$.
\end{itemize}

For an $N$-tuple $\oQ=(Q^1,\ldots,Q^N)\in\hcA_u^N$, denote by $D_{\oQ}$ the linear operator in $\hcA_u$ defined by
$$
D_{\oQ}:=\sum_{k\ge 0}\left(\d_x^k Q^\gamma\right)\frac{\d}{\d u^\gamma_k}.
$$
Let $\oP_{\beta,d}:=(P^1_{\beta,d},\ldots,P^N_{\beta,d})$. From the definition of the formal power series $u^{\str;\alpha}$ and $\mcF^{\DR;\alpha}$, it follows that
$$
\left.\frac{\d^n\mcF^{\DR;\alpha}_g}{\d t^{\alpha_1}_{d_1}\cdots\d t^{\alpha_n}_{d_n}}\right|_{t^*_*=0}=\Coef_{\eps^{2g}}\left.\left(D_{\d_x\oP_{\alpha_n,d_n}}\cdots D_{\d_x\oP_{\alpha_3,d_3}}D_{\d_x\oP_{\alpha_2,d_2}}P^{\alpha}_{\alpha_1,d_1}\right)\right|_{u^\gamma_k=\delta_{k,1}A^\gamma}.
$$
In the same way as in the proof of \cite[Lemma~6.9]{BDGR18}, this formula implies that
\begin{align*}
  &\left.\frac{\d^n\mcF^{\DR;\alpha}_g}{\d t^{\alpha_1}_{d_1}\cdots\d t^{\alpha_n}_{d_n}}\right|_{t^*_*=0}\\
 &\qquad =\frac{1}{(2g+n-1)!}\Coef_{a_1\cdots a_{2g+n-1}}\left(\sum_{T\in\SRT^n_{g,2g+n-1,1}}\sum_{l\in\mcL^{(i)}(T)}\right.\\
&\qquad\hphantom{=}\left.\int_{\oM_{g,2g+2n}}\hspace{-0.5cm}\lambda_g\,\perm^*_{1,2g+2n-1}\DR_{\overline{(T,l)}}\left(A,\ozero_n,-\sum a_i\right)\,c_{g,2g+2n}\left(e^\alpha\otimes e^{\otimes (2g+n-1)}\otimes\otimes_{i=1}^n e_{\alpha_i}\right)\prod_{i=1}^n\psi_{2g+n+i}^{d_i}\right),
\end{align*}
where $A=(a_1,\ldots,a_{2g+n-1})$ and $\ozero_n:=(0,\ldots,0)\in\mbZ^n$. In the case $\sum d_i\le 2g-1$, the proof of the vanishing of the right-hand side of this equation goes along the same lines as the proof of \cite[Lemma~6.11]{BDGR18}. Indeed, \cite[Lemma~6.11]{BDGR18} says that the right-hand side of the equation vanishes if we replace the class $c_{g,2g+2n}(e^\alpha\otimes e^{\otimes (2g+n-1)}\otimes\otimes_{i=1}^n e_{\alpha_i})$ with the class $c_{g,2g+2n}(e\otimes e^{\otimes (2g+n-1)}\otimes\otimes_{i=1}^n e_{\alpha_i})$ with $c_{g,2g+2n}$ being a CohFT. However, one can easily see that the same proof works in our case as well.

We now consider  the case $\sum d_i\ge 2g$. For a tree $\SRT_{g,n,1}$ and $v\in V(T)$, define  
$$
\ind(v):=\min\{1\le i\le n\;|\;\text{$\sigma_i$ is attached to $v$}\}.
$$
If $v$ is not incident to any regular leg, then we write $\ind(v):=\infty$. Let us call the tree $T$ \emph{special} (in~\cite[Section~6.5.1]{BDGR20}, the authors said ``admissible'')~if
\begin{itemize}
\item[a)] $\ind(v)<\infty$ for any vertex $v\in V(T)$;

\item[b)] $\ind(v_1)<\ind(v_2)$ for any two distinct vertices $v_1,v_2\in V(T)$ such that $v_2$ is a descendant of $v_1$.
\end{itemize}
Denote by $\SRT_{g,n,1}^{(s)}\subset\SRT_{g,n,1}$ the subset of all special trees.

For a tree $T\in\SRT^{(s)}_{g,n,1}$, define 
$$
S_T:=\left\{\oc=(c_1,\ldots,c_n)\in\mbZ_{\ge 0}^n\;\left|\;\begin{minipage}{4.5cm}
$c_i=0$ unless $i=\ind(v)$ for some $v\in V(T)$
\end{minipage}\right.\right\}.
$$
Let $d\ge 2g$. We claim that
\begin{align*}
&\sum_{\substack{d_1,\ldots,d_n\ge 0\\\sum d_i=d}}\frac{\prod b_i^{d_i}}{(2g+n-1)!}\Coef_{a_1\cdots a_{2g+n-1}}\left(\sum_{T\in\SRT^n_{g,2g+n-1,1}}\sum_{l\in\mcL^{(i)}(T)}\right.\\
&\quad\left.\int_{\oM_{g,2g+2n}}\hspace{-0.5cm}\lambda_g\,\perm^*_{1,2g+2n-1}\DR_{\overline{(T,l)}}\Big(A,\ozero_n,-\sum a_i\Big)\,c_{g,2g+2n}\left(e^\alpha\otimes e^{\otimes (2g+n-1)}\otimes\otimes_{i=1}^n e_{\alpha_i}\right)\prod_{i=1}^n\psi_{2g+n+i}^{d_i}\right)\\
&\quad=\sum_{T\in\SRT^{(s)}_{g,n,1}}\sum_{\substack{\oc\in S_T\\\sum c_i=d-(2g+|V(T)|-1)}}\int_{\oM_{g,n+1}}\hspace{-0.5cm}\lambda_g\,\perm^*_{1,n}\DR_T\left(B,-\sum b_i\right)c_{g,n+1}(e^\alpha\otimes\otimes_{i=1}^n e_{\alpha_i})\prod_{i=1}^n(b_i\psi_{i+1})^{c_i},
\end{align*}
where $B=(b_1,\ldots,b_n)$. Indeed, in~\cite[Section~6.5.2]{BDGR20}, this equation is proved if we replace the classes $c_{g,k+1}(e^\alpha\otimes\otimes_{i=1}^k v_i)$, $v_i\in V$, with the classes $c_{g,k+1}(e\otimes\otimes_{i=1}^k v_i)$ with $c_{g,k+1}$ being a CohFT. However, one can easily check that the same proof works in our case as well.

Therefore, it is sufficient to check the cohomological relation
\begin{align*}
&\sum_{T\in\SRT^{(s)}_{g,n,1}}\sum_{\substack{\oc\in S_T\\\sum c_i=d-(2g+|V(T)|-1)}}\lambda_g\DR_T\left(b_1,\ldots,b_n,-\sum b_i\right)\prod_{i=1}^n(b_i\psi_{i})^{c_i}\\
&\quad=\sum_{T\in\SRT^{d-2g+1}_{g,n,1}}C(T)\lambda_g\DR_T\left(b_1,\ldots,b_n,-\sum b_i\right),
\end{align*}
but this was done in~\cite[Equation~(6.13)]{BDGR20}.
\end{proof}

Note that the definition of formal power series $\mcF^{\DR;\alpha}$ implies that $\frac{\d}{\d t^\beta_d}\frac{\d\mcF^{\DR;\alpha}}{\d t^\un_0}=\frac{\d u^{\str;\alpha}}{\d t^\beta_d}$. Since, by Theorem~\ref{theorem:DR correlators for F-CohFT}, we have $\left.\frac{\d\mcF^{\DR;\alpha}}{\d t^\un_0}\right|_{t^*_*=0}=0$, we conclude that
$$
\frac{\d\mcF^{\DR;\alpha}}{\d t^\un_0}=u^{\str;\alpha}.
$$

\subsubsection{Conjecture~\ref{conjecture2} and the equivalence of the two hierarchies for F-CohFTs}

Consider an arbitrary F-CohFT. Suppose that Conjecture~\ref{conjecture2} is true. Then Theorems~\ref{theorem:reducing transformation for F-CohFT} and~\ref{theorem:DR correlators for F-CohFT} imply that $\Omega^{\red;\alpha,0}=\mcF^{\DR;\alpha}$ and therefore
$$
\mcF^{\alpha,0}-\left.\tOmega^{\alpha,0}\right|_{w^\gamma_c=w^{\top;\gamma}_c}=\mcF^{\DR,\alpha}.
$$
Differentiating both sides by $\frac{\d}{\d t^\un_0}$, we obtain
$$
w^{\top;\alpha}-\left.\left(\d_x\tOmega^{\alpha,0}\right)\right|_{w^\gamma_c=w^{\top;\gamma}_c}=u^{\str;\alpha}.
$$
Consider the Miura transformation
\begin{gather}\label{eq:DRDZ-Miura}
w^\alpha\longmapsto u^\alpha(w^*_*,\eps)=w^\alpha-\d_x\tOmega^{\alpha,0} 
\end{gather}
and its inverse $u^\alpha\mapsto w^\alpha(u^*_*,\eps)$. We see that the  Miura transformation~\eqref{eq:DRDZ-Miura} transforms the DR hierarchy to a hierarchy having the $N$-tuple $\ow^{\,\top}$ as its solution. Lemma~\ref{lemma:wt-change} implies that this hierarchy coincides with the DZ hierarchy corresponding to our F-CohFT. In particular, we obtain that the right-hand side of each equation of the DZ hierarchy~\eqref{eq:DZ for F-CohFT} is a differential polynomial of degree $1$.

Summarizing, we obtain the following result.

\begin{theorem}\label{theorem:implication of conjecture2}
Conjecture~\ref{conjecture2} implies that the DZ hierarchy corresponding to an arbitrary F-CohFT is polynomial and that it is related to the DR hierarchy by the Miura transformation~\eqref{eq:DRDZ-Miura}. Moreover, there is a geometric formula for the differential polynomials~$\tOmega^{\alpha,0}$ defining this Miura transformation, given by part~\eqref{th46-3} of Theorem~\ref{theorem:reducing transformation for F-CohFT}.
\end{theorem}

\subsubsection{Conjecture~\ref{conjecture3} and the DR/DZ equivalence conjecture for partial CohFTs}\label{subsubsection:implication of conjecture3}

Following~\cite{BGR19}, let us recall here the relation between Conjecture~\ref{conjecture3} and the DR/DZ equivalence conjecture for partial CohFTs, proposed in~\cite{BDGR18}. Again, what we will say was discussed in~\cite{BGR19} and~\cite{BDGR18} for CohFTs, but the required results are true for partial CohFTs, with the same proofs.

Consider an arbitrary partial CohFT, the associated F-CohFT, and the corresponding DR hierarchy. Define
$$
h_{\alpha,p}:=\eta_{\un\mu}P^\mu_{\alpha,p+1}\in\hcA_{u;0},\quad 1\le\alpha\le N,\,\,p\ge -1.
$$
In~\cite{BDGR18}, the authors proved that for any $1\le\alpha,\beta\le N$ and $p,q\ge 0$, there exists a unique differential polynomial $\Omega^\DR_{\alpha,p;\beta,q}\in\hcA_{u;0}$ such that
$$
\d_x\Omega^\DR_{\alpha,p;\beta,q}=D_{\d_x\oP_{\beta,q}}h_{\alpha,p-1},\quad\left.\Omega^\DR_{\alpha,p;\beta,q}\right|_{u^*_*=0}=0.
$$
In~\cite{BDGR18}, the authors proved that there exists a unique formal power series $\mcF^\DR(t^*_*,\eps)=\sum_{g\ge 0}\eps^{2g}\mcF^\DR_g(t^*_*)\in\mbC[[t^*_*,\eps]]$ such that
\begin{align*}
&\frac{\d^2 \mcF^\DR}{\d t^\alpha_p\d t^\beta_q}=\left.\Omega^\DR_{\alpha,p;\beta,q}\right|_{u^\gamma_c=u^{\str;\gamma}_c}, && 1\le\alpha,\beta\le N,\quad p,q\ge 0,\\
&\frac{\d \mcF^\DR}{\d t^\un_0}=\sum_{n\ge 0}t^\alpha_{n+1}\frac{\d\mcF^\DR}{\d t^\alpha_n}+\frac{1}{2}\eta_{\alpha\beta}t^\alpha_0 t^\beta_0, &&\\
&\left.\mcF^\DR\right|_{t^*_*=0}=0. &&
\end{align*}
The formal power series $\mcF^{\DR}$ is called the \emph{potential} of the DR hierarchy. By \cite[Proposition~6.10]{BDGR18} and \cite[Theorem 6.1]{BDGR20}, for $g\ge 0$, $n\ge 1$, $d_1,\ldots,d_n\ge 0$, and $1\le\alpha_1,\ldots,\alpha_n\le N$, we have
\begin{gather}\label{eq:DR correlators}
\left.\frac{\d^n\mcF^{\DR}}{\d t^{\alpha_1}_{d_1}\cdots\d t^{\alpha_n}_{d_n}}\right|_{t^*_*=0}=
\begin{cases}
0&\text{if $\sum d_i\le 2g-2$},\\
\int_{\oM_{g,n}}A^0_{g,\od}\,c_{g,n}\left(\otimes_{i=1}^n e_{\alpha_i}\right)&\text{if $\sum d_i\ge 2g-1$}.
\end{cases}
\end{gather}

Suppose that Conjecture~\ref{conjecture3} is true. Then Equations~\eqref{eq:reduced correlators} and~\eqref{eq:DR correlators} imply that $\mcF^\red=\mcF^\DR$ or, equivalently,
\begin{gather}\label{eq:Fred and FDR}
\mcF-\left.\cP\right|_{w^\gamma_c=w^{\top;\gamma}_c}=\mcF^\DR,
\end{gather}
which was formulated in~\cite{BDGR18} and called the \emph{strong DR/DZ equivalence conjecture}. Let us discuss consequences of Equation~\eqref{eq:Fred and FDR}.

Consider the Miura transformation
$$
u^\alpha\longmapsto\tu^{\,\alpha}(u^*_*,\eps):=\eta^{\alpha\mu}\Omega^\DR_{\mu,0;\un,0}.
$$
The variables $\tu^\alpha$ are called the \emph{normal coordinates} of the DR hierarchy. Obviously, the $N$-tuple of formal power series
$$
\tu^{\;\str;\alpha}(t^*_*,\eps):=\eta^{\alpha\mu}\frac{\d^2\mcF^\DR}{\d t^\mu_0\d t^\un_0},\quad 1\le\alpha\le N,
$$
is a solution of the DR hierarchy written in the normal coordinates. Now consider the following Miura transformation relating the variables $w^\alpha$ and $\tu^{\,\alpha}$:
\begin{gather}\label{eq:normal Miura}
w^\alpha\longmapsto\tu^{\,\alpha}(w^*_*,\eps):=w^\alpha-\eta^{\alpha\mu}\d_x D_{\d_x\oP_{\mu,0}}\cP.
\end{gather}
Equation~\eqref{eq:Fred and FDR} implies that the DR hierarchy written in the variables $w^\alpha$ has the $N$-tuple~$\ow^{\,\top}$ as its solution. Therefore, by Lemma~\ref{lemma:wt-change}, this hierarchy coincides with the DZ hierarchy corresponding to our partial CohFT. So we obtain that the DZ hierarchy corresponding to our partial CohFT is polynomial and that it is related to the DR hierarchy by a Miura transformation.

In fact, as is explained in~\cite{BDGR18}, Equation~\eqref{eq:Fred and FDR} implies that there is a relation between the two hierarchies that is stronger than the Miura equivalence. Both hierarchies are Hamiltonian, and they are endowed with a tau-structure, meaning that there a special choice of densities of the Hamiltonians satisfying certain properties. For the DZ hierarchy, these densities are $\eta_{\un\mu}\Omega^{\mu,0}_{\alpha,p+1}$, and for the DR hierarchy, these densities are $h_{\alpha,p}$, $1\le\alpha\le N$, $p\ge 0$. In~\cite{BDGR18}, the authors proved that Equation~\eqref{eq:Fred and FDR} implies that the two hierarchies, together with their Hamiltonian structure and tau-structure, are related by the normal Miura transformation given by~\eqref{eq:normal Miura}. Moreover, there is a geometric formula for the differential polynomial~$\cP$ describing this normal Miura transformation, given by Theorem~\ref{theorem:geometric formula for reducing Miura transformation}.

%%%%%%%%%%%%%%%%%%%%%%%%%%%%%%%%%%%%%%%%%%%%%%%%%%%%%%%%%%%%%%%%%%%
%%%%%%%%%%%%%%%%%%%%%%%%%%%%%%%%%%%%%%%%%%%%%%%%%%%%%%%%%%%%%%%%%%%

\section{Proof of Theorem~\ref{theorem:main}}\label{section:proof of main theorem}

\subsection{Proof of Conjecture~\ref{conjecture1} for $\boldsymbol{n=1}$}\label{subsection:VanishingOnePoint}

In the case $n=1$, Conjecture~\ref{conjecture1} says that $B^m_{g,d}=0$ for $m\ge 2$ and $d\geq 2g+m-1$. Let us prove that. To explain the proof, we need some notation. We denote by $\Gamma^{g,m}_{d\mathop{|}k}$, $g\geq 0$, $k\geq 1$, the following tautological class:
\begin{align*}
	\Gamma^{g,m}_{d\mathop{|}k} \coloneqq \sum_{\substack{g_1,\dots,g_k\\d_1,\dots,d_k}} 
\vcenter{\xymatrix@C=20pt@R=5pt{
	& & & & & & {}\ar@{.}@/^/[dd] \\
		 & *+[o][F-]{{g_1}}\ar@{-}[l]*{{}_{1\,}}_<<{\psi^{d_1}\,} & *+[o][F-]{{g_2}}\ar@{-}[l]_<<{\psi^{d_2}\,}\ar@{-}[r] & \cdots & *+[o][F-]{{g_{k-1}}}\ar@{-}[l]_<<<{\psi^{d_{k-1}}} & *+[o][F-]{{g_k}}\ar@{-}[l]_<<{\psi^{d_k}\,}\ar@{-}[ru]*{{}_{\,2}}\ar@{-}[rd]*{{}_{\hspace{0.25cm} m+1}}\ar@{-}[r] & \\
 & & & & & & }}\in R^d(\oM_{g,m+1}), 
\end{align*} 
where $d_1+\cdots+d_k+k-1=d$, $g_1+\cdots+g_k=g$, $g_1,\dots,g_{k-1}\geq 1$,
%$g_k\geq\begin{cases} 0,&\text{if $m\geq 2$},\\ 1,&\text{if $m=0,1$}\end{cases}$,
$g_k\geq 0$ if $m\geq 2$ and $g_k\geq 1$ if $m=0,1$,  
and for any $i=2,\dots,k$, we have $d_i+\cdots+d_k+k-i\leq 2(g_i+\cdots+g_{k})+m-2$. 

\begin{lemma}\label{lem:m1} We have $B^m_{g,d} = B^m_{g,2g+m-1} \psi_1^{d-2g-m+1}$ for $d\geq 2g+m-1$ and 
\begin{align*}
	B^m_{g,2g+m-1} = \sum_{k=1}^{\infty} (-1)^{k+1}  \Gamma^{g,m}_{2g+m-1\mathop{|} k}. 
\end{align*}
\end{lemma}
\begin{proof} 
This follows directly from unfolding the definitions. Note that the sum over $k$ is in fact finite since $\Gamma^{g,m}_{d\mathop{|}k}=0$ for $k\geq g+2$ for any $m$ and $\Gamma^{g,m}_{d\mathop{|}k}=0$ for $k\geq g+1$ for $m=0,1$.
\end{proof}

By this lemma, it is sufficient to prove that $B^m_{g,2g+m-1}=0$, $g\geq 0$, $m\geq 2$. The proof is based on the following Liu--Pandharipande relations, see ~\cite[Propositions 1.1 and 1.2]{LP11} (see also~\cite[Corollary~3.2]{BHIS21}), in the tautological ring:
\begin{align}
&\xymatrix@C=17pt@R=0pt{
& *+[o][F-]{{g}}\ar@{-}[l]*{{}_{1\,}}_<<<{\psi^{2g+r}}\ar@{-}[r]*{{}_{\,2}} & 
} +
(-1)^{2g+r+1}\xymatrix@C=17pt@R=0pt{
& *+[o][F-]{{g}}\ar@{-}[r]*{{}_{\,2}}^<<<<{\psi^{2g+r}}\ar@{-}[l]*{{}_{1\,}} & 
} =
\sum_{\substack{g_1+g_2=g\\g_1,g_2\ge 1\\d_1+d_2=2g+r-1\\d_1,d_2\ge 0}}(-1)^{d_1}\xymatrix@C=12pt@R=0pt{
& *+[o][F-]{{g_1}}\ar@{-}[rr]^<<<{\psi^{d_1}}\ar@{-}[l]*{{}_{1\,}} & & *+[o][F-]{{g_2}}\ar@{}[ll]_<<{\psi^{d_2\,}}\ar@{-}[r]*{{}_{\,2}} & 
},\label{eq:LP1}\\
&\vcenter{\xymatrix@C=15pt@R=5pt{
& & & {}\ar@{.}@/^/[dd] \\
& & *+[o][F-]{{g}}\ar@{-}[ll]*{{}_{1\,}}_<<<<<{\raisebox{0.15cm}{$\scriptstyle{\psi^{2g+m-1+r}}$}}\ar@{-}[ru]*{{}_{\,2}}\ar@{-}[rd]*{{}_{\hspace{0.25cm} m+1}}\ar@{-}[r] & \\
& & & }} =
\sum_{\substack{g_1+g_2=g\\g_1\ge 1,\,g_2\ge 0\\d_1+d_2=2g+m-2+r\\d_1,d_2\ge 0}}(-1)^{d_1}\vcenter{\xymatrix@C=15pt@R=5pt{
& & & & {}\ar@{.}@/^/[dd] \\
& *+[o][F-]{{g_1}}\ar@{-}[l]*{{}_{1\,}}\ar@{-}[rr]^<<<{\psi^{d_1}} & & *+[o][F-]{{g_2}}\ar@{}[ll]_<<{\psi^{d_2}\,}\ar@{-}[ru]*{{}_{\,2}}\ar@{-}[rd]*{{}_{\hspace{0.25cm} m+1}}\ar@{-}[r] & \\
& & & & }}\label{eq:LP2}
\end{align}
for any $r\geq 0$. 

We introduce more notation. For $g\geq 1$, $k\geq 1$, and $d\le 2g-1$, let $\gamma^g_{d\mathop{|}k}$ be the sum of decorated stable graphs given as 
\begin{align*}
	\gamma^{g}_{d\mathop{|}k}\coloneqq \sum_{\substack{g_1,\dots,g_k\\d_1,\dots,d_k}} 
	\xymatrix@C=25pt@R=0pt{
& *+[o][F-]{{g_1}}\ar@{-}[l]*{{}_{1\,}}\ar@{-}[r]^<<<{\psi^{d_1}} & *+[o][F-]{{g_2}}\ar@{-}[r]^<<<{\psi^{d_2}} & \cdots & *+[o][F-]{{g_{k-1}}}\ar@{-}[l]\ar@{-}[r]^<<<<{\psi^{d_{k-1}}} & *+[o][F-]{{g_k}}\ar@{-}[r]*{{}_{\,2}}^<<<{\psi^{d_k}} & }, 
\end{align*} 
where $d_1+\cdots+d_k+k-1=d$, $g_1+\cdots+g_k=g$, $g_1,\dots,g_{k}\geq 1$, and for any $i=1,\dots,k-1$, we have $d_1+\cdots+d_i+i-1\leq 2(g_1+\cdots+g_{i})-1$. For $d\ge 2g$, we define $\gamma^{g}_{d\mathop{|}k}\coloneqq 0$. Note that $\gamma^{g}_{d\mathop{|}k}=0$ for $k>g$. We also define 
$$
\tgamma^{g,m}_{d\mathop{|}k}\coloneqq\begin{cases}\Gamma^{g,m}_{d\mathop{|}k}&\text{if $d\le 2g+m-2$},\\ 0&\text{otherwise}.\end{cases}
$$

Let $\gamma_1\diamond\gamma_2$ be the operation of concatenation of two decorated stable graphs $\gamma_1\in G_{g_1,2}$ and $\gamma_2\in G_{g_2,m+1}$ that forms an edge from the second leg of $\gamma_1$ and the first leg of $\gamma_2$.

\begin{lemma} \label{lem:1ptInductiveLemma} 
For any $\ell\geq 1$, $m\geq 2$, we have
\begin{align} \label{eq:1ptInd}
& \sum_{k=1}^{\ell+1} (-1)^{k+1} \Gamma^{g,m}_{2g+m-1\mathop{|}k}\\
&\notag\; = \sum_{\substack{g_1+g_2=g\\ g_1\geq 1,\, g_2\geq 0 \\ d_1+d_2 = 2g+m-2\\ d_1,d_2\geq 0}} (-1)^{d_1+\ell-1}\gamma^{g_1}_{d_1\mathop{|}\ell} \diamond 
\vcenter{\xymatrix@C=14pt@R=5pt{
& & {}\ar@{.}@/^/[dd] \\ 
&*+[o][F-]{{g_2}}\ar@{-}[ru]*{{}_{\,2}}\ar@{-}[rd]*{{}_{\hspace{0.25cm} m+1}}\ar@{-}[l]*{{}_{1\,}}_{\psi^{d_2}} \ar@{-}[r] & \\
& &	}}\\
\notag & \hspace{1.3cm} + \sum_{\substack{g_1+g_2=g\\ g_1\geq 1,\,g_2\geq 0 \\ d_1+d_2 = 2g+m-2\\ d_1,d_2\geq 0}}(-1)^\ell
\left( 
\xymatrix@C=14pt@R=0pt{
&*+[o][F-]{{g_1}}\ar@{-}[r]*{{}_{\,2}}\ar@{-}[l]*{{\ }_{1\,}}_{\psi^{d_1}} & } + (-1)^{d_1+1}
\xymatrix@C=14pt@R=0pt{
&*+[o][F-]{{g_1}}\ar@{-}[l]*{{}_{1\,}}\ar@{-}[r]*{{}_{\,2}}^{\,\,\psi^{d_1}} & }\right) \diamond \tgamma^{g_2,m}_{d_2\mathop{|}\ell} \\ 
\notag & \hspace{1.3cm} + \sum_{\substack{g_1+g_2+g_3=g\\ g_1,g_2\geq 1,\, g_3\geq 0\\ d_1+d_2+d_3  = 2g+m-3\\ d_1,d_2,d_3\geq 0\\ k_1+k_3=\ell,\, k_1,k_3\geq 1}}
(-1)^{d_1+\ell-1} \gamma^{g_1}_{d_1\mathop{|}k_1}\diamond 
\left(\xymatrix@C=14pt@R=0pt{
&*+[o][F-]{{g_2}}\ar@{-}[r]*{{}_{\,2}}\ar@{-}[l]*{{}_{1\,}}_{\psi^{d_2}} & } + (-1)^{d_2+1}
\xymatrix@C=14pt@R=0pt{
&*+[o][F-]{{g_2}}\ar@{-}[l]*{{}_{1\,}}\ar@{-}[r]*{{}_{\,2}}^{\,\,\psi^{d_2}} & } \right)
\diamond \tgamma^{g_3,m}_{d_3\mathop{|}k_3}.
\end{align}
\end{lemma}

Before proving this lemma, let us show how to derive the desired vanishing $B^m_{g,2g+m-1}=0$ from it. Indeed, consider Equation~\eqref{eq:1ptInd} for $\ell=g+1$. Note that since $m\geq 2$, we have  $\Gamma^{g,m}_{d\mathop{|} k} =0$ for $g+1<k$. Hence the left-hand side of Equation~\eqref{eq:1ptInd} with $\ell\geq g$, so in particular for $\ell=g+1$, is equal to $B^m_{g,2g+m-1}$. 

On the right-hand side, we have $\gamma^{g_1}_{d_1\mathop{|} g+1} =0$ in the first term (since $g_1\leq g<g+1$) and $\tgamma^{g_2,m}_{d_2\mathop{|} g+1} =0$ in the second term (since $g_2+1<g+1$). Consider the third term. We have $k_1+k_3 = g+1$ and $g_1+g_3<g$, so either $g_1<k_1$, or $g_3+1<k_3$, or both. Hence in the third term either $\gamma^{g_1}_{d_1\mathop{|} k_1} =0$, or  $\tgamma^{g_3,m}_{d_3\mathop{|} k_3} =0$, or both. Hence the right-hand side of Equation~\eqref{eq:1ptInd} with $\ell=g+1$ is equal to $0$.

\begin{proof}[Proof of Lemma~\ref{lem:1ptInductiveLemma}]
We prove the lemma by induction. The base is the $\ell =1$ case. It is equivalent to  
\begin{gather} \label{eq:1ptInd-base}
\Gamma^{g,m}_{2g+m-1 \mathop{|} 1} = \sum_{\substack{g_1+g_2=g\\g_1\geq 1,\,g_2\geq 0 \\ d_1+d_2 = 2g+m-2\\d_1,d_2\geq 0}}(-1)^{d_1}\gamma^{g_1}_{d_1\mathop{|}1} \diamond 
\vcenter{\xymatrix@C=14pt@R=5pt{
& & {}\ar@{.}@/^/[dd] \\ 
&*+[o][F-]{{g_2}}\ar@{-}[ru]*{{}_{\,2}}\ar@{-}[rd]*{{}_{\hspace{0.25cm} m+1}}\ar@{-}[l]*{{}_{1\,}}_{\psi^{d_2}} \ar@{-}[r] & \\
& &	}} + \sum_{\substack{g_1+g_2=g\\g_1\geq 1,\,g_2\geq 0\\d_1+d_2 = 2g+m-2\\d_1,d_2\geq 0}}(-1)^{d_1} 
\xymatrix@C=14pt@R=0pt{
&*+[o][F-]{{g_1}}\ar@{-}[l]*{{}_{1\,}}\ar@{-}[r]*{{}_{\,2}}^{\,\,\psi^{d_1}} & } \diamond \tgamma^{g_2,m}_{d_2\mathop{|}1},
\end{gather}
which is just a way to rewrite Equation~\eqref{eq:LP2} for $r=0$. Indeed, on the right-hand side of Equation~\eqref{eq:LP2} for $r=0$, we have
\begin{align*}
\sum_{\substack{g_1+g_2=g,\,g_1\geq 1,\,g_2\geq 0\\ d_1+d_2=2g+m-2,\, d_1,d_2\geq 0}} 
(-1)^{d_1}
\xymatrix@C=14pt@R=0pt{
&*+[o][F-]{{g_1}}\ar@{-}[l]*{{}_{1\,}}\ar@{-}[r]*{{}_{\,2}}^{\,\,\psi^{d_1}} & } \diamond
\vcenter{\xymatrix@C=14pt@R=5pt{
& & {}\ar@{.}@/^/[dd] \\ 
&*+[o][F-]{{g_2}}\ar@{-}[ru]*{{}_{\,2}}\ar@{-}[rd]*{{}_{\hspace{0.25cm} m+1}}\ar@{-}[l]*{{}_{1\,}}_{\psi^{d_2}} \ar@{-}[r] & \\
& &	}},
\end{align*}
and either $d_1\leq 2g_1-1$, which gives us the first summand on the right-hand side of~\eqref{eq:1ptInd-base}, or $d_2\leq 2g_2+m-2$, which gives us the second summand on the right-hand side of~\eqref{eq:1ptInd-base}.

For the induction step, we have to prove that
\allowdisplaybreaks
\begin{align} 
& \underline{(-1)^{\ell+1}\Gamma^{g,m}_{2g+m-1\mathop{|} \ell+2}} + \boxed{\sum_{\substack{g_1+g_2=g,\, g_1\geq 1,\, g_2\geq 0 \\ d_1+d_2 = 2g+m-2,\, d_1,d_2\geq 0}} (-1)^{d_1+\ell-1}\gamma^{g_1}_{d_1\mathop{|}\ell} \diamond 
\vcenter{\xymatrix@C=14pt@R=5pt{
& & {}\ar@{.}@/^/[dd] \\ 
&*+[o][F-]{{g_2}}\ar@{-}[ru]*{{}_{\,2}}\ar@{-}[rd]*{{}_{\hspace{0.25cm} m+1}}\ar@{-}[l]*{{}_{1\,}}_{\psi^{d_2}} \ar@{-}[r] & \\
& &	}}}  \label{eq:1ptInd-2}\\
& + \underline{\underline{\sum_{\substack{g_1+g_2=g,\, g_1\geq 1,\,g_2\geq 0 \\ d_1+d_2 = 2g+m-2,\, d_1,d_2\geq 0}}(-1)^\ell
\left( 
\xymatrix@C=14pt@R=0pt{
&*+[o][F-]{{g_1}}\ar@{-}[r]*{{}_{\,2}}\ar@{-}[l]*{{\ }_{1\,}}_{\psi^{d_1}} & } + (-1)^{d_1+1}
\xymatrix@C=14pt@R=0pt{
&*+[o][F-]{{g_1}}\ar@{-}[l]*{{}_{1\,}}\ar@{-}[r]*{{}_{\,2}}^{\,\,\psi^{d_1}} & }\right) \diamond \tgamma^{g_2,m}_{d_2\mathop{|}\ell}}}  \label{eq:1ptInd-3}\\
& + \sum_{\substack{g_1+g_2+g_3=g\\ g_1,g_2\geq 1,\, g_3\geq 0\\ d_1+d_2+d_3  = 2g+m-3\\ d_1,d_2,d_3\geq 0\\ k_1+k_3=\ell,\, k_1,k_3\geq 1}}
(-1)^{d_1+\ell-1} \gamma^{g_1}_{d_1\mathop{|}k_1}\diamond 
\left(\xymatrix@C=14pt@R=0pt{
&*+[o][F-]{{g_2}}\ar@{-}[r]*{{}_{\,2}}\ar@{-}[l]*{{}_{1\,}}_{\psi^{d_2}} & } + (-1)^{d_2+1}
\xymatrix@C=14pt@R=0pt{
&*+[o][F-]{{g_2}}\ar@{-}[l]*{{}_{1\,}}\ar@{-}[r]*{{}_{\,2}}^{\,\,\psi^{d_2}} & } \right)
\diamond \tgamma^{g_3,m}_{d_3\mathop{|}k_3} \label{eq:1ptInd-4}\\
&=\;\boxed{\sum_{\substack{g_1+g_2=g,\, g_1\geq 1,\, g_2\geq 0 \\ d_1+d_2 = 2g+m-2,\, d_1,d_2\geq 0}} (-1)^{d_1+\ell}\gamma^{g_1}_{d_1\mathop{|}\ell+1} \diamond 
\vcenter{\xymatrix@C=14pt@R=5pt{
& & {}\ar@{.}@/^/[dd] \\ 
&*+[o][F-]{{g_2}}\ar@{-}[ru]*{{}_{\,2}}\ar@{-}[rd]*{{}_{\hspace{0.25cm} m+1}}\ar@{-}[l]*{{}_{1\,}}_{\psi^{d_2}} \ar@{-}[r] & \\
& &	}}}  \label{eq:1ptInd-5}\\
& \hphantom{=|}\; + \underline{\sum_{\substack{g_1+g_2=g\\ g_1\geq 1,\,g_2\geq 0 \\ d_1+d_2 = 2g+m-2\\ d_1,d_2\geq 0}}\hspace{-0.2cm}(-1)^{\ell+1}
\xymatrix@C=14pt@R=0pt{
&*+[o][F-]{{g_1}}\ar@{-}[r]*{{}_{\,2}}\ar@{-}[l]*{{\ }_{1\,}}_{\psi^{d_1}} & }\diamond \tgamma^{g_2,m}_{d_2\mathop{|}\ell+1}} + \hspace{-0.1cm}\underline{\underline{\sum_{\substack{g_1+g_2=g\\ g_1\geq 1,\,g_2\geq 0 \\ d_1+d_2 = 2g+m-2\\ d_1,d_2\geq 0}}\hspace{-0.2cm}(-1)^{d_1+\ell}
\xymatrix@C=14pt@R=0pt{
&*+[o][F-]{{g_1}}\ar@{-}[l]*{{}_{1\,}}\ar@{-}[r]*{{}_{\,2}}^{\,\,\psi^{d_1}} & } \diamond \tgamma^{g_2,m}_{d_2\mathop{|}\ell+1}}}  \label{eq:1ptInd-6}\\
& \hphantom{=|}\;+ \underbrace{\sum_{\substack{g_1+g_2+g_3=g,\, g_1,g_2\geq 1,\, g_3\geq 0\\ d_1+d_2+d_3  = 2g+m-3,\, d_1,d_2,d_3\geq 0\\ k_1+k_3=\ell+1,\, k_1,k_3\geq 1}}
(-1)^{d_1+\ell} \gamma^{g_1}_{d_1\mathop{|}k_1}\diamond 
\xymatrix@C=14pt@R=0pt{
&*+[o][F-]{{g_2}}\ar@{-}[r]*{{}_{\,2}}\ar@{-}[l]*{{}_{1\,}}_{\psi^{d_2}} & }\diamond \tgamma^{g_3,m}_{d_3\mathop{|}k_3}}_{=:\sum C_{k_1,k_3}=\underline{\underline{C_{1,\ell}}}+\sum_{k_1\ge 2}C_{k_1,k_3}}  \label{eq:1ptInd-7}\\
%\end{align}
%\begin{align}
&\hphantom{=|}\; + \underbrace{\sum_{\substack{g_1+g_2+g_3=g,\, g_1,g_2\geq 1,\, g_3\geq 0\\ d_1+d_2+d_3  = 2g+m-3,\, d_1,d_2,d_3\geq 0\\ k_1+k_3=\ell+1,\, k_1,k_3\geq 1}}
(-1)^{d_1+d_2+\ell+1} \gamma^{g_1}_{d_1\mathop{|}k_1}\diamond 
\xymatrix@C=14pt@R=0pt{
&*+[o][F-]{{g_2}}\ar@{-}[l]*{{}_{1\,}}\ar@{-}[r]*{{}_{\,2}}^{\,\,\psi^{d_2}} & } \diamond \tgamma^{g_3,m}_{d_3\mathop{|}k_3}}_{=:\sum D_{k_1,k_3}=\boxed{\scriptstyle{D_{\ell,1}}}+\sum_{k_3\ge 2}D_{k_1,k_3}}. \label{eq:1ptInd-8}
\end{align}

Note that the first summand in~\eqref{eq:1ptInd-2} is equal to the first summand in~\eqref{eq:1ptInd-6} by definition.

Consider the second summand in~\eqref{eq:1ptInd-2}. Since $d_1\leq 2g_1-1$, we have $d_2\geq 2g_2+m-1$; hence we can apply the Liu--Pandharipande relation~\eqref{eq:LP2}. We obtain
\begin{align}\label{eq:1pt-indstep-1line}
\sum_{\substack{g_1+g_2+g_3=g,\, g_1,g_2\geq 1,\, g_3\geq 0 \\ d_1+d_2+d_3 = 2g+m-3,\, d_1,d_2,d_3\geq 0}}(-1)^{d_1+d_2+\ell-1}
\gamma^{g_1}_{d_1\mathop{|}\ell} \diamond	
\xymatrix@C=14pt@R=0pt{
&*+[o][F-]{{g_2}}\ar@{-}[l]*{{}_{1\,}}\ar@{-}[r]*{{}_{\,2}}^{\,\,\psi^{d_2}} & } \diamond
\vcenter{\xymatrix@C=14pt@R=5pt{
& & {}\ar@{.}@/^/[dd] \\ 
&*+[o][F-]{{g_3}}\ar@{-}[ru]*{{}_{\,2}}\ar@{-}[rd]*{{}_{\hspace{0.25cm} m+1}}\ar@{-}[l]*{{}_{1\,}}_{\psi^{d_3}} \ar@{-}[r] & \\
& &	}}.
\end{align}
Note that either $d_1+d_2\leq 2(g_1+g_2)-2$ or $d_3\leq 2g_3+m-2$, but not both. Hence~\eqref{eq:1pt-indstep-1line} is equal to the sum of~\eqref{eq:1ptInd-5} and the $k_3=1$ term in~\eqref{eq:1ptInd-8}.

Consider~\eqref{eq:1ptInd-3}. Since $d_2\leq 2g_2+m-2$, we have $d_1\geq 2g_1$, hence we can apply the Liu--Pandharipande relation~\eqref{eq:LP1}. We obtain
\begin{align}\label{eq:1pt-indstep-secondline}
\sum_{\substack{g_1+g_2+g_3=g,\, g_1,g_2\geq 1,\, g_3\geq 0 \\ d_1+d_2+d_3 = 2g+m-3,\, d_1,d_2,d_3\geq 0}} 
(-1)^{\ell+d_1}
\xymatrix@C=14pt@R=0pt{
&*+[o][F-]{{g_1}}\ar@{-}[l]*{{}_{1\,}}\ar@{-}[r]*{{}_{\,2}}^{\,\,\psi^{d_1}} & } \diamond
\xymatrix@C=14pt@R=0pt{
&*+[o][F-]{{g_2}}\ar@{-}[r]*{{}_{\,2}}\ar@{-}[l]*{{}_{1\,}}_{\psi^{d_2}} & } \diamond
\tgamma^{g_3,m}_{d_3\mathop{|}\ell}.
\end{align}
Note that either $d_1\leq 2g_1-1$ or $d_2+d_3\leq 2(g_2+g_3)+m-3$, but not both. Hence~\eqref{eq:1pt-indstep-secondline} is equal to the sum of the $k_1=1$ term in~\eqref{eq:1ptInd-7} and the second summand in~\eqref{eq:1ptInd-6}.

Finally, consider~\eqref{eq:1ptInd-4}. Since $d_1\leq 2g_1-1$ and $d_3\leq 2g_3+m-2$, we have $d_2\geq 2 g_2$; hence we can apply the Liu--Pandharipande relation~\eqref{eq:LP1}. We obtain
\begin{align}\label{eq:1pt-indstep-thirdline}
\sum_{\substack{g_1+g_2+g_3+g_4=g\\g_1,g_2,g_3\geq 1,\, g_3\geq 0 \\ d_1+d_2+d_3+d_4 = 2g+m-4\\ d_1,d_2,d_3,d_4\geq 0\\ k_1+k_4=\ell,\, k_1,k_4\geq 1}} (-1)^{d_1+d_2+\ell-1}
\gamma^{g_1}_{d_1\mathop{|}k_1} \diamond
\xymatrix@C=14pt@R=0pt{
&*+[o][F-]{{g_2}}\ar@{-}[l]*{{}_{1\,}}\ar@{-}[r]*{{}_{\,2}}^{\,\,\psi^{d_3}} & } \diamond
\xymatrix@C=14pt@R=0pt{
&*+[o][F-]{{g_3}}\ar@{-}[r]*{{}_{\,2}}\ar@{-}[l]*{{}_{1\,}}_{\psi^{d_3}} & } \diamond
\tgamma^{g_4,m}_{d_4\mathop{|}k_4}.
\end{align}
Note that either $d_1+d_2\leq 2(g_1+g_2)-2$ or $d_3+d_4\leq 2(g_3+g_4)+m-3$, but not both. Hence~\eqref{eq:1pt-indstep-thirdline} is equal to the sum of the $k_1\ge 2$ terms in~\eqref{eq:1ptInd-7} and the $k_3\ge 2$ terms in~\eqref{eq:1ptInd-8}.

Summarizing the computations above, we see that the sum of \eqref{eq:1ptInd-2}, \eqref{eq:1ptInd-3}, and~\eqref{eq:1ptInd-4} is equal to the sum of~\eqref{eq:1ptInd-5}, \eqref{eq:1ptInd-6},~\eqref{eq:1ptInd-7}, and~\eqref{eq:1ptInd-8}. This proves the  induction step and completes the proof of Lemma~\ref{lem:1ptInductiveLemma}.
\end{proof}

\subsection{Proof of Conjecture~\ref{conjecture2} for $\boldsymbol{n=1}$}

We have to check that $B^1_{g,d}=A^1_{g,d}$ for any $g\ge 1$ and $d\ge 2g$. From the definition of the class $B^1_{g,d}$, it follows that 
$$
B^1_{g,d}=\psi_1^{d-2g}B^1_{g,2g}.
$$
On the other hand, setting $k:=d-2g$, we have
$$
A^1_{g,d}=\sum_{\substack{g_1+\cdots+g_k=g\\g_1,\ldots,g_k\ge 1}}\left(\prod_{i=0}^{k-1}\frac{g_{k-i}}{g_1+\cdots+g_{k-i}}\right)\lambda_{g_1}\DR_{g_1}(1,-1)\diamond \lambda_{g_2}\DR_{g_2}(1,-1)\diamond\cdots\diamond\lambda_{g_k}\DR_{g_k}(1,-1).
$$
Using the formula 
$$
\psi_1\lambda_g\DR_g(1,-1)=\sum_{\substack{g_1+g_2=g\\g_1,g_2\ge 1}}\frac{g_2}{g}\lambda_{g_1}\DR_{g_1}(1,-1)\diamond\lambda_{g_2}\DR_{g_2}(1,-1),
$$
\textit{cf.} \cite[Theorem~4]{BSSZ15}, it is easy to check by induction that
$$
A^1_{g,d}=\psi_1^{d-2g}A^1_{g,2g}.
$$
We conclude that it is sufficient to prove that 
\begin{gather}\label{eq:main equation for n1}
B^1_{g,2g} = \lambda_g\DR_g(1,-1).
\end{gather} 

As a preliminary step, in Section~\ref{subsection:a new relation}, we derive a new tautological relation using a variation of the method from the paper~\cite{LP11} by Liu and Pandharipande, and then in Section~\ref{subsection:proof of main equation}, we use it to prove Equation~\eqref{eq:main equation for n1}.

\subsubsection{A new relation via the Liu--Pandharipande method}\label{subsection:a new relation} 

We use the same notation as in Section~\ref{subsection:VanishingOnePoint}, with the following addendum to the notation. Let
$$
\xymatrix@C=17pt@R=0pt{
& *+[o][F-,]{{g}}\ar@{-}[l]*{{}_{1\,}}\ar@{-}[r]*{{}_{\,2}} & 
}
\coloneqq \lambda_g\DR_g(-1,1).
$$

\begin{theorem} 
For any $g\geq 1$ and $r\geq 0$, we have
\begin{gather}\label{eq:NewRelation}
\xymatrix@C=17pt@R=0pt{
& *+[o][F-,]{{g}}\ar@{-}[l]*{{}_{1\,}}_<<{\psi^{r}\,\,}\ar@{-}[r]*{{}_{\,2}} & 
} +
(-1)^{2g+r+1}\xymatrix@C=17pt@R=0pt{
& *+[o][F-]{{g}}\ar@{-}[r]*{{}_{\,2}}^<<<<{\psi^{2g+r}}\ar@{-}[l]*{{}_{1\,}} & 
} =
\sum_{\substack{g_1+g_2=g\\g_1,g_2\ge 1\\d_1+d_2=2g_1+r-1\\d_1,d_2\ge 0}}(-1)^{d_1}\xymatrix@C=12pt@R=0pt{
& *+[o][F-]{{g_1}}\ar@{-}[rr]^<<<{\psi^{d_1}}\ar@{-}[l]*{{}_{1\,}} & & *+[o][F-,]{{g_2}}\ar@{}[ll]_<<<{\psi^{d_2}}\ar@{-}[r]*{{}_{\,2}} & 
}.
\end{gather}
\end{theorem}
\begin{proof}
The proof is by localization in the moduli space of stable relative maps to $(\mbP^1,\infty)$, which we review in detail in the appendix.

We consider the moduli space $\oM_{g,n}(\mbP^1,\mu)$ of stable relative maps to $(\mbP^1,\infty)$ with $n=1$ and $\mu=1$. The source curves of these maps have two marked points, and we assume that the marked point corresponding to the only part of $\mu$ is labeled by $1$. Consider the $\mbC^*$-action on~$\mbP^1$ given~by
$$
t\cdot [x,y]:=[tx,y],\quad [x,y]\in\mbP^1,\; t\in\mbC^*,
$$
and the induced $\mbC^*$-action on $\oM_{g,1}(\mbP^1,1)$. Denote by  
$$
\pi\colon U\lra\oM_{g,1}(\mbP^1,1),\quad f\colon U\lra\mbP^1,
$$
the $\mbC^*$-equivariant universal curve and the universal map, respectively. Consider a lifting of the $\mbC^*$-action on $\mbP^1$ to the line bundle $\mcO_{\mbP^1}(-1)\to\mbP^1$ with fiber weights $-1$, $0$ over the fixed points $0,\infty\in\mbP^1$, respectively. The sheaf 
$$
B:=R^1\pi_*f^*(\mcO_{\mbP^1}(-1))\lra \oM_{g,1}(\mbP^1,1)
$$
is a $\mbC^*$-equivariant vector bundle of rank $g$. 

We consider the following $\mbC^*$-equivariant cohomology class on $\oM_{g,1}(\mbP^1,1)$:
$$
I_g:=\ev_2^*([0])e_{\mbC^*}(B)\in H^{2(g+1)}_{\mbC^*}(\oM_{g,1}(\mbP^1,1)),
$$
where $\ev_2\colon\oM_{g,1}(\mbP^1,1)\to\mbP^1$ is the evaluation map corresponding to the second marked point, $[0]\in H^2_{\mbC^*}(\mbP^1)$ is the $\mbC^*$-equivariant cohomology class dual to the point $0\in\mbP^1$, and by $e_{\mbC^*}(\cdot)$ we denote the $\mbC^*$-equivariant Euler class of a $\mbC^*$-equivariant vector bundle. Denote by 
$$
\epsilon\colon\oM_{g,1}(\mbP^1,1)\lra\oM_{g,2}
$$
the forgetful map, which is $\mbC^*$-equivariant with respect to the trivial $\mbC^*$-action on $\oM_{g,2}$. Consider the pushforward
\begin{gather}\label{eq:pushforward with zero constant term}
\epsilon_*(I_g\cap[\oM_{g,1}(\mbP^1,1)]^\vir)\in H^{\mbC^*}_{2g}(\oM_{g,2}),
\end{gather}
where $[\oM_{g,1}(\mbP^1,1)]^\vir$ is the $\mbC^*$-equivariant virtual fundamental class of the moduli space $\oM_{g,1}(\mbP^1,1)$. Since $H^{\mbC^*}_*(\oM_{g,2})=H_*(\oM_{g,2})\otimes_{\mbC}\mbC[u]$, where $u$ is the generator of the equivariant cohomology ring of a point, the class~\eqref{eq:pushforward with zero constant term} is a polynomial in $u$ with coefficients in the space~$H_*(\oM_{g,2})$. The class $I_g\cap[\oM_{g,1}(\mbP^1,1)]^\vir$ considered as an element of $H^{\mbC^*}_*(\oM_{g,1}(\mbP^1,1))\otimes_{\mbC[u]}\mbC[u,u^{-1}]$ can be computed using the localization formula~\eqref{eq:general localization formula}, and then using formulas~\eqref{eq:main localization formula 1} and~\eqref{eq:main localization formula 2}, one can get an explicit formula for the pushforward of this class to $\oM_{g,2}$ in terms of tautological classes. The fact that the coefficients of the negative powers of $u$ in the resulting expression vanish gives relations in the homology of $\oM_{g,2}$. Let us compute these relations explicitly. 

Consider the connected components of the $\mbC^*$-fixed point set $\oM_{g,1}(\mbP^1,1)^{\mbC^*}$. Apart from the component $\oF_0$ formed by the stable relative maps with the target equal to $\mbP^1$, we have components $\oF_\Gamma$ labeled by decorated bipartite graphs $\Gamma$ (see the appendix for details). Note that the decorated bipartite graphs $\Gamma$ such that $\oF_\Gamma\ne\emptyset$ have one of the following two forms:
\begin{align*}
&\Gamma_{g_1,g_2}:=\,\xymatrix@C=14pt@R=2pt{
& *+[o][F-]{{g_1}}\ar@{}[d]*{{}_0}\ar@{-}[rr]^<<<<<{\,\,\,\,1}\ar@{-}[l]*{{}_{2\,}} & & *+[o][F-]{{g_2}}\ar@{}[d]*{{}_\infty}\ar@{}[ll]\ar@{-}[r]*{{}_{\,1}} & \\
& & & &
}, && g_1\ge 0,\, g_2\ge 1,\\ 
&\tGamma_{g_1,g_2}:=\,
\vcenter{\xymatrix@C=14pt@R=2pt{
& & & \\
*+[o][F-]{{g_1}}\ar@{}[d]*{{}_0}\ar@{-}[rr]^<<<<<{\,\,\,\,1} & & *+[o][F-]{{g_2}}\ar@{}[d]*{{}_\infty}\ar@{}[ll]\ar@{-}[ru]*{{}_{\,1}}\ar@{-}[rd]*{{}_{\,2}} & \\
& & & }}
, && g_1,g_2\ge 0.
\end{align*}
Since $\iota_{\tGamma_{g_1,g_2}}^*(\ev_2^*([0]))=0$, we have $\iota_{\tGamma_{g_1,g_2}}^*(I_g)=0$, where we refer a reader to the appendix for a definition of spaces $\oM_{\Gamma}$ and natural surjective maps $\iota_0\colon\oM_{g,2}\to\oF_0$ and $\iota_\Gamma\colon\oM_{\Gamma}\to\oF_\Gamma$. 

Using Equations~\eqref{eq:covering of connected components} and~\eqref{eq:general localization formula}, we obtain the following equality in $H_*(\oM_{g,2})\otimes_{\mbC}\mbC[u,u^{-1}]$:
\begin{align*}
\epsilon_*(I_g\cap[\oM_{g,1}(\mbP^1,1)]^\vir=&
\;\epsilon_*\iota_{0*}\left(\iota^*_0\left(\frac{I_g}{e_{\mbC^*}(N_0^\vir)}\right)\cap[\oM_{g,2}]\right)\\
&+\sum_{\substack{g_1\ge 0,\,g_2\ge 1\\g_1+g_2=g}}\epsilon_*\iota_{\Gamma_{g_1,g_2}*}\left(\iota^*_{\Gamma_{g_1,g_2}}\left(\frac{I_g}{e_{\mbC^*}(N_{\Gamma_{g_1,g_2}}^\vir)}\right)\cap[\oM_{\Gamma_{g_1,g_2}}]^\vir\right),
\end{align*}
where $N^\vir_0$ and $N_{\Gamma_{g_1,g_2}}^\vir$ are the virtual normal bundles to the connected components~$\oF_0$ and~$\oF_{\Gamma_{g_1,g_2}}$, respectively. Let us now compute explicitly the contribution of each connected component.

Consider the component $\oF_0$. Note that we have an isomorphism $\iota_0^* B\cong\mbE^\vee_g\otimes\mbC_{-1}$ as $\mbC^*$-equivariant vector bundles, where by $\mbC_a$ we denote the trivial line bundle with $t\in\mbC^*$ acting on it by the multiplication by $t^a$. Therefore,
$$
\iota_0^* I_g=u\Lambda_g^\vee(-u),
$$
where
$$
\Lambda_g^\vee(u):=\sum_{i=0}^g(-1)^i\lambda_i u^{g-i}.
$$
Using Equation~\eqref{eq:main localization formula 1}, we obtain
\begin{align}
&\iota^*_0\left(\frac{I_g}{e_{\mbC^*}(N_0^\vir)}\right)=\frac{\Lambda^\vee_g(-u)\Lambda^\vee_g(u)}{u-\psi_1}=(-1)^g\frac{u^{2g-1}}{1-\frac{\psi_1}{u}},\notag\\
&\epsilon_*\iota_{0*}\left(\iota^*_0\left(\frac{I_g}{e_{\mbC^*}(N_0^\vir)}\right)\cap[\oM_{g,2}]\right)=(-1)^g\sum_{i\ge 0}u^{2g-1-i}\,
\xymatrix@C=17pt@R=0pt{
& *+[o][F-]{{g}}\ar@{-}[r]*{{}_{\,1}}^<<{\,\,\psi^{i}}\ar@{-}[l]*{{}_{2\,}} & }.
\label{eq:proof of DRLiuPan relation,1}
\end{align}

Now consider the component $\oF_{\Gamma_{0,g}}$ of the fixed point set. We have $\oM_{\Gamma_{0,g}}=\oM^\sim_{g,0}(\mbP^1,1,1)$, where $\oM^\sim_{g,0}(\mbP^1,1,1)$ is the moduli space of relative stable maps to rubber $(\mbP^1,0,\infty)$ and $\iota_{\Gamma_{0,g}}^*B\cong\mbE^\vee_g$ as $\mbC^*$-equivariant vector bundles, where we emphasize that $\mbC^*$ acts trivially on $\mbE^\vee_g$. Therefore,
$$
\iota_{\Gamma_{0,g}}^* I_g=u(-1)^g\lambda_g.
$$
Using Equation~\eqref{eq:main localization formula 2}, we obtain
\begin{gather*}
\iota^*_{\Gamma_{0,g}}\left(\frac{I_g}{e_{\mbC^*}(N_{\Gamma_{0,g}}^\vir)}\right)=\frac{(-1)^g\lambda_g}{-u-\tpsi_0}=(-1)^{g-1}\frac{u^{-1}\lambda_g}{1+\frac{\tpsi_0}{u}},
\end{gather*}
where $\tpsi_0$ is the first Chern class of the cotangent line bundle over $\oM^\sim_{g,0}(\mbP^1,1,1)$ corresponding to the point~$0$ in the target curve of a stable relative map (see  the appendix for details). Noting that $\tpsi_0=(\epsilon\circ\iota_{\Gamma_{0,g}})^*\psi_2$, we obtain
\begin{gather}\label{eq:proof of DRLiuPan relation,2}
\epsilon_*\iota_{\Gamma_{0,g}*}\left(\iota^*_{\Gamma_{0,g}}\left(\frac{I_g}{e_{\mbC^*}(N_{\Gamma_{0,g}}^\vir)}\right)\cap[\oM_{\Gamma_{0,g}}]^\vir\right)=(-1)^g\sum_{i\ge 0}u^{-i-1}(-1)^{i-1}\,
\xymatrix@C=17pt@R=0pt{
& *+[o][F-,]{{g}}\ar@{-}[r]*{{}_{\,1}}\ar@{-}[l]*{{}_{2\,}}_<<{\psi^i\,\,\,} & }.
\end{gather}

Consider the component $\oF_{\Gamma_{g_1,g_2}}$ with $g_1,g_2\ge 1$. We have $\oM_{\Gamma_{g_1,g_2}}=\oM_{g_1,2}\times\oM^\sim_{g_2,0}(\mbP^1,1,1)$ and $\iota_{\Gamma_{g_1,g_2}}^*B\cong\mbE^\vee_{g_1}\otimes\mbC_{-1}\oplus \mbE^\vee_{g_2}$ as $\mbC^*$-equivariant vector bundles. Therefore,
$$
\iota_{\Gamma_{g_1,g_2}}^* I_g=u\Lambda^\vee_{g_1}(-u)(-1)^{g_2}\lambda_{g_2}.
$$
Using Equation~\eqref{eq:main localization formula 2}, we obtain
\begin{align}
&\iota^*_{\Gamma_{g_1,g_2}}\left(\frac{I_g}{e_{\mbC^*}(N_{\Gamma_{g_1,g_2}}^\vir)}\right)=\frac{\Lambda^\vee_{g_1}(-u)(-1)^{g_2}\lambda_{g_2}\Lambda_{g_1}^\vee(u)}{(u-\psi_1)(-u-\tpsi_0)}=(-1)^{g-1}\frac{u^{2g_1-2}\lambda_{g_2}}{(1-\frac{\psi_1}{u})(1+\frac{\tpsi_0}{u})},\notag\\
&\epsilon_*\iota_{\Gamma_{g_1,g_2}*}\left(\iota^*_{\Gamma_{g_1,g_2}}\left(\frac{I_g}{e_{\mbC^*}(N_{\Gamma_{g_1,g_2}}^\vir)}\right)\cap[\oM_{\Gamma_{g_1,g_2}}]^\vir\right)\notag\\
&\hspace{3.5cm}=(-1)^g\sum_{d_1,d_2\ge 0}u^{2g_1-2-d_1-d_2}(-1)^{d_2-1}\,
\xymatrix@C=14pt@R=0pt{
& *+[o][F-]{{g_1}}\ar@{-}[rr]^<<<{\psi^{d_1}}\ar@{-}[l]*{{}_{2\,}} & & *+[o][F-,]{{g_2}}\ar@{}[ll]_<<<{\psi^{d_2}}\ar@{-}[r]*{{}_{\,1}} & 
}.\label{eq:proof of DRLiuPan relation,3}
\end{align}
Summing the coefficients of $u^{-r-1}$ in the expressions on the right-hand side of Equations~\eqref{eq:proof of DRLiuPan relation,1},~\eqref{eq:proof of DRLiuPan relation,2}, and~\eqref{eq:proof of DRLiuPan relation,3}, multiplying the sum by $(-1)^{g+r}$, applying the map $\perm_{1,1}^*$, and equating the result to $0$, we obtain exactly the desired relation.
\end{proof}

\subsubsection{Proof of Equation~\eqref{eq:main equation for n1}}\label{subsection:proof of main equation}

Let us prove by induction that
\begin{gather}\label{eq:extended new relation} 
\xymatrix@C=15pt@R=0pt{
&*+[o][F-,]{g}\ar@{-}[r]*{{}_{\,2}}\ar@{-}[l]*{{}_{1\,}} & }=\sum_{k=1}^{\ell}(-1)^{k+1}\perm^*_{1,1}\Gamma^{g,1}_{2g\mathop{|}k}+\hspace{-0.15cm}\sum_{\substack{g_1+g_2=g\\g_1,g_2\ge 1\\d_1+d_2=2g_1-1\\d_1,d_2\ge 0}}\hspace{-0.15cm}
(-1)^{d_2+\ell}\gamma^{g_1}_{d_1\mathop{|}\ell}\diamond
\xymatrix@C=15pt@R=0pt{
&*+[o][F-,]{g_2}\ar@{-}[r]*{{}_{\,2}}\ar@{-}[l]*{{}_{1\,}}_{\psi^{d_2}\,\,\,\,} & },\quad\ell\ge 1.
\end{gather}
Indeed, for $\ell=1$, this is exactly Equation~\eqref{eq:NewRelation} when $r=0$. For the induction step, we have to check that
\begin{align*}
&\sum_{\substack{g_1+g_2=g,\,g_1,g_2\ge 1\\d_1+d_2=2g_1-1,\,d_1,d_2\ge 0}}
(-1)^{d_2+\ell}\gamma^{g_1}_{d_1\mathop{|}\ell}\diamond
\xymatrix@C=15pt@R=0pt{
&*+[o][F-,]{g_2}\ar@{-}[r]*{{}_{\,2}}\ar@{-}[l]*{{}_{1\,}}_{\psi^{d_2}\,\,\,\,} & }\\
&\hspace{1.2cm}=(-1)^\ell\perm^*_{1,1}\Gamma^{g,1}_{2g\mathop{|}\ell+1}+\sum_{\substack{g_1+g_2=g,\,g_1,g_2\ge 1\\d_1+d_2=2g_1-1,\,d_1,d_2\ge 0}}
(-1)^{d_2+\ell+1}\gamma^{g_1}_{d_1\mathop{|}\ell+1}\diamond
\xymatrix@C=15pt@R=0pt{
&*+[o][F-,]{g_2}\ar@{-}[r]*{{}_{\,2}}\ar@{-}[l]*{{}_{1\,}}_{\psi^{d_2}\,\,\,\,} & },\quad \ell\ge 1,
\end{align*}
or equivalently
\begin{align*}
&\sum_{\substack{g_1+g_2=g\\g_1,g_2\ge 1\\d_1+d_2=2g_1-1\\d_1,d_2\ge 0}}
(-1)^{d_2}\gamma^{g_1}_{d_1\mathop{|}\ell}\diamond
\xymatrix@C=15pt@R=0pt{
&*+[o][F-,]{g_2}\ar@{-}[r]*{{}_{\,2}}\ar@{-}[l]*{{}_{1\,}}_{\psi^{d_2}\,\,\,\,} & }\\
&\hspace{0.7cm}=\sum_{\substack{g_1+g_2=g\\g_1,g_2\ge 1\\d_1+d_2=2g-1\\d_1,d_2\ge 0}}\hspace{-0.2cm}
\gamma^{g_1}_{d_1\mathop{|}\ell}\diamond
\xymatrix@C=15pt@R=0pt{
&*+[o][F-]{g_2}\ar@{-}[r]*{{}_{\,2}}^{\,\,\psi^{d_2}}\ar@{-}[l]*{{}_{1\,}} & }
+\hspace{-0.25cm}\sum_{\substack{g_1+g_2+g_3=g\\g_1,g_2,g_3\ge 1\\d_1+d_2+d_3=2(g_1+g_2)-2\\d_1,d_2,d_3\ge 0}}\hspace{-0.25cm}
(-1)^{d_3+1}\gamma^{g_1}_{d_1\mathop{|}\ell}\diamond
\xymatrix@C=15pt@R=0pt{
&*+[o][F-]{g_2}\ar@{-}[r]*{{}_{\,2}}^{\,\,\psi^{d_2}}\ar@{-}[l]*{{}_{1\,}} & }\diamond
\xymatrix@C=15pt@R=0pt{
&*+[o][F-,]{g_3}\ar@{-}[r]*{{}_{\,2}}\ar@{-}[l]*{{}_{1\,}}_{\psi^{d_3}\,\,\,\,} & },
\end{align*}
which follows from Equation~\eqref{eq:NewRelation}.

Since $\gamma^g_{d\mathop{|}k}=0$ if $k>g$, substituting $\ell=g$ in Equation~\eqref{eq:extended new relation}, we obtain
$$
\lambda_g\DR_g(1,-1)=\perm^*_{1,1}B^1_{g,2g},
$$
which, because of the obvious property $\perm^*_{1,1}\lambda_g\DR_g(1,-1)=\lambda_g\DR_g(1,-1)$, implies Equation~\eqref{eq:main equation for n1}.

%%%%%%%%%%%%%%%%%%%%%%%%%%%%%%%%%%%%%%%%%%%%%%%%%%%%%%%%%%%%%%%%%

\subsection{Proof of Conjecture~\ref{conjecture3} for $\boldsymbol{n=1}$}\label{section:conjecture3 for n1}

As  is proved in~\cite[Theorem~2.4]{BHIS21}, for $n=1$, Conjecture~\ref{conjecture3} follows from Conjecture~\ref{conjecture2}.

%%%%%%%%%%%%%%%%%%%%%%%%%%%%%%%%%%%%%%%%%%%%%%%%%%%%%%%%%%%%%%%%%
%%%%%%%%%%%%%%%%%%%%%%%%%%%%%%%%%%%%%%%%%%%%%%%%%%%%%%%%%%%%%%%%%

\section{Proof of Theorem~\ref{theorem:main2}}\label{section:genus 0}

Consider an arbitrary CohFT $c_{g,n}$. In~\cite{BPS12}, the authors proved that
$$
\frac{\d^2\mcF_0}{\d t^\alpha_a\d t^\beta_b}=\left.\frac{\d^2\mcF_0}{\d t^\alpha_a\d t^\beta_b}\right|_{t^\gamma_c=\delta_{c,0}v^{\top;\gamma}},\quad 1\le\alpha,\beta\le N,\,a,b\ge 0,
$$
where $v^{\top;\gamma}:=\eta^{\gamma\mu}\frac{\d^2\mcF_0}{\d t^\mu_0\d t^\un_0}$. Consider formal variables $v^1,\ldots,v^N$ and the associated algebra of differential polynomials $\cA_v$. Define 
$$
\Omega^{[0]}_{\alpha,a;\beta,b}:=\left.\frac{\d^2\mcF_0}{\d t^\alpha_a\d t^\beta_b}\right|_{t^\gamma_c=\delta_{c,0}v^\gamma},\quad P^{[0];\alpha}_{\beta,b}:=\eta^{\alpha\mu}\Omega^{[0]}_{\mu,0;\beta,b},\quad \oP^{[0]}_{\beta,b}:=(P^{[0];1}_{\beta,b},\ldots,P^{[0];N}_{\beta,b}).
$$
We see that for any $m\ge 2$, we have
$$
\frac{\d^m\mcF_0}{\d t^{\alpha_1}_{d_1}\cdots\d t^{\alpha_m}_{d_m}}=\left.\left(D_{\d_x\oP^{[0]}_{\alpha_1,d_1}}\cdots D_{\d_x\oP^{[0]}_{\alpha_{m-2},d_{m-2}}}\Omega^{[0]}_{\alpha_{m-1},d_{m-1};\alpha_m,d_m}\right)\right|_{v^\gamma_k=v^{\top;\gamma}_k},
$$
which by Theorem~\ref{theorem:reducing transformation for F-CohFT} implies that
$$
\int_{\oM_{0,m+n}}B^m_{0,(d_1,\ldots,d_n)}c_{0,m+n}(\otimes_{j=1}^{m+n} v_j)=0
$$
for any $m\ge 2$, $n\ge 1$, $\sum d_i\ge 2g+m-1$, and $v_j\in V$. Using the nondegeneracy of the Poincar\'e pairing in cohomology, we see that it is sufficient to prove the following statement.

\begin{proposition}\label{prop61}
Cohomology classes $\omega\in H^*(\oM_{0,n})$ of the form 
\begin{gather}\label{eq:omegaform}
\omega=c_{0,n}(\otimes_{j=1}^n v_j),\quad\text{where $c_{g,n}$ is an arbitrary CohFT and $v_j\in V$},
\end{gather}
linearly span the cohomology space $H^*(\oM_{0,n})$.
\end{proposition}
\begin{proof}
For $I\subset\br{n}$, $2\le |I|\le n-2$, denote by $\delta_I\in H^2(\oM_{0,n})$ the cohomology class that is Poincar\'e dual to the fundamental class of the closure in $\oM_{0,n}$ of the subvariety formed by curves with exactly one node and marked points $x_i$ with $i\in I$ on one bubble, and marked points~$x_j$ with $j\in\br{n}\backslash I$ on the other bubble. The cohomology algebra of the moduli space~$\oM_{0,n}$ is generated by the classes $\delta_I$ (see, \textit{e.g.},~\cite{Man99}). Using the tensor product of CohFTs, we see that it is sufficient to prove the following special case of the proposition.

\begin{lemma}
Any cohomology class $\delta_I\in H^2(\oM_{0,n})$ can be obtained as a linear combination of cohomology classes of the form~\eqref{eq:omegaform}.
\end{lemma}
\begin{proof}
Consider the CohFT $c_{g,n}^{(0)}$ defined as follows:
\begin{itemize}
\item $V:=\mbC^2$ with standard basis $e_1,e_2$,

\item $e:=e_1+e_2$,

\item $\eta_{\alpha\beta}:=\delta_{\alpha\beta}$,

\item $c_{g,n}^{(0)}(\otimes_{i=1}^n e_{\alpha_i}):=
\begin{cases}
1&\text{if $\alpha_1=\ldots=\alpha_n$},\\
0&\text{otherwise}.
\end{cases}$
\end{itemize}
For $t\in\mbC$, define 
$$
R(z):=\exp(r_1 z),\quad\text{where } r_1:=
\begin{pmatrix}
0 & t\\
t & 0
\end{pmatrix}.
$$
Let us apply this $R$-matrix to the CohFT $c_{g,n}^{(0)}$ and denote the resulting CohFT (with unit) by~$c^{(t)}_{g,n}$ (we follow the approach from~\cite[Section~2]{PPZ15}). It is clear that $c^{(t)}_{g,n}(\otimes_{i=1}^n v_i)$ depends polynomially on $t$ and, moreover,
$$
\Coef_{t^k}c^{(t)}_{g,n}(\otimes_{i=1}^n v_i)\in H^{2k}(\oM_{g,n}).
$$
For any $0\le k\le 3g-3+n$, this coefficient can be obtained as a linear combination of the classes $c^{(t_j)}_{g,n}(\otimes_{i=1}^n v_i)$, $1\le j\le 3g-2+n$, where $t_1,\ldots,t_{3g-2+n}$ are arbitrary pairwise distinct complex numbers. Therefore, it is sufficient to check that 
\begin{gather}\label{eq:delta and c}
\delta_I=\Coef_t c^{(t)}_{0,n}(\otimes_{i=1}^n v_i)\quad\text{for some vectors $v_1,\ldots,v_n\in V$}.
\end{gather}
 
Let us prove that Equation~\eqref{eq:delta and c} is true for 
$$
v_i=
\begin{cases}
e_1&\text{if $i\in I$},\\
e_2&\text{otherwise}.
\end{cases}
$$
Indeed, we compute
\begin{align}
\Coef_t c^{(t)}_{0,n}(\otimes_{i=1}^n v_i)=&\;\frac{1}{2}\sum_{\substack{\tI\subset\br{n}\\2\le|\tI|\le n-2}}\gl_*\left(c^{(0)}_{0,|\tI|+1}(\otimes_{i\in\tI}v_i\otimes e_\mu)\delta^{\mu,3-\nu}\otimes c^{(0)}_{0,n-|\tI|+1}(\otimes_{j\in\br{n}\backslash\tI}v_j\otimes e_\nu)\right)\label{eq:R-action}\\
&-\sum_{i=1}^n\psi_i c^{(0)}_{0,n}(v_1\otimes\cdots\otimes r_1(v_i)\otimes\cdots\otimes v_n)\notag\\
&+\pi_*\left(\psi_{n+1}^2 c^{(0)}_{0,n+1}(v_1\otimes\cdots\otimes v_n\otimes r_1(e))\right),\notag
\end{align}
where $\gl\colon\oM_{0,|\tI|+1}\otimes\oM_{0,n-|\tI|+1}\to\oM_{0,n}$ is the obvious gluing map and $\pi\colon\oM_{0,n+1}\to\oM_{0,n}$ forgets the last marked point. From the definition of the CohFT $c_{g,n}^{(0)}$, it is easy to see that the first summand on the right-hand side of Equation~\eqref{eq:R-action} is exactly the class $\delta_I$, while the second and the third summands vanish. This completes the proof of the lemma.
\end{proof}

This concludes the proof of Proposition~\ref{prop61}. 
\end{proof}

%%%%%%%%%%%%%%%%%%%%%%%%%%%%%%%%%%%%%%%%%%%%%%%%%%%%%%%%%%%%%%%%%
%%%%%%%%%%%%%%%%%%%%%%%%%%%%%%%%%%%%%%%%%%%%%%%%%%%%%%%%%%%%%%%%%

\section{A reduction of the system of relations in the case \texorpdfstring{$\boldsymbol{m\ge 2}$}{m at least 2}}\label{section:reduction} 

Recall that in Section~\ref{section:equivalent formulation}, we introduced the polynomial $P_{g,n,m}(x_1,\ldots,x_n)\in R^*(\oM_{g,n+m})[x_1,\ldots,x_n]$ and considered the following cohomology classes:
\begin{align}
\tB^m_{g,\od}&=\Coef_{x_1^{d_1}\cdots x_n^{d_n}}P_{g,n,m}\label{eq:formula1 for tB}\\
&=\sum_{T\in\SRT^{(b,nd)}_{g,n,m;\circ}}(-1)^{|E(T)|}\ee_*[T,\od]\in R^{\sum d_i}(\oM_{g,n+m});\notag
\end{align}
see Equation~\eqref{eq:definition of tB} and Lemma~\ref{lemma:tB and coefficient of P}. By Theorem~\ref{theorem:equivalent relations for m2}, an equivalent way to state Conjecture~\ref{conjecture1} is to conjecture that
\begin{gather}\label{eq:systemofrelations}
\tB_{g,\od}^m = 0 \quad\text{for any $g\geq 0$, $n\geq 1$, $m\geq 2$, $\od\in\mbZ_{\ge 0}^n$ satisfying $d_1+\dots+d_n\geq 2g+m-1$}.
\end{gather}
One more equivalent reformulation reduces the number of relations one has to check.

\begin{theorem}\label{theorem:reductionSmall}
The system of relations~\eqref{eq:systemofrelations}, as a whole, is equivalent to the following one:
\begin{align}\label{eq:smallsystemofrelations}
\tB_{g,\od}^m = 0 \quad \text{for any $g\geq 0$, $n\geq 1$, $m\geq 2$, $d_i\ge 1$ satisfying $d_1+\dots+d_n=2g+m-1$}.
\end{align}
\end{theorem}

\begin{proof} 
For a tree $T\in\SRT_{g,n,m}$ and $h\in\tH^{em}_+(T)$, let us use the more detailed notation $I_{h,T}$ instead of $I_h$.

We start with the following lemma. 

\begin{lemma}\label{lem:string} 
We have $\tB^m_{g,(d_1,\dots,d_{n},0)} = \pi^*\tB^m_{g,(d_1,\dots,d_{n})}$, where $\pi\colon \oM_{g,n+1+m}\to \oM_{g,n+m}$ forgets the marked point number $n+1$ \textup{(}and shifts the numbers of the last $m$ marked points\textup{)}. 
\end{lemma}

\begin{proof} 
Up to a relabeling of the marked points or legs in the stable rooted trees, it is convenient to assume that the marked point that we forget under the projection $\pi$ is labeled by $0$ and that the labels of all other $n+m$ points are preserved by $\pi$. With this new convention, we have to prove that $\tB^m_{g,(0,d_1,\dots,d_{n})} = \pi^*\tB^m_{g,(d_1,\dots,d_{n})}$ or equivalently
$$
P_{g,n+1,m}(0,x_1,\ldots,x_n)=\pi^*P_{g,n,m}(x_1,\ldots,x_n).
$$
	
We use the formula~\eqref{eq:formula for P} for the polynomial $P_{g,n,m}$ and consider the contribution of a pair $(T,p)$, $T\in\SRT_{g,n,m}$, $p\colon\tH_+^{em}(T)\to\mbZ_{\ge 0}$. From the pullback formula for the $\psi$-classes, we have
\begin{align}\label{eq:pullbackTree}
\pi^*\xi_{T*}\left(\prod_{h \in \tH_+^{em}(T)}\hspace{-0.1cm}\psi_h^{p(h)}\right) 
= \hspace{-0.1cm}\sum_{v\in V(T)}\hspace{-0.1cm} \xi_{T_v*}\left(\prod_{h \in \tH_+^{em}(T_v)} \hspace{-0.1cm}\psi_h^{p_v(h)}\right) 
- \hspace{-0.2cm}\sum_{f\in \tH_+^{em}(T)}\hspace{-0.1cm}\xi_{T_f*}\left(\prod_{h \in \tH_+^{em}(T_f)}\hspace{-0.1cm}\psi_h^{p_f(h)}\right).
\end{align}
Here $T_v$ is the tree $T$ with an extra leg $\sigma_0$ attached to the vertex $v$. Naturally, $\tH^{em}_+(T_v)=\tH^{em}_+(T)\sqcup\{\sigma_0\}$, and we define $p_v(\sigma_0):=0$, and $p_v(h):=p(h)$ for any $h\in \tH^{em}_+(T)$. The tree $T_f$ is obtained as follows (see also \eqref{eq:ExampleTreeEdge}):
\begin{itemize}
\item If $f$ is a leg $\sigma_i$, $1\le i\le n$, then we attach to it a new vertex of genus $0$ and attach to this new vertex the leg $\sigma_i$ and an extra leg $\sigma_0$.

\item If $f$ is a part of an edge, then we break this edge into two half-edges, insert a vertex of genus $0$ between them, and attach an extra leg $\sigma_0$ to this vertex.
\end{itemize}
We have a natural inclusion $\tH^{em}_+(T)\subset\tH^{em}_+(T_f)$ with $|\tH^{em}_+(T_f)\backslash \tH^{em}_+(T)|=2$. The two half-edges from $\tH^{em}_+(T_f)\backslash \tH^{em}_+(T)$ are attached to the new vertex of genus $0$: one of them is $\sigma_0$, and we denote the other one by $\tf$. By definition, the function $p_f$ coincides with $p$ on $\tH^{em}_+(T)\backslash\{f\}$, $p_f(f):=p(f)-1$, and $p_f(\sigma_0)=p_f(\tf):=0$.
\begin{align} \label{eq:ExampleTreeEdge}
\vcenter{\xymatrix@C=15pt@R=5pt{
& &   \\ 
*+[o][F-]{{h}}\ar@{-}[rru]^<<<<{\psi^{p(f)}} & &
}} 
\quad\leadsto\quad
\vcenter{\xymatrix@C=10pt@R=5pt{
		& & & & &  \\ 
& & & *+[o][F-]{{0}}\ar@{-}[rru]^<<<{\psi^{0}} \ar@{-}[rrd]*{{}_{\,\,\sigma_0}}_<<<{\psi^0} & & \\ 
*+[o][F-]{{h}}\ar@{-}[rrru]^<<<<<{\psi^{p(f)-1}} & & & & & 
}} 
\end{align}
This way, we see that taking the sum over all pairs $(T,p)$ contributing to $P_{g,n,m}(x_1,\ldots,x_n)$ and using~\eqref{eq:pullbackTree}, we can list all pairs contributing to $P_{g,n+1,m}(0,x_1,\ldots,x_n)$, and moreover the signs are exactly the ones we have to use in the formula for $P_{g,n+1,m}(0,x_1,\ldots,x_n)$.

So we just have to check the coefficients. In the case of the pairs $(T_v,p_v)$, the equality of the coefficients is obvious because $\left.x_{I_{h,T_v}}^{p_v(h)+1}\right|_{x_0=0}=x_{I_{h,T}}^{p(h)+1}$ for any $h\in \tH^{em}_+(T)=\tH^{em}_+(T_v)\backslash\{\sigma_0\}$. In the case of the pairs $(T_f,p_f)$, we have
$\left.x_{I_{h,T_f}}^{p_f(h)+1}\right|_{x_0=0}=x_{I_{h,T}}^{p(h)+1}$ for any $h\in \tH^{em}_+(T)\backslash\{f\}\subset \tH^{em}_+(T_f)$. Regarding the remaining half-edges, note that $x_{I_{\tf,T_f}}=\left.x_{I_{f,T_f}}\right|_{x_0=0}=x_{I_{f,T}}$, which implies that $x_{I_{f,T}}^{p(f)+1}=\left.x_{I_{f,T_f}}^{p_f(f)+1} x_{I_{\tf,T_f}}^{p_f(\tf)+1}\right|_{x_0=0}$. Thus, the coefficients match for this type of trees as well, and we conclude that $P_{g,n+1,m}(0,x_1,\ldots,x_n)=\pi^*P_{g,n,m}(x_1,\ldots,x_n)$. 
\end{proof}

\begin{lemma} \label{lem:reductionpsi} 
Assume that the relations~\eqref{eq:systemofrelations} hold for all triples $(g',n',m')$ with either $g'<g$, $n'\leq n$ or $g'\leq g$, $n'<n$, for $m'=m$ and $m'=2$. Then the difference 
\begin{align}\label{eq:difference}
\tB^m_{g,(d_1,\ldots,d_{i-1},d_i+1,d_{i+1},\ldots,d_n)} - \psi_i \tB^m_{g,(d_1,\ldots,d_n)}
\end{align}	
is equal to zero for any $d_1,\ldots,d_n\ge 0$ such that $d_1+\cdots+d_n\geq 2g+m-1$. 
\end{lemma}

\begin{proof}
Using the formula~\eqref{eq:formula1 for tB}, we see that the pairs $(T,p)$, $T\in\SRT_{g,n,m}$, $p\colon\tH^{em}_+(T)\to\mbZ_{\ge 0}$, contributing to the difference $\tB^m_{g,(d_1,\dots,d_i+1,\dots,d_n)}-\psi_i \tB^m_{g,(d_1,\dots,d_n)}$ satisfy the property $p(\sigma_i)=0$; these graphs come from the first summand. Obviously, $\sigma_i$ cannot be attached to the root vertex (otherwise, the coefficient of $x_i^{d_i+1}$ is equal to zero). So $\sigma_i$ is attached to the vertex that is the first descendant of some $f\in \tH_+^{em}(T)$. 

We cut the edge that contains $f$ and obtain two trees, $T_1$ and $T_2$, of genera $g_1$ and $g_2$, respectively. The regular legs of $T_1$ are $\sigma_j$, $j\not\in I_{f,T}$, and $f$, the root vertex is the original root vertex, and the frozen legs are the  frozen legs of~$T$. The root vertex of $T_2$ is the one where~$\sigma_i$ is attached, the regular legs are $\sigma_j$, $j\in I_{f,T} \backslash\{i\}$, and the frozen legs are $\sigma_i$ and the half-edge $f'$ that formed an edge with $f$ in $T$. So $T_1\in\SRT_{g_1,n-|I_{f,T}|+1,m}$ and $T_2\in\SRT_{g_2,|I_{f,T}|-1,2}$. We have natural inclusions $\tH^{em}_+(T_1)\subset\tH^{em}_+(T)$ and $\tH^{em}_+(T_2)\subset\tH^{em}_+(T)$, and we denote by $p^{(1)}$ and $p^{(2)}$, respectively, the restrictions of the function~$p$ to these subsets. 

\begin{example}
Consider, for instance, the following tree $T$, which is cut into $T_1$ and $T_2$ at the dashed place on the picture: 
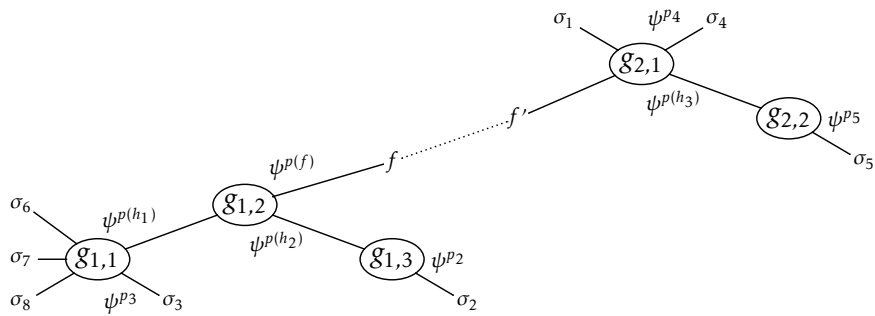
\begin{gather*}
\vcenter{\xymatrix@C=12pt@R=5pt{
& & & & & & & & & & & & & &  \\
& & & & & & & & & & *+[o][F-]{{g_{2,1}}} \ar@{-}[lld]*{{}_{\,f'}}\ar@{-}[lu]*{{}_{\sigma_1}\,\,}	\ar@{-}[ru]*{{}_{\,\,\sigma_4}}^<<<{\psi^{p_4}}\ar@{-}[rrd]_<<<{\psi^{p(h_3)}} \\
& & & & & & & & & & & &  *+[o][F-]{{g_{2,2}}} \ar@{-}[rd]*{{}_{\,\,\sigma_5}}^<<<{\psi^{p_5}} & &\\
& & & & & & \ar@{.}[rru] & & & & & & & & & \\
& & & & *+[o][F-]{{g_{1,2}}}\ar@{-}[rru]*{{}_{\,f}}^<<<<{\psi^{p(f)}} \ar@{-}[rrd]_<<<{\psi^{p(h_2)}}& \\ 
& &*+[o][F-]{{g_{1,1}}}\ar@{-}[rru]^<<<{\psi^{p(h_1)}}\ar@{-}[rd]*{{}_{\,\,\sigma_3}}_<<<{\psi^{p_3}}\ar@{-}[ul]*{{}_{\sigma_6\,\,}}\ar@{-}[l]*{{}_{\sigma_7\,\,}}\ar@{-}[dl]*{{}_{\sigma_8\,\,}} & &  & & *+[o][F-]{{g_{1,3}}}\ar@{-}[rd]*{{}_{\,\,\sigma_2}}^<<<{\psi^{p_2}} \\
& & & & & & & & }}
\end{gather*}
Here we assume  $n=5$, $m=3$, $g=g_1+g_2$, where $g_1=g_{1,1}+g_{1,2}+g_{1,3}$ and $g_2=g_{2,1}+g_{2,2}$, $i=1$. The tree $T$ is cut into two trees at the edge $(f,f')$, $I_{f,T} = \{1,4,5\}$, and we set $\{h_1,h_2,h_3,f\}:=H^e_+(T)$.
\end{example}
Note that if $I_{f,T}\ne \{i\}$, then the number of regular legs both in $T_1$ and in $T_2$ is less than $n$. If $I_{f,T}=\{i\}$, then the genus of $T_1$ is less than $g$.

We can express the contribution of the pair $(T,p)$ to the difference~\eqref{eq:difference} as follows:
\begin{align}
\Coef_{x_1^{d_1}\cdots x_i^{d_i+1}\cdots x_n^{d_n}}\gl_*&\left[(-1)\hspace{-0.2cm}\prod_{i\in\br{n}\backslash I_{f,T}} \hspace{-0.25cm}x_i^{-1}\cdot(-1)^{|E(T_1)|} \xi_{T_1*}\left(\prod_{h \in \tH_+^{em}(T_1)}\hspace{-0.25cm}\psi_h^{p^{(1)}(h)}\right)\hspace{0cm} \prod_{h \in\tH_+^{em}(T_1)}\hspace{-0.25cm}x_{I_{h,T}}^{p^{(1)}(h)+1}\right.\label{eq:decomposition of T}\\
&\hspace{0.2cm}\left.\otimes\hspace{-0.25cm}\prod_{j\in I_{f,T}\backslash\{i\}}\hspace{-0.25cm}x_j^{-1}\cdot(-1)^{|E(T_2)|} \xi_{T_2*}\left(\prod_{h \in \tH_+^{em}(T_2)}\hspace{-0.25cm}\psi_h^{p^{(2)}(h)}\right)\hspace{0cm} \prod_{h \in\tH_+^{em}(T_2)}\hspace{-0.25cm}x_{I_{h,T}}^{p^{(2)}(h)+1}\right],\notag
\end{align}
where  $\gl\colon\oM_{g_1,(n-|I_{f,T}|+1)+m}\times \oM_{g_2,(|I_{f,T}|-1)+2}\to \oM_{g,n+m}$ is the natural gluing map that glues the marked point corresponding to the regular leg $f$ on the curves of the first space and the marked point corresponding to the frozen leg $f'$ on the curves of the second space into a node.

Summing the expressions~\eqref{eq:decomposition of T} over all pair $(T,p)$, $T\in\SRT_{g,n,m}$, $T\colon\tH^{em}_+(T)\to\mbZ_{\ge 0}$, such that $p(\sigma_i)=0$ and $\sigma_i$ is not attached to the root of $T$, we obtain that the difference~\eqref{eq:difference} is equal to
\begin{align*}
&-\Coef_{x_1^{d_1}\cdots x_i^{d_i+1}\cdots x_n^{d_n}}\sum_{\substack{g_1+g_2=g\\g_1,g_2\ge 0}}\sum_{\substack{I\sqcup J=\br{n}\\i\in J\ne\{i\}}}\gl_*\Big[\underbrace{x_J P_{g_1,|I|+1,m}(X_I,x_J)}_{A:=}\otimes \underbrace{P_{g_2,|J|-1,2}(X_{J\backslash\{i\}})}_{B:=}\Big]\\
&-\Coef_{x_1^{d_1}\cdots x_i^{d_i+1}\cdots x_n^{d_n}}\sum_{\substack{g_1+g_2=g\\g_1\ge 0,\,g_2\ge 1}}\gl_*\Big[\underbrace{x_i P_{g_1,n,m}(x_1,\ldots,\widehat{x_i},\ldots,x_n,x_i)}_{C:=}\otimes 1\Big],
\end{align*}
where by $X_I$ we denote the tuple of numbers $x_{i_1},\ldots,x_{i_{|I|}}$, $\{i_1,\ldots,i_{|I|}\}=I$. By the assumptions of the lemma, we have
$$
\deg A\le 2g_1+m-2+1, \quad \deg B\le 2g_2+2-2,\quad \deg C\le 2g_1+m-2+1\le 2g+m-1,
$$
which implies that the difference~\eqref{eq:difference} is equal to $0$ when $\sum d_i\ge 2g+m-1$.
\end{proof}
 
Now we can complete the proof of Theorem~\ref{theorem:reductionSmall}. With Lemma~\ref{lem:reductionpsi}, we can first prove the equivalence of~\eqref{eq:systemofrelations} and~\eqref{eq:smallsystemofrelations} for $m=2$ by induction on the pairs $(g,n)$ ignoring the condition $d_i\geq 1$ in~\eqref{eq:smallsystemofrelations}. The condition $d_i\geq 1$ is then restored by Lemma~\ref{lem:string}. Once it is done for $m=2$, we can do it for any $m\geq 2$, again, first using Lemma~\ref{lem:reductionpsi} and induction on the pairs $(g,n)$ ignoring the condition $d_i\geq 1$ in~\eqref{eq:smallsystemofrelations} and then applying Lemma~\ref{lem:string} to restore this condition. 
\end{proof}

\begin{remark} 
Theorem~\ref{theorem:reductionSmall} reduces the whole system of tautological relations from Conjecture~\ref{conjecture1} to a finite number of relations of fixed degree $2g+m-1$ for each $g$ and $m$. The total number of relations to check is equal to the number of partitions of $2g+m-1$. 

In particular, the crucial case for the application to the polynomiality of the Dubrovin--Zhang hierarchies for arbitrary F-CohFTs (see Section~\ref{subsec:poly}), the case $m=2$, is reduced to $|\{\lambda\vdash (2g+1) \}|$ relations for each $g\geq 0$. One of these relations is proved for any $g\geq 0$ in Section~\ref{subsection:VanishingOnePoint}.
\end{remark}

\renewcommand\thesection{\Alph{section}}
\setcounter{section}{0}

\section*{Appendix. Localization in the moduli space of stable relative maps}\label{appendix:localization}

\addcontentsline{toc}{section}{Appendix. Localization in the moduli space of stable relative maps}
\refstepcounter{section}
\setcounter{subsection}{0}

For convenience of a reader, we briefly review here the localization formula for the moduli space of stable relative maps to $(\mbP^1,\infty)$, following the papers~\cite{GV05} (which presents the formula in a much more general setting) and \cite{Liu11} (containing details in the case of stable relative maps to $(\mbP^1,\infty)$).

\subsection{Stable relative maps}

For $m\ge 1$, we denote by $\mbP^1(m)=\mbP^1_1\cup\ldots\cup\mbP^1_m$ a chain of $m$ copies of $\mbP^1$. For $l=1,\ldots,m-1$, let $q_l$ be the node at which $P^1_l$ and $\mbP^1_{l+1}$ intersect. Let $q_0\in\mbP^1_1$ and $q_m\in\mbP^1_m$ be smooth points. Identifying $q_0$ with $\infty\in\mbP^1=\mbP^1_0$, we obtain a chain of $m+1$ copies of $\mbP^1$ that we denote by~$\mbP^1[m]$. In the case $m=0$, we define $\mbP^1[0]:=\mbP^1=\mbP^1_0$ and $q_0:=\infty$. The component $\mbP^1=\mbP^1_0$ is called the \emph{root} component, and the components $\mbP^1_1,\ldots,\mbP^1_m$ are called the \emph{bubble} components.

Given $g\ge 0$, $d\ge 1$, $n\ge 0$, and a partition $\mu$ of $d$ of length $h=l(\mu)$, a \emph{stable relative map} to $(\mbP^1,\infty)$ is the following data:
\begin{gather}\label{eq:stable relative map}
(f\colon C\lra\mbP^1[m];x_1,\ldots,x_{h+n}),
\end{gather}
where $C$ is a connected complex algebraic curve of genus $g$, with at most nodal singularities,~$f$~is a morphism, and $x_1,\ldots,x_{h+n}\in C$ are smooth pairwise distinct marked points with the following properties:  
\begin{enumerate}
\item[a)] We have the degree condition over each $\mbP^1_i$, $0\le i\le m$. 

\item[b)] We have $f^{-1}(q_m)=\{x_1,\ldots,x_h\}$, and the map $f$ is ramified at $x_i$, $1\le i\le h$, with multiplicity~$\mu_i$.

\item[c)] We have the predeformability condition over each node of $\mbP^1[m]$. 

\item[d)] The automorphism group of~\eqref{eq:stable relative map} is finite, where in the target $\mbP^1[m]$, we allow automorphisms fixing all of the points from $\mbP^1_0\cup\{q_m\}$.
\end{enumerate}
The space of isomorphism classes of stable relative maps is denoted by $\oM_{g,n}(\mbP^1,\mu)$. This space is connected, and it is endowed with a virtual fundamental class
$$
[\oM_{g,n}(\mbP^1,\mu)]^\vir\in H_{2\cdot\vdim}(\oM_{g,n}(\mbP^1,\mu),\mbC),\quad \vdim=2g-2+d+h.
$$

Given $g\ge 0$, $d\ge 1$, $n\ge 0$, and partitions $\nu,\mu$ of $d$ of lengths $k=l(\nu)$ and $h=l(\mu)$, a \emph{stable relative map to rubber $\mbP^1$} is the following data:
\begin{gather}\label{eq:stable relative map to rubber}
(f\colon C\lra\mbP^1(m);x_1,\ldots,x_{h+n+k}),
\end{gather}
where $C$ is a connected complex algebraic curves of genus $g$ with at most nodal singularities,~$f$~is a morphism, and $x_1,\ldots,x_{h+n+k}\in C$ are smooth pairwise distinct marked points with the following properties:  
\begin{enumerate}
\item[a)] We have the degree condition over each $\mbP^1_i$, $1\le i\le m$. 

\item[b)] We have $f^{-1}(q_0)=\{x_{h+n+1},\ldots,x_{h+n+k}\}$, and the map $f$ is ramified at $x_{h+n+i}$, $1\le i\le k$, with multiplicity~$\nu_i$. 

\item[c)] We have $f^{-1}(q_m)=\{x_1,\ldots,x_h\}$, and the map $f$ is ramified at $x_i$, $1\le i\le h$, with multiplicity~$\mu_i$. 

\item[d)] We have the predeformability condition over each node of $\mbP^1(m)$. 

\item[e)] The automorphism group of~\eqref{eq:stable relative map to rubber} is finite, where in the target $\mbP^1(m)$, we allow automorphisms fixing the points $q_0$ and $q_m$.
\end{enumerate}
The space of isomorphism classes of stable relative maps to rubber $\mbP^1$ is denoted by $\oM^\sim_{g,n}(\mbP^1,\nu,\mu)$. This space is connected, and it is endowed with a virtual fundamental class
$$
[\oM^\sim_{g,n}(\mbP^1,\nu,\mu)]^\vir\in H_{2\cdot\vdim}(\oM^\sim_{g,n}(\mbP^1,\nu,\mu),\mbC),\quad \vdim=2g-3+k+h.
$$
We define $\oM^\sim_{g,n}(\mbP^1,\nu,\mu):=\emptyset$ if $|\mu|\ne|\nu|$ or if $|\mu|=|\nu|=0$.

One can also consider stable relative maps to rubber~$\mbP^1$ where the source curve is not necessarily connected. The moduli space of such maps will be denoted by $\oM^{\sim,\bullet}_{g,n}(\mbP^1,\nu,\mu)$. Note that the genus $g$ can be negative here. This space is not necessarily connected. The connected components can be described as follows. For $r\ge 1$, consider a decomposition
$$
\br{h+n+k}=\bigsqcup_{i=1}^r A_i,\quad g=\sum_{i=1}^r g_i+1-r,\quad g_i\ge 0.
$$
We denote by $A$ the set of pairs
$$
A:=\{(g_1,A_1),(g_2,A_2),\ldots,(g_r,A_r)\}.
$$
Denote by 
$$
\oM_{g,n}^{\sim,\bullet}(\mbP^1,\nu,\mu)_A
$$
the subspace of $\oM_{g,n}^{\sim,\bullet}(\mbP^1,\nu,\mu)$ formed by stable relative maps to rubber $\mbP^1$, where the source curve has $r$ connected components, and for each $1\le j\le r$, all of the points from $\{x_{i}\}_{i\in A_j}$ belong to one connected component of genus $g_j$. If $\oM_{g,n}^{\sim,\bullet}(\mbP^1,\nu,\mu)_A\ne\emptyset$, then it is a connected component of $\oM^{\sim,\bullet}_{g,n}(\mbP^1,\nu,\mu)$.

Assigning to a stable relative map to rubber $\mbP^1$ the cotangent space at the point $q_0\in\mbP(m)$ gives a line bundle over $\oM_{g,n}^{\sim,\bullet}(\mbP^1,\nu,\mu)$ whose first Chern class is denoted by $\tpsi_0\in H^2(\oM_{g,n}^{\sim,\bullet}(\mbP^1,\nu,\mu))$.  

\subsection{$\boldsymbol{\mbC^*}$-fixed points}

Consider the $\mbC^*$-action on $\mbP^1$ given by
$$
t\cdot [x,y]:=[tx,y],\quad [x,y]\in\mbP^1,\; t\in\mbC^*,
$$
and the induced $\mbC^*$-action on $\oM_{g,n}(\mbP^1,\mu)$. For the localization formula, we will need a description of the connected components of the $\mbC^*$-fixed point set $\oM_{g,n}(\mbP^1,\mu)^{\mbC^*}$. 

We assume $2g-2+h+n>0$.

Consider a stable relative map $(f\colon C\to\mbP^1[m];x_1,\ldots,x_{h+n})$ from $\oM_{g,n}(\mbP^1,\mu)^{\mbC^*}$. Then we have the following: 
\begin{itemize}
\item $f^{-1}(\mbP^1_0\backslash\{0,\infty\})$ is a disjoint union of twice-punctured spheres $S_1,\ldots,S_k$, $k\ge 1$, and $f|_{S_i}\colon S_i\to\mbP^1_0\backslash\{0,\infty\}$ is an honest covering map, whose degree we denote by $d_i$.

\item $f^{-1}(0)$ is a disjoint union of connected nodal curves $C^{(0)}_1,\ldots,C^{(0)}_p$ and of some number of points.
\end{itemize}
Regarding the behavior of $f$ over $\infty$, there are two cases.

\emph{Case $1$: $m=0$}.~ Then we have $f^{-1}(\infty)=\{x_1,\ldots,x_h\}$, $p=1$, $f^{-1}(0)=C^{(0)}_1$, $k=h$, and after a renumbering of the spheres $\oS_1,\ldots,\oS_h$, we have $d_i=\mu_i$. Denote the space of such maps by~$\oF_0\subset\oM_{g,n}(\mbP^1,\mu)^{\mbC^*}$; it is connected. Note that $x_{h+1},\ldots,x_{h+n}\in C^{(0)}_1$, and therefore the curve~$C^{(0)}_1$ equipped with the marked points $C^{(0)}_1\cap\oS_i$, $1\le i\le h$, and $x_{h+1},\ldots,x_{h+n}$ is a stable curve from $\oM_{g,h+n}$. Conversely, given a stable curve from $\oM_{g,h+n}$, let us attach $h$ copies of $\mbP^1$ at the first $h$ marked points and construct a map from the resulting curve to $\mbP^1$ by sending the original stable curve to $0$ and mapping the $\supth{i}$ copy of $\mbP^1$ to the target $\mbP^1$ by $[x,y]\mapsto [x^{\mu_i},y^{\mu_i}]$. This gives a surjective map 
$$
\iota_0\colon\oM_{g,h+n}\lra\oF_0,
$$
for which we have 
$$
\iota_{0*}[\oM_{g,h+n}]=\prod_{i=1}^h\mu_i\cdot[\oF_0]^\vir.
$$

\emph{Case $2$: $m\ge 1$}.~ Then $f^{-1}(\mbP^1(m))$ is a disjoint union of connected nodal curves $C^{(\infty)}_1,\ldots,C^{(\infty)}_r$. Let us assign to our stable relative map a \emph{decorated bipartite graph} $\Gamma$ as follows: 
\begin{itemize}
\item The vertices $v\in V(\Gamma)$ are labeled by $0$ or $\infty$, which gives a decomposition $V(\Gamma)=V^{0}(\Gamma)\sqcup V^{\infty}(\Gamma)$. The vertices from $V^{0}(\Gamma)$ correspond to the connected components of~$f^{-1}(0)$, and the vertices from $V^{\infty}(\Gamma)$ correspond to the curves $C^{(\infty)}_i$, $1\le i\le r$. 

\item A vertex $v\in V^0(\Gamma)$ is called \emph{unstable} if it corresponds to a point in~$f^{-1}(0)$. All other vertices from $V^0(\Gamma)$ are called \emph{stable}. The sets of stable and unstable vertices are denoted by $V^0_\st(\Gamma)$ and $V^0_\unst(\Gamma)$, respectively; $V^0(\Gamma)=V^0_\st(\Gamma)\sqcup V^0_\unst(\Gamma)$.

\item Each vertex $v\in V(\Gamma)$ is decorated with a number $g(v)\in\mbZ_{\ge 0}$ that is equal to
\begin{itemize}
\item the genus of the corresponding curve if $v$ corresponds to a curve,
\item $0$ if $v$ corresponds to a point.
\end{itemize}

\item The edges $e\in E(\Gamma)$ correspond to the spheres $\oS_1,\ldots,\oS_k$. The edge $e\in E(\Gamma)$ corresponding to $\oS_i$ is decorated with $d_e:=d_i$. By definition, we assign the same number to both half-edges $h_1$, $h_2$ forming the edge $e$; that is, $d_{h_1}=d_{h_2}:=d_e$. 

\item The graph $\Gamma$ carries $h+n$ legs $L(\Gamma)$ that correspond to the marked points on $C$.

\item We say that an unstable vertex $v\in V^0_\unst(\Gamma)$ is of 
\begin{itemize}
\item \emph{first type} if $n(v)=1$ (the set of such vertices is denoted by $V_\unst^{0,1}(\Gamma)$),
\item \emph{second type} if $n(v)=2$ and $|L[v]|=1$ (the set of such vertices is denoted by~$V_\unst^{0,2}(\Gamma)$),
\item \emph{third type} if $n(v)=2$ and $L[v]=\emptyset$ (the set of such vertices is denoted by $V_\unst^{0,3}(\Gamma)$).
\end{itemize}
If $v$ is an unstable vertex of first type or of second type, then it is incident to exactly one edge $e\in E(\Gamma)$. Set $d_v:=d_e$. If $v$ is an unstable vertex of third type, then it is incident to exactly two edges $e,\te\in E(\Gamma)$. Set $d_v:=d_e$ and $\td_v:=d_{\te}$.

\item The graph $\Gamma$ is connected.
\end{itemize}

Denote by $\oF_\Gamma$ the subspace of $\oM_{g,n}(\mbP^1,\mu)^{\mbC^*}$ formed by stable relative maps with a given decorated bipartite graph $\Gamma$. If $\oF_\Gamma\ne\emptyset$, then it is a connected component of $\oM_{g,n}(\mbP^1,\mu)^{\mbC^*}$. Introduce the following notation:
$$
g_\infty(\Gamma):=\sum_{v\in V^\infty(\Gamma)}(g(v)-1)+1,\quad n_\infty(\Gamma):=\sum_{v\in V^\infty(\Gamma)}|L[v]|-h.
$$
Consider the set of pairs 
$$
A(\Gamma):=\{(g(v),H[v])\}_{v\in V^\infty(\Gamma)},
$$
and denote by $\nu(\Gamma)$ the partition of $d=|\mu|$ given by the numbers $d_e$, $e\in E(\Gamma)$. Set
$$
\oM_\Gamma:=\prod_{v\in V^0_\st(\Gamma)}\oM_{g(v),n(v)}\times\oM_{g_{\infty}(\Gamma),n_\infty(\Gamma)}^{\sim,\bullet}(\mbP^1,\nu(\Gamma),\mu)_{A(\Gamma)}.
$$
We have
$$
[\oM_\Gamma]^\vir=\prod_{v\in V^0_\st(\Gamma)}[\oM_{g(v),n(v)}]\times\left[\oM_{g_{\infty}(\Gamma),n_\infty(\Gamma)}^{\sim,\bullet}(\mbP^1,\nu(\Gamma),\mu)_{A(\Gamma)}\right]^\vir,
$$
and for the natural surjective morphism
$$
\iota_\Gamma\colon\oM_{\Gamma}\lra\oF_\Gamma, 
$$
we have
\begin{gather}\label{eq:covering of connected components}
\iota_{\Gamma*}[\oM_{\Gamma}]^\vir=|\Aut(\Gamma)|\prod_{e\in E(\Gamma)}d_e\cdot[\oF_\Gamma]^\vir.
\end{gather}

\subsection{Localization formula}

For a topological space $X$ with a $\mbC^*$-action, we denote by $H_*^{\mbC^*}(X)$ and $H^*_{\mbC^*}(X)$ the $\mbC^*$-equivariant homology and cohomology groups of~$X$ with  coefficients in~$\mbC$, respectively. We denote by~$u$ the generator of the equivariant cohomology ring of a point: $H^*_{\mbC^*}(\pt)=\mbC[u]$. The equivariant Euler class of a $\mbC^*$-equivariant complex vector bundle $V\to X$ is denoted by $e_{\mbC^*}(V)$. 

We now consider the $\mbC^*$-action on $\oM_{g,n}(\mbP^1,\mu)$ given in the previous section. The moduli space $\oM_{g,n}(\mbP^1,\mu)$ is endowed with a $\mbC^*$-equivariant virtual fundamental class, which abusing notation we denote by $[\oM_{g,n}(\mbP^1,\mu)]^\vir\in H^{\mbC^*}_{2\cdot\vdim}(\oM_{g,n}(\mbP^1,\mu))$. The virtual localization formula for the moduli space $\oM_{g,n}(\mbP^1,\mu)$ proved in~\cite{GV05} gives the following formula for this class considered as an element of $H^{\mbC^*}_*(\oM_{g,n}(\mbP^1,\mu))\otimes_{\mbC[u]}\mbC[u,u^{-1}]$: 
\begin{gather}\label{eq:general localization formula}
[\oM_{g,n}(\mbP^1,\mu)]^\vir=\frac{[\oF_0]^\vir}{e_{\mbC^*}(N_0^\vir)}+\sum_{\substack{\text{decorated}\\\text{graphs $\Gamma$}}}\frac{[\oF_\Gamma]^\vir}{e_{\mbC^*}(N_\Gamma^\vir)}\in H^{\mbC^*}_*(\oM_{g,n}(\mbP^1,\mu))\otimes_{\mbC[u]}\mbC[u,u^{-1}],
\end{gather}
where $N_0^\vir$ and $N_\Gamma^\vir$ are the virtual normal bundles to the subspaces $\oF_0$ and $\oF_\Gamma$, respectively. The following formulas are very useful for applications of the localization formula:
\begin{align}
&\frac{1}{\iota_0^*e_{\mbC^*}(N_0^\vir)}=\Lambda_g^{\vee}(u)u^{-1}\prod_{i=1}^h\left(\frac{\mu_i^{\mu_i+1}}{\mu_i!}\frac{u^{1-\mu_i}}{u-\mu_i\psi_i}\right),\label{eq:main localization formula 1}\\
\frac{1}{\iota_\Gamma^*e_{\mbC^*}(N_\Gamma^\vir)}=&\;\frac{1}{-u-\tpsi_0}\prod_{v\in V^0_\st(\Gamma)}\left(\Lambda_{g(v)}^{\vee}(u)u^{-1}\prod_{h\in H[v]\backslash L[v]}\frac{d_h^{d_h+1}}{d_h!}\frac{u^{1-d_h}}{u-d_h\psi_h}\right) \label{eq:main localization formula 2}\\
&\hphantom{\frac{1}{-u}}\times\prod_{v\in V^{0,1}_{\unst}(\Gamma)}\left(\frac{d_v^{d_v-1}}{d_v!}u^{1-d_v}\right)\prod_{v\in V^{0,2}_{\unst}(\Gamma)}\left(\frac{d_v^{d_v}}{d_v!}u^{-d_v}\right)\prod_{v\in V^{0,3}_{\unst}(\Gamma)}\left(\frac{d_v^{d_v+1}\td_v^{\td_v+1}}{d_v!\td_v!}\frac{u^{-d_v-\td_v}}{d_v+\td_v}\right),\notag
\end{align}
where 
$$
\Lambda_g^\vee(u):=\sum_{i=0}^g(-1)^i\lambda_i u^{g-i}.
$$

%%%%%%%%%%%%%%%%%%%%%
% References
%%%%%%%%%%%%%%%%%%%%%

\end{document}